%% file: On_Supercompactly_and_Compactly_Generated_Toposes.tex
\documentclass[11pt]{article}
\usepackage[utf8]{inputenc}

\usepackage{geometry} % to change the page dimensions
\usepackage{graphicx} % support the \includegraphics command and options

\usepackage{booktabs} % for much better looking tables
\usepackage{array} % for better arrays (eg matrices) in maths
\usepackage{paralist} % very flexible & customisable lists (eg. enumerate/itemize, etc.)
\usepackage{verbatim} % adds environment for commenting out blocks of text & for better verbatim
\usepackage{subfig} % make it possible to include more than one captioned figure/table in a single float
\usepackage{hyperref}
\usepackage{enumerate}
\usepackage{amsmath}
\usepackage{amssymb}
\usepackage{authblk}
\usepackage{stmaryrd}
\usepackage{bbm}
\usepackage{mathtools}
\usepackage{textcomp}
\usepackage{xcolor}
\usepackage{listings}
\usepackage{float}
\usepackage{centernot}
\usepackage[mathscr]{euscript}
\usepackage{tikz-cd}

\newcommand{\Ccal}{\mathcal{C}}
\newcommand{\Dcal}{\mathcal{D}}
\newcommand{\Ecal}{\mathcal{E}}
\newcommand{\Fcal}{\mathcal{F}}
\newcommand{\Gcal}{\mathcal{G}}
\newcommand{\Ical}{\mathcal{I}}

\newcommand{\Ocal}{\mathcal{O}}

\newcommand{\Scal}{\mathcal{S}}
\newcommand{\Tcal}{\mathcal{T}}

\newcommand{\Ibb}{\mathbb{I}}

\newcommand{\Nbb}{\mathbb{N}}

\newcommand{\Rbb}{\mathbb{R}}

\newcommand{\Zbb}{\mathbb{Z}}
\newcommand{\yon}{\mathbf{y}}
\newcommand{\Set}{\mathbf{Set}}
\newcommand{\FinSet}{\mathbf{FinSet}}
\newcommand{\Pos}{\mathbf{Pos}}

\newcommand{\Arch}{\mathbf{Arch}}
\newcommand{\mArch}{\mathbf{mArch}}
\newcommand{\Hom}{\mathrm{Hom}}

\newcommand{\Cat}{\mathbf{Cat}}

\newcommand{\TOP}{\mathfrak{TOP}}
\newcommand{\op}{^{\mathrm{op}}}
\newcommand{\co}{^{\mathrm{co}}}

\newcommand{\too}{\twoheadrightarrow}

\mathchardef\mhyphen="2D

\DeclareMathOperator{\im}{im}

\DeclareMathOperator{\id}{id}

\DeclareMathOperator{\Sh}{Sh}

\DeclareMathOperator{\Span}{Span}
\DeclareMathOperator{\mSpan}{mSpan}

\DeclareMathOperator{\colim}{colim}
\tikzset{
  no line/.style={draw=none,
    commutative diagrams/every label/.append style={/tikz/auto=false}},
  from/.style args={#1 to #2}{to path={(#1)--(#2)\tikztonodes}}
	}
\tikzset{symbol/.style={draw=none, every to/.append style={edge node = {node [sloped, allow upside down, auto=false] {$#1$}}}}}

\usepackage{amsthm}
\newtheorem{thm}{Theorem}[section]
\newtheorem{proposition}[thm]{Proposition}
\newtheorem{prop}[thm]{Proposition}
\newtheorem{lemma}[thm]{Lemma}
\newtheorem{fact}[thm]{Fact}
\newtheorem{crly}[thm]{Corollary}
\newtheorem{corollary}[thm]{Corollary}
\newtheorem{scholium}[thm]{Scholium}
\newtheorem{schl}[thm]{Scholium}

\theoremstyle{definition}
\newtheorem{definition}[thm]{Definition}
\newtheorem{dfn}[thm]{Definition}
\newtheorem{xmpl}[thm]{Example}
\theoremstyle{remark}
\newtheorem{remark}[thm]{Remark}
\newtheorem{rmk}[thm]{Remark}

\usepackage{fancyhdr} % This should be set AFTER setting up the page geometry
\usepackage[nottoc,notlof,notlot]{tocbibind} % Put the bibliography in the ToC
\usepackage[titles,subfigure]{tocloft} % Alter the style of the Table of Contents

\usepackage[nodayofweek]{datetime}
\longdate

\date{\today} 

\title{On Supercompactly and Compactly Generated Toposes}
\author[1]{Morgan Rogers \thanks{Universit\`a degli Studi dell{'}Insubria, Via Valleggio n. 11, 22100 Como CO} \thanks{Marie Sklodowska-Curie fellow of the Istituto Nazionale di Alta Matematica} \thanks{email: mrogers@uninsubria.it}}

\begin{document}

\maketitle{}

\abstract{We present and characterize the classes of Grothendieck toposes having enough supercompact objects or enough compact objects. In the process, we examine the subcategories of supercompact objects and compact objects within such toposes and classes of geometric morphism which interact well with these objects. We also present canonical classes of sites generating such toposes.}

\tableofcontents

\input{TSGT_Introduction}

\input{TSGT_Similar_Properties}

\input{TSGT_Principal_Sites}

\input{TSGT_Examples}

\bibliographystyle{plain}
\bibliography{classificationbib}

\end{document}

%% file: TSGT_Introduction.tex
\section*{Introduction}

In \cite{TTMA}, the author of the current paper obtained a characterisation of toposes of right actions of discrete monoids on sets, which are a special case of presheaf toposes. A natural next step in this direction is to consider sets as discrete spaces and to examine the categories of actions of topological monoids on them. It turns out that the resulting categories, to which a future paper will be devoted, are also toposes, and that moreover they fall into the class of \textit{supercompactly generated toposes} which we present and thoroughly investigate in this paper.

By a supercompactly generated topos, we mean a topos with a separating set of supercompact objects (see Definition \ref{dfn:scompact}). Apart from the toposes of topological monoid actions mentioned above, this class includes all regular toposes (Example \ref{xmpl:regular}), as well as all presheaf toposes and some other important classes described in Proposition \ref{prop:xmpls}, and so is of general interest in topos theory. Since it is convenient to do so, we also study \textit{compactly generated toposes}, which are conceptually similar enough that we can prove analogous results about them with little extra work. For brevity, some of the section headings refer only to the supercompact naming conventions.

This is not the first time these concepts have been studied. Some relevant results appear in Section 4.1 of the thesis of Bridge \cite{TAC}, where our supercompactly generated toposes are referred to as `\textit{locally supercompact}'; the reason for the terminology choice of the present author is that the adverb ``locally'' is already overloaded in topos-theory literature. Moreover, a special case of principal sites, `\textit{$B$-sites}', are the eventual focus of the paper \cite{SPM} of Kondo and Yasuda. We posit that where there are overlaps in the basic results of the present paper and those above (which we shall point out in situ), the main distinction of the present paper is the emphasis on topos-theoretic machinery: the present author avoids reasoning directly with sheaves as far as possible (with the exception of the non-constructive results of Section \ref{ssec:Deligne}), which allows for more concise categorical proofs.

We assume the reader is familiar with the basics of category theory and Grothendieck toposes; the content of Mac Lane and Moerdijk's textbook \cite{MLM} should be adequate. For Section \ref{sec:principal} we rely heavily on the recent monograph of Caramello \cite{Dense} which contains general results about site representations of toposes and morphisms of sites; we quote some of those results without proof here.

This work was supported by INdAM and the Marie Sklodowska-Curie Actions as a part of the \textit{INdAM Doctoral Programme in Mathematics and/or Applications Cofunded by Marie Sklodowska-Curie Actions}. The author would like to thank Olivia Caramello for her patience and helpful suggestions.

\subsection*{Overview}

The structure of this paper is as follows. In Section \ref{ssec:super}, we recall the definitions of supercompact and compact objects in a topos. This leads us to formally define supercompactly and compactly generated toposes in Section \ref{ssec:scompgen}.

In Section \ref{ssec:Cs}, we turn to the full subcategories of supercompact and compact objects in a topos, presenting the structure they inherit from their ambient toposes, with a focus on monomorphisms, epimorphisms, and the classes of \textit{funneling} and \textit{multifunneling} colimits, which we introduce in Definitions \ref{dfn:funnel} and \ref{dfn:multifun}. The purpose of this investigation is to present these subcategories as canonical sites for supercompactly and compactly generated toposes in Section \ref{ssec:cansite}, which we do in Theorem \ref{thm:canon}. Motivated by the special case of toposes of monoid actions, in Section \ref{ssec:twoval} we investigate some extra conditions on a supercompactly generated topos which guarantee further properties of its category of supercompact objects.

In Section \ref{ssec:proper}, we examine classes of geometric morphism whose inverse image functors preserve supercompact or compact objects, introducing notions of \textit{precise} and \textit{polite} geometric morphisms in analogy with proper geometric morphisms (Definition \ref{dfn:polite}), before focussing on relative versions of these properties (Definition \ref{dfn:relpolite}) which are more directly useful in our analysis. After establishing these definitions, we examine how more familiar classes of geometric morphism interact with supercompactly and compactly generated toposes: surjections and inclusions in Section \ref{ssec:surj}, then hyperconnected morphisms in Section \ref{ssec:hype}. This exploration gives us several tools for constructing such toposes, which are summarised in Theorem \ref{thm:closure}.

The focus of Section \ref{sec:principal} is a broader site-theoretic investigation. In Section \ref{ssec:stable}, we exhibit the categorical data of \textit{principal} and \textit{finitely generated} sites, which are natural classes of sites whose categories of sheaves are supercompactly and compactly generated toposes, respectively. In Section \ref{ssec:representable}, we examine the morphisms between the representable sheaves on these sites, and then show in Section \ref{ssec:quotient} how a general such site may be reduced via a canonical congruence without changing the resulting topos. We use what we have learned about these sites in Section \ref{ssec:redcat} to characterize the categories of supercompact and compact objects which were the subject of Section \ref{ssec:cansite} as \textit{reductive} and \textit{coalescent} categories, respectively (Definition \ref{dfn:reductive}), satisfying additional technical conditions. We make the correspondence between such categories and the toposes they generate explicit in Theorem \ref{thm:correspondence}. It is natural to compare these classes of categories to the well-known classes of (locally) regular and coherent categories, which we do in Section \ref{ssec:regcoh}.

Moving onto morphisms, we recall the definition of \textit{morphisms of sites} in Section \ref{ssec:morsites}, showing that, according to the class of sites under consideration, these induce the relatively precise, polite or proper geometric morphisms introduced in Section \ref{ssec:proper}. More significantly, restricting to canonical sites, we are able to extend the correspondences of Theorem \ref{thm:correspondence} to some $2$-equivalences between $2$-categories of sites and $2$-categories of toposes. In Section \ref{ssec:Deligne}, we examine the points of supercompactly and compactly generated toposes, extending the classical result of Deligne to show that any compactly generated topos has enough points. To ground the discussion, we present some examples and counterexamples of reductive and coalescent sites and their properties in Section \ref{ssec:xmpl}.

Finally, in Section \ref{sec:xmpls} we examine the special case of supercompactly and compactly generated localic toposes, recovering some Stone-type dualities and counterexamples in the process.

%% file: TSGT_Similar_Properties.tex
\section{Supercompact and Compact Objects}
\label{sec:Supercompact}

Throughout, when $(\Ccal,J)$ is a site, we write $\ell:\Ccal \to \Sh(\Ccal,J)$ for the composite of the Yoneda embedding and the sheafification functor, assuming the Grothendieck topology is clear in context, and call the images $\ell(C)$ of the objects $C \in \Ccal$ the \textbf{representable sheaves}.

\subsection{Supercompact Objects}
\label{ssec:super}

The following definitions can be found in \cite[Definition 2.1.14]{TST}:

\begin{dfn}
\label{dfn:scompact}
An object $C$ of a category $\Ecal$ is \textbf{supercompact} (resp. \textbf{compact}) if any jointly epic family of morphisms $\{A_i \to C \mid i \in I\}$ contains an epimorphism (resp. a finite jointly epic sub-family). 
\end{dfn}

Clearly every supercompact object is compact. Compact objects are more widely studied, notably in the lifting of the concept of compactness from topological spaces to toposes reviewed by Moerdijk and Vermeulen in \cite{Compact}. Since the two classes of objects behave very similarly, we treat them in parallel. 

In a topos, we may re-express the definitions of supercompact and compact objects in terms of their subobjects. As is standard, we can further convert any statement about subobjects of an object $X$ in a topos $\Ecal$ into a statement about the subterminal object in the slice topos $\Ecal/X$.

\begin{lemma}
\label{lem:scompact}
An object $C$ of a Grothendieck topos $\Ecal$ is supercompact (resp. compact) if and only if every covering of $C$ by a family (resp. a directed family) of subobjects $A_i \hookrightarrow C$ contains an isomorphism. This occurs if and only if the global sections functor $\Gamma: \Ecal/C \to \Set$ preserves arbitrary (resp. directed) unions of subobjects.
\end{lemma}
\begin{proof}
For the first part, suppose that we are given a family (resp. a directed family) of subobjects covering a supercompact (resp. compact) object $C$. Then one of the monomorphisms involved must be epic and hence an isomorphism (resp. this collection contains a finite covering family, but the union of these subobjects is also a member of the family and must be covering). In the opposite direction it suffices to consider images of the morphisms in an arbitrary covering family (resp. finite unions of these images). 

The remainder of the proof is modelled after that of Johnstone in \cite[C1.5.5]{Ele}. Given an object $f: A \to C$ of $\Ecal/C$ which is a union (resp. a directed union) of subobjects $A_i \hookrightarrow A \to C$ and given a global section $x: C \to A$ of $f$, we may consider the pullbacks:
\[\begin{tikzcd}
C_i \ar[r, hook] \ar[d] \ar[dr, phantom, "\lrcorner", very near start] &
C \ar[d, "x"]\\
A_i \ar[r, hook] & A.
\end{tikzcd}\]
By extensivity, $C$ is the union of the $C_i$, and by the above one of the $C_i \hookrightarrow C$ must be an isomorphism, so that $x$ factors through one of the $A_i$, which gives the result.

Conversely, given a jointly epic family (resp. directed family) of subobjects $C_i \hookrightarrow C$ considered as subterminals in $\Ecal/C$, we may apply $\Gamma$ to see that one of them must be an isomorphism, as required.
\end{proof}

% \begin{rmk}
% The definition of compact object implicitly includes the initial object, but the definition of supercompact object does not. In what follows it will sometimes be convenient to consider the collection of `supercompact or initial' objects.
% \end{rmk}

\subsection{Supercompactly Generated Toposes}
\label{ssec:scompgen}

The collections of supercompact and compact objects in \textit{any} Grothendieck topos are conveniently tractable:

\begin{lemma}
\label{lem:scsite}
Let $\Ecal \simeq \Sh(\Ccal,J)$ be a Grothendieck topos of sheaves on a small site $(\Ccal,J)$. Then the supercompact objects are quotients of the representable sheaves $\ell(C)$ for $C \in \Ccal$. In particular, they are indexed (up to isomorphism) by a set. Similarly, the compact objects are quotients of finite coproducts of the images $\ell(C)$, and so (up to isomorphism) also form a set.
\end{lemma}
\begin{proof}
Given a supercompact object $Q$, since the objects $\ell(C)$ are separating in $\Ecal$, the collection of morphisms $\ell(C) \to Q$ (is inhabited and) jointly epimorphic. It follows that one such must be epimorphic.

Given a compact object $Q$, the above argument instead yields a (possibly empty) finite jointly epimorphic family of morphisms $\ell(C_i)\to Q$, which corresponds to an epimorphism $\coprod_{i\in I} \ell(C_i) \too Q$, as claimed. 
\end{proof}

Lemma \ref{lem:scsite} ensures that we can always consider the full subcategories on the supercompact (resp. compact) objects, equipped with the canonical topology induced by the topos (which shall be recalled in Definition \ref{dfn:effective} below) as an essentially small site $(\Ccal_s,J_{can}^{\Ecal}|_{\Ccal_s})$ (resp. $(\Ccal_c,J_{can}^{\Ecal}|_{\Ccal_c})$). By the Comparison Lemma, the induced canonical comparison morphism $\Ecal \to \Sh(\Ccal_s,J_{can}^{\Ecal}|_{\Ccal_s})$ is an equivalence if and only if the collection of supercompact objects is separating, and similarly for the compact case. We shall continue to use $\Ccal_s$ and $\Ccal_c$ to denote these categories in the remainder; for simplicity, we shall actually assume that we have chosen a representative set of the supercompact or compact objects, such as the quotients of representables in Lemma \ref{lem:scsite}, so that we are working with small sites.

\begin{dfn}
We say a topos is \textbf{supercompactly generated} (resp. \textbf{compactly generated}) if its collection of supercompact (resp. compact) objects is separating.
\end{dfn}

\begin{xmpl}
\label{xmpl:regular}
The syntactic category of a regular theory can be recovered from its classifying topos as the full category of regular objects. In general, we say that an object $X$ in a topos is \textbf{regular} if $X$ is supercompact and for any cospan
\[\begin{tikzcd}
Y \ar[r, "f"] & X & Z \ar[l, "g"']
\end{tikzcd}\]
with $Y$ and $Z$ supercompact, the pullback $Y \times_X Z$ is also supercompact. In particular, classifying toposes of regular theories are special cases of supercompactly generated toposes. The same can be said when `supercompact' is replaced by `compact' and `regular' is replaced by `coherent'. See the work of Caramello \cite{SCCT} for this result and a more detailed discussion (note that Caramello refers to regular objects as \textit{supercoherent} objects). We shall return to examination of categories of regular and coherent objects in Section \ref{ssec:regcoh}.
\end{xmpl}

Supercompactly generated toposes include several other important established classes of Grothendieck topos.

\begin{prop}
\label{prop:xmpls}
\begin{enumerate}[(i)]
	\item Every atomic topos is supercompactly generated.
	\item Every supercompactly generated topos is compactly generated and locally connected.
	\item Every presheaf topos is supercompactly generated.
\end{enumerate}
\end{prop}
\begin{proof}
For (i), recall that a Grothendieck topos is atomic if and only if it has a separating set of atoms, and these are easily seen to be supercompact. For (ii), we can similarly observe that any supercompact object is compact and indecomposable, then recall (by Theorem 2.7 of \cite{SCCT}, say) that a topos is locally connected if and only if it has a separating set of indecomposable objects.

For (iii), note that the representable presheaves are irreducible (every jointly epic family over a representable contains a split epimorphism) so they are in particular supercompact.
\end{proof}

\subsection{Categories of Supercompact Objects}
\label{ssec:Cs}

We now examine properties of the categories $\Ccal_s$ and $\Ccal_c$ in a general topos $\Ecal$.

\begin{lemma}
\label{lem:closed}
Let $\Ecal$ be a Grothendieck topos and let $\Ccal_s$, $\Ccal_c$ be the categories of supercompact and compact objects of $\Ecal$ respectively. These categories are closed in $\Ecal$ under quotients.
\end{lemma}
\begin{proof}
Given an epimorphism $k: D \too C$ with $D$ supercompact and a covering family over $C$, pulling back this family along $k$ we immediately conclude that one of the constituent morphisms must be an epimorphism. Thus $C$ is a member of $\Ccal_s$. The argument for $\Ccal_c$ is analogous, except that we end up with a finite family of morphisms.
\end{proof}

Lemma \ref{lem:closed} has as a consequence that when considering a covering family of supercompact or compact objects over an object $X$ of $\Ecal$, we may without loss of generality assume that the morphisms in the family are monomorphisms. That is, we may restrict attention to covering families of (super)compact \textit{subobjects} when we so choose, because a family of morphisms with common codomain in a topos is jointly epic if and only if the union of their images is the maximal subobject. 

The subcategories inherit some further structure from $\Ecal$.

\begin{crly}
\label{crly:images}
For $\Ecal$, $\Ccal_s$, $\Ccal_c$ as in Lemma \ref{lem:closed}, $\Ccal_s$ and $\Ccal_c$ are closed under image factorizations in $\Ecal$, so that in particular they have image factorizations.
\end{crly}
\begin{proof}
Given a morphism $C \to C'$ between supercompact (resp. compact) objects, the image object $C''$ in the factorization $C \too C'' \hookrightarrow C'$ is also supercompact (resp. compact) by Lemma \ref{lem:closed}, whence the factoring morphisms lie in $\Ccal_s$ (resp. $\Ccal_c$) since it is a full subcategory.
\end{proof}

Note that the resulting orthogonal factorization systems on $\Ccal_s$ and $\Ccal_c$ are not between all monomorphisms and all epimorphisms; only between those inherited from $\Ecal$. We spend the rest of this section deriving an intrinsic characterisation of these morphisms.

\begin{dfn}
\label{dfn:funnel}
We say a small indexing category $\Dcal$ is a \textbf{funnel} if it has a weakly terminal object. A \textbf{funneling diagram} in an arbitrary category $\Ccal$ is a functor $F: \Dcal \to \Ccal$ with $\Dcal$ a funnel. %We shall denote the image of the weakly terminal object of $\Dcal$ under $F$ by $D$; 
For example:
\[\begin{tikzcd}[row sep = small]
A_i \ar[dr, "f_i", shift left] \ar[dr, "f'_i"', shift right] & \\
\vdots & D. \\
A_j \ar[ur, "f_j", shift left] \ar[ur, "f'_j"', shift right] &
\end{tikzcd}\]
The colimit of $F$, if it exists, is an object $C$ of $\Ccal$ equipped with an epimorphism $f: D \too C$ through which all legs of the colimit cone factor.
\end{dfn}

Recall that a morphism $h:D \to C$ in a category $\Ccal$ is called a \textbf{strict epimorphism} if whenever another morphism $k: D \to E$ satisfies the condition that for each parallel pair $p,q: B \rightrightarrows D$ with $h \circ p = h \circ q$ we have $k \circ p = k \circ q$, it follows that $k$ factors uniquely through $h$. The dual concept appears in \cite[Theorem 4.1]{TGT}.

\begin{lemma}
In a small category $\Ccal$, a morphism $h:C' \to C$ is a strict epimorphism if and only if there exists a funneling diagram $F:\Dcal \to \Ccal$ with weakly terminal object $C'$ whose colimit is expressed by $h$.
\end{lemma}
\begin{proof}
By definition, if $h$ is a strict epimorphism, it is a colimit for the diagram consisting of all pairs of morphisms with domain $C'$ which $h$ coequalizes. Conversely, if $h$ expresses the colimit of any funneling diagram, and $k$ coequalizes all of the same parallel pairs that $h$ does, then it clearly induces a cone by composition with the morphisms of the funneling diagram, whence it has a universal factorisation through $h$, as required.
\end{proof}

Notably, strict epimorphisms include isomorphisms and regular epimorphisms. Continuing with the parallel treatment of compactly generated toposes, we arrive at the following definitions.

\begin{dfn}
\label{dfn:multifun}
A small indexing category $\Dcal$ is a \textbf{multifunnel} if it has a (possibly empty) finite collection of objects $D_1, \dotsc, D_n$ to which all other objects admit morphisms\footnote{A reader interested in the obvious generalisations of this concept to higher cardinalities might prefer to employ a name such as `finitely funneled' to emphasise the finitary aspect.}. A colimit of a multifunneling diagram (a diagram indexed by a multifunnel) in a category $\Ccal$ shall be called a \textbf{multifunneling colimit}, and is defined by a finite jointly epic family from the images of the objects $D_1, \dotsc, D_n$. A finite jointly epic family obtained in this way will be called a \textbf{strictly epic finite family}.
\end{dfn}

\begin{lemma}
\label{lem:multi}
A category has multifunneling colimits if and only if it has finite coproducts and funneling colimits.
\end{lemma}
\begin{proof}
Clearly finite coproducts and funneling colimits are special cases of multifunneling colimits. Conversely, given a multifunneling diagram $F: \Dcal \to \Ccal$ with weakly terminal objects $F(D_1),\dotsc,F(D_n)$, consider the coproduct of these objects. Composing the morphisms in the diagram $F$ with the coproduct inclusions, we get a funneling diagram. The universal property of the coproduct ensures that the colimit of this diagram coincides with the colimit of $F$.
\end{proof}

\begin{rmk}
\label{rmk:simple}
Note that we can make the further simplification, implicit in the diagram of Definition \ref{dfn:funnel}, that all of the non-identity morphisms in a funneling diagram have the weakly terminal object as their codomains, since given $t:A_i \to A_j$, there exists some morphism $f_j:A_j \to D$, and the cocone commutativity conditions for $f_j \circ t$ and $f_j$ ensure that $\lambda_i = \lambda_D \circ (f_j \circ t) = \lambda_j \circ t$ is automatically satisfied, so we may omit $t$ from the diagram. 
\end{rmk}

\begin{lemma}
\label{lem:closed2}
Let $\Ecal$ be a Grothendieck topos and let $\Ccal_s$, $\Ccal_c$ be the categories of supercompact and compact objects of $\Ecal$ respectively. Then $\Ccal_s$ is closed in $\Ecal$ under funneling colimits and $\Ccal_c$ is closed in $\Ecal$ under multifunneling colimits.
\end{lemma}
\begin{proof}
Let $F:\Dcal \to \Ccal_s$ be a funneling diagram with weakly terminal object $F(D)$. In $\Ecal$, this diagram has a colimit determined by an epimorphism $f: F(D) \too C$; the colimit $C$ is supercompact by Lemma \ref{lem:closed}, as required. The argument for multifunneling colimits in $\Ccal_c$ is analogous, except that we must pull back along each member of a finite family in the proof of Lemma \ref{lem:closed} to obtain the finite covering subfamily of a given covering family over the colimit.
\end{proof}

\begin{crly}
\label{crly:strict}
Let $\Ecal$ be a supercompactly (resp. compactly) generated topos and $\Ccal_s$, $\Ccal_c$ the usual subcategories. Then a morphism of $\Ccal_s$ is an epimorphism in $\Ecal$ if and only if it is a strict epimorphism in $\Ccal_s$, and a finite family of morphisms into $C$ in $\Ccal_c$ is jointly epic in $\Ecal$ if and only if it is a strictly epic finite family in $\Ccal_c$.
\end{crly}
\begin{proof}
This could be deduced by checking the conditions of the dual of \cite[Proposition 4.9]{TGT}, but rather than reproducing that result, we give a direct proof.

Suppose $h: C' \too C$ is epic in $\Ecal$ with $C',C$ in $\Ccal_s$. Any epimorphism in $\Ecal$ is regular, so is the coequalizer of some pair $p,q:D \rightrightarrows C'$. Since $D$ is covered by supercompact objects, composing $p$ and $q$ with the monomorphisms $C_i \hookrightarrow D$ such that $C_i$ is in $\Ccal_s$ we obtain a funneling diagram in $\Ccal_s$ whose colimit is still $C$, as required. The argument for $\Ccal_c$ is analogous.

Conversely, the inclusion of $\Ccal_s$ and $\Ccal_c$ into $\Ecal$ preserves funneling (resp. multifunneling) colimits by Lemma \ref{lem:closed2}, whence the strict epimorphisms (resp. strictly epic finite families) from these categories are still epic in $\Ecal$.
\end{proof}

In fact, epimorphic families in $\Ccal_c$ are better behaved than those in $\Ccal_s$ in general:
\begin{lemma}
\label{lem:compactepi}
Let $\Ecal$ be any Grothendieck topos. Then a family of morphisms in $\Ccal_c$ with common codomain is jointly epic in $\Ccal_c$ if and only if it is so in $\Ecal$. In particular, when $\Ecal$ is compactly generated, every jointly epimorphic family (including every epimorphism) in $\Ccal_c$ is strict.
\end{lemma}
\begin{proof}
Observe that a family of morphisms $f_i:C_i \to C$ in a category is jointly epimorphic if and only if the diagram:
\[\begin{tikzcd}
C_i \ar[dr, phantom, "\ddots"] \ar[drr, "f_i", bend left] \ar[ddr, "f_i"', bend right] & & \\
& C_j \ar[r, "f_j"] \ar[d, "f_j"'] & C \ar[d, equal] \\
& C \ar[r, equal] & C
\end{tikzcd}\]
is a colimit diagram. But the diagram (after removing the copy of $C$ in the lower right corner) is clearly an instance of a multifunneling colimit, so by Lemma \ref{lem:closed2} its colimit is created by the inclusion of $\Ccal_c$ into $\Ecal$, so a family is jointly epic in $\Ccal_c$ if and only if it is so in $\Ecal$, as required.
\end{proof}

Having extensively discussed the epimorphisms, we should also discuss monomorphisms in the subcategories under investigation.

\begin{lemma}
\label{lem:monocoincide}
Monomorphisms in $\Ccal_s$ and $\Ccal_c$ coincide with monomorphisms in $\Ecal$ when $\Ecal$ is supercompactly or compactly generated, respectively.
\end{lemma}
\begin{proof}
Certainly a monomorphism of $\Ecal$ lying in $\Ccal_s$ or $\Ccal_c$ is still monic, since there are fewer morphisms which it needs to distinguish in general.

Suppose $\Ecal$ is supercompactly generated and let $s : A \hookrightarrow B$ be a monomorphism in $\Ccal_s$, $Q$ an object of $\Ecal$ and $f, g : Q \rightrightarrows A$ such that $sf = sg$. Covering $Q$ with supercompact subobjects $q_i: Q_i \hookrightarrow Q$, consider $f q_i, g q_i: Q_i \rightrightarrows A$, which are morphisms in $\Ccal_s$. These are equalized by $s$ and hence are equal for every $i$. The $q_i$ being jointly epic then forces $f = g$. Thus $s$ is monic in $\Ecal$, as claimed. The argument for $\Ccal_c$ is analogous, replacing supercompact subobjects with compact ones.
\end{proof}

Once again, we can immediately strengthen this result for $\Ccal_c$.

\begin{lemma}
\label{lem:compactmono}
For any Grothendieck topos $\Ecal$, the monomorphisms of $\Ecal$ lying in $\Ccal_c$ are regular monomorphisms there. In particular, when $\Ecal$ is compactly generated, every monomorphism in $\Ccal_c$ is regular.
\end{lemma}
\begin{proof}
If $e:C' \hookrightarrow C$ is a morphism in $\Ccal_s$ which is monic in $\Ecal$, consider its cokernel pair in $\Ecal$:
\[\begin{tikzcd}
C' \ar[r, "e", hook] \ar[d, "e"', hook] \ar[dr, phantom, "\lrcorner", very near start] & C \ar[d, hook, "s"] \\
C \ar[r, hook, "t"'] & D. \ar[ul, phantom, "\ulcorner", very near start]
\end{tikzcd}\]
Since a topos is an adhesive category, this is also a pullback square and so $e$ is the equalizer of $s$ and $t$. Since pushouts are multifunnel colimits, $D$ lies in $\Ccal_s$, so the same is true there.
\end{proof}

Thus we can make Corollary \ref{crly:images} more precise.
\begin{crly}
\label{crly:orthog}
If $\Ecal$ is supercompactly generated, then $\Ccal_s$ has an orthogonal (strict epi,mono)-factorisation system. More generally, if $\Ecal$ is merely compactly generated, $\Ccal_c$ has an orthogonal (epi,mono)-factorisation system.
\end{crly}

For later reference, we observe that even though the categories $\Ccal_s$ and $\Ccal_c$ need not have finite products (see Example \ref{xmpl:nonreg}), we can still extend Corollaries \ref{crly:images} and \ref{crly:orthog} with factorizations of spans through jointly monic spans. Corollary \ref{crly:orthog} is the case $I = 1$ of the following Lemma.

\begin{lemma}
\label{lem:jmonic}
Let $\{f_j: B \to A_j \mid j \in I\}$ be a collection of morphisms with common domain in $\Ccal_s$ or $\Ccal_c$. Then there exists a strict epimorphism $e: B \too R$ and morphisms $\{r_j: R \to A_j \mid j \in I\}$ which are jointly monic, such that $f_j = r_j \circ e$.
\end{lemma}
\begin{proof}
Let $e$ be the strict epimorphism obtained from the funneling colimit of the collection of all parallel pairs of morphism which are coequalized by all of the $f_j$. By definition, all of the $f_j$ factorize through it, and by construction the factors form a jointly monic family.
\end{proof}

\subsection{Canonical Sites of Supercompact Objects}
\label{ssec:cansite}

We have now done enough work to usefully apply the proof of Giraud's theorem and obtain a canonical site of definition for a supercompactly or compactly generated Grothendieck topos.

\begin{dfn}
\label{dfn:effective}
Recall that a sieve $S$ on an object $C$ of a category $\Ccal$ is \textbf{effective-epimorphic} if, when $S$ is viewed as a full subcategory of $\Ccal/C$, $C$ is the colimit of the (possibly large) diagram $D_S: S \hookrightarrow \Ccal/C \to \Ccal$ obtained by composing with the forgetful functor. A sieve generated by a single morphism $f$ is effective-epimorphic if and only if the morphism is a strict epimorphism. Such a sieve $S$ is \textbf{universally} effective-epimorphic if its pullback along the functor $\Ccal/D \to \Ccal/C$ induced by a morphism $f: D \to C$ is effective-epimorphic for any $f$.

The \textbf{canonical Grothendieck topology} $J_{can}^{\Ccal}$ on $\Ccal$ is the topology whose covering sieves are precisely the universally effective-epimorphic ones. If $\Ccal$ is a Grothendieck topos, this coincides with the Grothendieck topology whose covering sieves are those containing small jointly epic families. %cite
\end{dfn}

We recall the following result, which appears as \cite[Proposition 4.36]{Dense}:

\begin{lemma}
\label{lem:coincide}
Let $\Ecal$ be a Grothendieck topos and $\Ccal$ a small full separating subcategory of $\Ecal$. Let $S$ be a sieve in $\Ccal$ on an object $C$ and let $D_S$ be the diagram in $\Ccal$ described in Definition \ref{dfn:effective}. Suppose that the colimit of $D_S$ in $\Ecal$ lies in $\Ccal$. Then $S$ is universally effective-epimorphic in $\Ccal$ if and only if it is the restriction to $\Ccal$ of a sieve containing a small jointly epic family in $\Ecal$.
\end{lemma}

\begin{thm}
\label{thm:canon}
Suppose $\Ecal$ is supercompactly generated. Let $J_r$ be the Grothendieck topology on $\Ccal_s$ whose covering sieves are those containing strict epimorphisms. Then $\Ecal \simeq \Sh(\Ccal_s,J_r)$.

Similarly, if $\Ecal$ is compactly generated, and $J_c$ is the Grothendieck topology on $\Ccal_c$ whose covering sieves are those containing strictly epic (equivalently, jointly epic) finite families. Then $\Ecal \simeq \Sh(\Ccal_c,J_c)$.
\end{thm}
\begin{proof}
By Giraud's theorem, given a (small, full) separating subcategory $\Ccal$ of objects in a Grothendieck topos $\Ecal$, we have an equivalence of toposes $\Ecal \simeq \Sh(\Ccal,J_{can}^{\Ecal}|_{\Ccal})$, where $J_{can}^{\Ecal}|_{\Ccal}$ is the restriction of the canonical topology on $\Ecal$ to $\Ccal$, whose covering sieves are the intersections of $J_{can}^{\Ecal}$-sieves with $\Ccal$. Thus it suffices to show in each case that the restriction of the canonical topology is the topology described in the statement.

By Lemma \ref{lem:closed2}, the principal (resp. finitely generated) sieves $S$ on $\Ccal_s$ (resp. $\Ccal_c$), whose corresponding diagrams $D_S$ are funneling (resp. multifunneling) diagrams, have colimits contained in $\Ccal_s$ (resp. $\Ccal_c$), so Lemma \ref{lem:coincide} applies. Thus these are effective epimorphic sieves if and only if the generating morphism is a strict epimorphism in $\Ccal_s$ (resp. the generating morphisms form a strictly epimorphic finite family in $\Ccal_c$).

Now given any sieve $S$ containing a jointly epic family on an object of $\Ccal_s$ (resp. $\Ccal_c$) in $\Ecal$, by the definition of supercompact (resp. compact) objects, $S$ must contain an epimorphism (resp. a finite covering family). In particular, every $J_{can}^{\Ecal}|_{\Ccal_s}$-covering sieve contains a $J_{can}^{\Ecal}|_{\Ccal_s}$-covering \textit{principal} sieve. Similarly, every $J_{can}^{\Ecal}|_{\Ccal_c}$-covering sieve contains a $J_{can}^{\Ecal}|_{\Ccal_c}$-covering \textit{finitely generated} sieve.

It follows that the strict epimorphisms in $\Ccal_s$ are precisely the morphisms generating universally effective-epimorphic families (and similarly for strict jointly epimorphic families in $\Ccal_c$), as required. Alternatively, see the proof of Proposition \ref{prop:representable} below for a direct argument showing that the strict epimorphisms (resp. strictly epic finite families) are stable.
\end{proof}

\begin{rmk}
It should be clear by now from our joint treatment of supercompactness and compactness that much of our analysis can be applied to more general notions of compactness. Indeed, Theorem \ref{thm:canon} is an explicit special case of Caramello's \cite[Proposition 4.36]{Dense}.

Suppose $P$ is some property of pre-sieves (families of morphisms with common codomain); then we may define \textit{$P$-compact} objects in a topos $\Ecal$ as those for which every jointly epic covering family contains a jointly epic presieve satisfying $P$. If $P$ satisfies suitable composition and stability criteria, which Caramello specifies, and the full subcategory $\Ccal_P$ of $\Ecal$ on the $P$-compact objects is separating and closed under certain colimits, then $\Ecal$ is equivalent to the category of sheaves on $\Ccal_P$ for the topology generated by the effective-epimorphic $P$-presieves in $\Ccal_P$. For supercompactness, $P$ is the property `is a singleton', while for ordinary compactness, $P$ is the property `is finite', and our earlier results show that these do satisfy Caramello's criteria.

While we shall not attempt to extend the present paper to this most general case, we encourage the reader to explore whether any given topos of interest to them is $P$-compactly generated for some suitable property $P$, and if so to compute the corresponding site produced by Caramello's result.
\end{rmk}

The advantage of the intrinsic expressions for the Grothendieck topologies in Theorem \ref{thm:canon} is that it guarantees that the categories of (super)compact objects contain enough information to completely reconstruct the toposes \textit{by themselves}. This immediately gives us results such as the following:
\begin{crly}
Suppose $\Ecal$ and $\Ecal'$ are supercompactly generated toposes and $\Ccal_s$, $\Ccal'_s$ are their respective categories of supercompact objects. Then $\Ecal \simeq \Ecal'$ if and only if $\Ccal_s \simeq \Ccal'_s$.
\end{crly}

\subsection{Cokernels, Well-Supported Objects and Two-Valued Toposes}
\label{ssec:twoval}

We saw in Lemmas \ref{lem:compactepi} and \ref{lem:compactmono} that when $\Ecal$ is compactly generated, every epimorphism in $\Ccal_c$ is strict and every monomorphism in $\Ccal_c$ is regular. This leads us to wonder under what extra conditions these facts hold true in $\Ccal_s$, given that $\Ecal$ is supercompactly generated.

\begin{xmpl}
To properly motivate this section, we show that epimorphisms in $\Ccal_s$ need not coincide with those in $\Ecal$. Let $\Dcal$ be the category
\[\begin{tikzcd}
A & B \ar[l, "l"'] \ar[r, "r"] & C.
\end{tikzcd}\]
In the topos $\Ecal$ of presheaves on $\Dcal$ it is easily calculated that the supercompact objects are precisely the representables, so $\Dcal$ coincides with $\Ccal_s$ (we shall see in Proposition \ref{prop:localic} that this argument is valid for all posets). The morphisms $l$ and $r$ are trivially epic in $\Dcal$ but are not epic in $\Ecal$.
\end{xmpl}

Our main tool in this section is the following definition.
\begin{dfn}
Given a morphism $f:A \to B$, its \textbf{cokernel}\footnote{Not to be confused with the \textit{cokernel pairs} mentioned in the proof of Lemma \ref{lem:compactmono}.} $B \to B/f$ is the pushout:
\[\begin{tikzcd}
A \ar[r, "f"] \ar[d, "!"'] & B \ar[d] \\
1 \ar[r, "x"] & B/f \ar[ul, "\ulcorner", phantom, very near start].
\end{tikzcd}\]
\end{dfn}

Cokernels are useful for understanding epimorphisms thanks to the following result.
\begin{lemma}
\label{lem:isokernel}
A morphism $f: A \to B$ of a topos $\Ecal$ is an epimorphism if and only if the lower morphism $x:1 \to B/f$ of its cokernel is an isomorphism (or equivalently, an epimorphism).
\end{lemma}
\begin{proof}
If $f$ is an epimorphism we have:
\[\begin{tikzcd}
A \ar[r, "f", two heads] \ar[d, "!"'] & B \ar[d] \\
1 \ar[r, "x", two heads] & B/f \ar[ul, "\ulcorner", phantom, very near start],
\end{tikzcd}\] 
since the pushout of an epimorphism is epic. But any quotient of $1$ in a topos is an isomorphism, as required.

Conversely, if $x:1 \to B/f$ is an isomorphism, we can consider the epi-mono factorization $f = m \circ e$:
\[\begin{tikzcd}
A \ar[rr, bend left, "f"] \ar[r, "e", two heads] \ar[d, "!"'] & A' \ar[r, "m", hook] \ar[d, "!"'] & B \ar[d] \\
1 \ar[rr, bend right, "x"'] \ar[r, "\sim"] & 1 \ar[r, "\sim"] \ar[ul, "\ulcorner", phantom, very near start] & B/f.
\end{tikzcd}\]
By the first part, the left hand square and outside rectangle are both pushouts, which makes the right hand square a pushout. But in a topos (or any adhesive category), a pushout square in which the upper horizontal morphism is monic is also a pullback square. Thus $m$ is an isomorphism, and $f$ is epic.
\end{proof}

Recall that an object $A$ of a category with a terminal object is \textbf{well-supported} if the unique morphism $!_A:A \to 1$ is an epimorphism in $\Ecal$. Using cokernels, we obtain a partial dual to Lemma \ref{lem:monocoincide} even without requiring $\Ecal$ to be (super)compactly generated.
\begin{lemma}
\label{lem:presic}
Let $\Ecal$ be a topos and $\Ccal_s$ the usual subcategory. Let $A$ be an object of $\Ccal_s$ which is well-supported as an object of $\Ecal$. Then a morphism $A \to B$ of $\Ccal_s$ is an epimorphism in that category if and only if it is an epimorphism in $\Ecal$.
\end{lemma}
\begin{proof}
An epimorphism of $\Ecal$ lying in $\Ccal_s$ clearly remains epic there.

Conversely, if $e : A \too B$ is epic in $\Ccal_s$, consider the cokernel $B \to B/e$:
\[\begin{tikzcd}
A \ar[r, "e"] \ar[d, "!_{A}"', two heads] & B \ar[d, "q", two heads] \\
1 \ar[r, "x"] & B/e \ar[ul, "\ulcorner", phantom, very near start].
\end{tikzcd}\]
By assumption $B/e$ is supercompact as $q$ is epic in $\Ecal$. The unique morphism $!_{A}:A \to 1$ factors through $e$ via the unique morphism $!_{B}:B \to 1$. Thus we have $qe = x!_{A} = x!_{B}e$, and since these expressions are composed of morphisms lying in $\Ccal_s$ where $e$ is epic, it follows that $q = x!_{B}$ whence $x$ is epic and hence an isomorphism. Hence $e$ is epic in $\Ecal$, by Lemma \ref{lem:isokernel}.
\end{proof}

It is worth noting that the existence of any well-supported supercompact object forces the terminal object of $\Ecal$ to be supercompact. Moreover, this proof concertedly fails when the object $A$ is not well-supported:

\begin{lemma}
\label{lem:coker}
Let $\Ecal$ be a supercompact topos. Then $\Ccal_s$ is closed under cokernels if and only if every supercompact object is well-supported.
\end{lemma}
\begin{proof}
Given a supercompact object $A$, consider its \textbf{support}, which is the subterminal object $U \hookrightarrow 1$ in the factorization of the morphism $!_A$. By Lemma \ref{lem:closed}, $U$ is supercompact, so if $\Ccal_s$ is closed under cokernels, the pushout of $U \hookrightarrow 1$ along itself must be in $\Ccal_s$. But the colimit morphisms $1 \rightrightarrows 1 +_U 1$ are jointly epic, so one of them must be an epimorphism and hence an isomorphism, which forces $U \cong 1$. Thus $A$ is well-supported, as required.

Conversely, if every object $A$ of $\Ccal_s$ is well-supported then the cokernel (in $\Ecal$) of a morphism $A \to B$ in $\Ccal_s$ is a quotient of $B$ and so is supercompact. Thus $\Ccal_s$ is closed under cokernels, as required.
\end{proof}

We shall see a relevant sufficient condition for this to occur in Proposition \ref{prop:hype2}. In this setting we can also strengthen Lemma \ref{lem:monocoincide}.

\begin{schl}
\label{schl:regularmono}
Let $\Ecal$ be a topos such that every object of $\Ccal_s$ is well-supported. Then monomorphisms in $\Ccal_s$ inherited from $\Ecal$ are regular. In particular, if $\Ecal$ is also supercompactly generated, then all monomorphisms in $\Ccal_s$ are regular.
\end{schl}
\begin{proof}
By Proposition \ref{prop:hype2}, the hypotheses guarantee that the cokernel of a morphism in $\Ccal_s$ also lies in that category.

As remarked in the proof of Lemma \ref{lem:isokernel}, the pushout square defining the cokernel of an inclusion of supercompact objects $i: A \hookrightarrow B$ is also a pullback in $\Ecal$. It follows easily that $i$ is the equalizer in $\Ecal$ of the morphisms $q, x!_B: B \rightrightarrows B/i$ described in the proof of Lemma \ref{lem:presic}, and is consequently also their equalizer in $\Ccal_s$.

If $\Ecal$ is supercompactly (resp. compactly) generated then we may apply Lemma \ref{lem:monocoincide} to conclude that the above applies to all monomorphisms of $\Ccal_s$ (resp. $\Ccal_c$).
\end{proof}

On the other hand, we shall see in Example \ref{xmpl:freemon} that it is not in general possible to strengthen the properties of epimorphisms in $\Ccal_s$ beyond the consequence of the proof of Theorem \ref{thm:canon} that they are strict. In particular, they are not regular in general.

Finally, we explicitly record how Lemma \ref{lem:presic} sjimplifies the expression for the Grothendieck topology $J_r$ induced on $\Ccal_s$ from Theorem \ref{thm:canon}.
\begin{crly}
\label{crly:site}
Let $\Ecal$ be a supercompactly generated topos and $\Ccal_s$ its full subcategory of supercompact objects. Suppose every object of $\Ccal_s$ is well-supported. Let $J_r$ be the topology on $\Ccal_s$ whose covering sieves are precisely those containing epimorphisms. Then $\Ecal \simeq \Sh(\Ccal_s,J_r)$.
\end{crly}

\subsection{Proper, Polite and Precise Morphisms}
\label{ssec:proper}

In this section we present some classes of geometric morphism whose inverse image functors interact well with supercompact and compact objects. The material in this section contains more technical topos theoretic concepts than that in the earlier sections, so may be skipped on a first reading; the key result is Proposition \ref{prop:relpres}. 

The latter part of Lemma \ref{lem:scompact}, regarding compact objects, states precisely that the geometric morphism $\Ecal/C \to \Set$ is proper in the sense of Moerdijk and Vermeulen \cite[Definition I.1.8]{Compact} (see also Johnstone \cite[C3.2.12 and C1.5.5]{Ele}), or equivalently that $\Ecal/C$ is a compact Grothendieck topos. We can generalise Moerdijk and Vermeulen's definition of proper morphisms to the arbitrary union case in order to capture the idea of supercompactness. In order to do this, we recall some classic topos-theoretic constructions, from \cite[Chapter 2]{TT}.

Recall from \cite[Definition 2.11]{TT} that an \textbf{internal category} $\Ibb$ in a topos (or, more generally, a cartesian category) $\Ecal$ consists of:
\begin{enumerate}[(i)]
	\item An \textit{objects of objects} $I_0$ and an \textit{object of morphisms} $I_1$ in $\Ecal$, and
	\item Morphisms $i:I_0 \to I_1$, $d,c:I_1 \rightrightarrows I_0$ and $m: I_2 \to I_1$,
\end{enumerate}
where $I_2$ is the \textit{object of composable pairs}, defined as the pullback:
\[\begin{tikzcd}
I_2 \ar[r, "\pi_2"] \ar[d,"\pi_1"']
\ar[dr, phantom, "\lrcorner", very near start] &
I_1 \ar[d, "d"]\\
I_1 \ar[r, "c"] &
I_0.
\end{tikzcd}\]
The morphisms define the \textit{identity morphisms}, \textit{domains}, \textit{codomains} and \textit{composition}, respectively. These are required to satisfy the equations $di = ci = \id_{I_0}$, $dm = d\pi_1$, $cm = c\pi_2$, $m(\id \times m) = m(m \times \id)$ and $m(\id \times i) = m(i \times \id) = \id_{I_1}$. These are diagrammatic translations of the axioms for ordinary categories in $\Set$. An \textbf{internal functor} between internal categories is a pair of morphisms between the respective objects of objects and objects of morphisms commuting with the respective structure morphisms.

\begin{dfn}
\label{dfn:diagram}
We say an internal category in $\Ecal$ is \textbf{inhabited}\footnote{This is the constructive term for `non-emptiness', meaning `has an element'. Translating this condition into the internal language of the topos $\Ecal$, we get the well-supportedness condition.} if its object of objects is well-supported. An internal category is \textbf{directed} if it is moreover directed in the internal logic of $\Ecal$ (which implies inhabitedness).
\end{dfn}

Given an internal category $\Ibb$ in $\Ecal$, we recall from \cite[Definition 2.14]{TT} that an (internal) \textbf{diagram} of shape $\Ibb$ consists of:
\begin{enumerate}
 	\item An object $a: F_0 \to I_0$ of $\Ecal/I_0$, and
 	\item A morphism $b: F_1 \to F_0$, 
\end{enumerate}
where $F_1$ is the defined as the pullback
\[\begin{tikzcd}
F_1 \ar[r, "\pi_2"] \ar[d,"\pi_1"']
\ar[dr, phantom, "\lrcorner", very near start] &
F_0 \ar[d, "a"]\\
I_1 \ar[r, "d"] &
I_0,
\end{tikzcd}\]
such that $ab = c\pi_2$, $b(\id \times i) = \id_{F_0}$, and $e(e \times \id) = e(\id \times m)$. In $\Set$, this data captures the encoding of a functor into $\Set$ via the Grothendieck construction.

There is an accompanying notion of natural transformation, and hence we obtain the internal diagram category $[\Ibb,\Ecal]$. This is a topos over $\Ecal$, which is proved using the comonadicity theorem in \cite[Corollary 2.33]{TT}.

\begin{dfn}
\label{dfn:polite}
For every object $E$ of $\Ecal$ and every internal category $\Ibb$ (resp. inhabited internal category $\Ibb$; filtered internal category $\Ibb$) in $\Ecal/E$, there is an induced pullback square of diagram toposes\footnote{See \cite[Corollary B3.2.12]{Ele} for an explanation of why this square is a pullback.}:
\begin{equation}
\label{eq:polite}
\begin{tikzcd}[column sep = large]
{[f^*(\Ibb){,}\Fcal/f^*(E)]} \ar[r, "(f/E)^{\Ibb}"] \ar[d,"\pi"']
\ar[dr, phantom, "\lrcorner", very near start] &
{[\Ibb{,}\Ecal/E]} \ar[d, "\pi"]\\
\Fcal/f^*(E) \ar[r, "f/E"] &
\Ecal/E,
\end{tikzcd}
\end{equation}
where $f^*(\Ibb)$ is the internal category in $\Fcal/f^*(E)$ obtained by applying $(f/\Ibb)^*$; this functor preserves finite limits, and hence the structure of an internal category. Each vertical morphism labelled $\pi$ has inverse image functor sending an object to the `constant diagram'; as well as a right adjoint $\pi_*$, this functor always has an $\Ecal$-indexed left adjoint $\pi_!$. The functors $\pi_*$ and $\pi_!$ send an internal diagram to its (internal) limit and colimit, respectively.

We call a geometric morphism $f:\Fcal\to \Ecal$ \textbf{precise} (resp. \textbf{polite}; \textbf{proper}) if the square above satisfies the condition $\pi_! \circ (f/E)^{\Ibb}_*(V) \cong (f/E)_* \circ \pi_!(V)$ for every subterminal object $V$ of $[f^*(\Ibb),\Fcal/f^*(E)]$. This can be understood as stating that the direct image of $f$ preserves $\Ecal$-indexed (resp. $\Ecal$-indexed inhabited; $\Ecal$-indexed directed) unions of subobjects.

We call a Grothendieck topos \textbf{supercompact} (resp. \textbf{compact}) if its unique geometric morphism to $\Set$ is precise (resp. proper), which by Lemma \ref{lem:scompact} occurs if and only if the terminal object has these properties. The global sections morphism of a Grothendieck topos is polite if and only if the topos is supercompact or degenerate (that is, the terminal object is either supercompact or initial).
\end{dfn}

Note that Moerdijk and Vermeulen denote $\pi_!$ by $\infty^*$ because, in the proper case, $\pi_!$ preserves finite limits and hence is the inverse image functor of a geometric morphism in the opposite direction. In the precise and polite cases the left adjoint $\pi_!$ is in general not left exact, so this notation no longer makes sense.

\begin{xmpl}
\label{xmpl:compshf}
The presheaf topos $[\Ccal\op,\Set]$ is supercompact if and only if $\Ccal$ is a funnel in the sense of Definition \ref{dfn:funnel}. Indeed, if $\Ccal$ has a weakly terminal object $C_0$ then there is by inspection an epimorphism $\yon(C_0) \too 1$ in $[\Ccal\op,\Set]$, and conversely if $1$ is supercompact then one of the morphisms $\yon(C) \to 1$ must be an epimorphism, since these are jointly epic, whence every object of $\Ccal$ admits at least one morphism to the corresponding object of $\Ccal$.

More generally, $[\Ccal\op,\Set]$ is compact if and only if $\Ccal$ is a multifunnel in the sense of Definition \ref{dfn:multifun}.
\end{xmpl}

An immediate consequence of Definition \ref{dfn:polite} is that we can relativise the concepts of supercompactness and compactness to depend on the base topos over which we work (so far we have been implicitly working over $\Set$). Viewing a geometric morphism $\Fcal \to \Ecal$ as expressing $\Fcal$ as a topos over $\Ecal$, an object $X$ of $\Fcal$ is \textbf{$\Ecal$-compact} if the composite geometric morphism $\Fcal/X \to \Fcal \to \Ecal$ is proper, for example. Conversely, since in this paper we will only be concerned with objects which are supercompact relative to some fixed base topos $\Scal$, it makes sense to employ the broader classes of geometric morphism introduced in \cite[Chapter V]{Compact}.

\begin{dfn}
\label{dfn:relpolite}
Let $p:\Ecal \to \Scal$ and $q:\Fcal \to \Scal$ be toposes over $\Scal$. A geometric morphism $f:\Fcal \to \Ecal$ over $\Scal$ is \textbf{$\Scal$-relatively precise} (resp. \textbf{$\Scal$-relatively polite}; \textbf{$\Scal$-relatively proper}) if its direct image preserves arbitrary $\Scal$-indexed unions (resp. inhabited $\Scal$-indexed unions; directed $\Scal$-indexed unions) of subobjects. Explicitly, this requires that the respective conditions of Definition \ref{dfn:polite} hold for diagrams $\Ibb$ in $\Ecal$ of the form $p^*(\Ibb')$, where $\Ibb'$ is a diagram category of the appropriate type in $\Scal$. We shall assume $\Scal$ is $\Set$ in the remainder, and so we drop the `$\Scal$-' prefixes. %Since our arguments are constructive (with one or two exceptions amongst the examples later on), however, they apply over an arbitrary base topos.
\end{dfn}

All of these definitions appear hard to work with in general for the simple reason that internal diagram categories take a significant amount of computation to express and work with concretely (which is to say externally) in any given case. However, the $\Scal$-relative notions conveniently coincide with their external counterparts, in a sense made precise in Lemma \ref{lem:relprop} below. Also, all of the notions are clearly stable under slicing and composition, from which we can extract general consequences which are sufficient for the purposes of this paper.

By inspection, we have the following relationships between Definitions \ref{dfn:polite} and \ref{dfn:relpolite}.
\begin{lemma}
\label{lem:relvsnorel}
Consider a commuting triangle of geometric morphisms:
\[\begin{tikzcd}
\Fcal \ar[rr, "f"] \ar[dr, "q"'] & & \Ecal \ar[dl, "p"]\\
& \Set. &
\end{tikzcd}\]
Then:
\begin{enumerate}
	\item If $f$ is precise, it is relatively precise.
	\item The morphisms $p$ and $q$ are relatively precise morphism if and only if they are precise.
	\item If $f$ is relatively precise and $p$ is precise, then $q$ is precise.
\end{enumerate}
The same statements hold when `precise' is replaced by `polite' or `proper'.
\end{lemma}

A handy consequence of this for our objects of interest is the following:
\begin{crly}
\label{crly:polite}
Let $f: \Fcal \to \Ecal$ be a geometric morphism between Grothendieck toposes. If $f$ is relatively precise (resp. relatively polite, relatively proper), then $f^*$ preserves supercompact (resp. `supercompact or initial', compact) objects.
\end{crly}
\begin{proof}
Given an object $E$ of $\Ecal$ and a relatively precise (resp. relatively polite, relatively proper) morphism $f$, consider the triangle:
\[\begin{tikzcd}
\Fcal/f^*(E) \ar[rr, "f/E"] \ar[dr] & & \Ecal/E \ar[dl]\\
& \Set. &
\end{tikzcd}\]
If $E$ is supercompact (resp. `supercompact or initial', compact), then the global sections morphism of $\Ecal/E$ is also precise (resp. polite, proper), so $f^*(E)$ must be supercompact (resp. supercompact or initial, compact) by Lemma \ref{lem:relvsnorel}.3.
\end{proof}

In order to make more explicit arguments, we now extend Moerdijk and Vermeulen's characterisation of relatively proper geometric morphisms in \cite[Proposition V.3.7(i)]{Compact}.
\begin{lemma}
\label{lem:relprop}
A geometric morphism $f:\Fcal \to \Ecal$ over $\Scal$ is relatively precise (resp. relatively polite, relatively proper) if and only if for any $\Scal$-indexed\footnote{Taking $\Scal$ to be $\Set$, $\Scal$-indexed just means set-indexed, or small.} (resp. $\Scal$-inhabited, $\Scal$-directed) jointly epimorphic family $\{g_i: X_i \hookrightarrow f^*(Y)\}$ of subobjects in $\Fcal$ there exists a jointly epimorphic family $\{h_j: Y_j \to Y\}$ in $\Ecal$ such that each $f^*(h_j)$ factors through some $g_i$.
\end{lemma}
\begin{proof}
Suppose $f$ has one of the relative properties and we are given a collection of subobjects of the relevant type. By assumption, their union is preserved by $(f/Y)_*$, whence there are subobjects $h_j: Y_j \hookrightarrow Y$ (the images under $(f/Y)_*$ of the subobjects) which are also jointly epic. By construction, each of these must have image under $f^*$ which factors through one or more of the $X_i$.

Conversely, given an $\Scal$-indexed diagram (of the relevant type) of subterminal objects $\{g_i: X_i \to f^*(Y)\}$ in $\Fcal/f^*(Y)$, let $U \hookrightarrow f^*(Y)$ be the union of the $g_i$, and let $Y' \hookrightarrow Y$ be its image under $(f/Y)_*$ in $\Ecal/Y$. Applying $f^*$, we have a monomorphism $f^*(Y') \hookrightarrow U$, so we can pull back to get a jointly epimorphic family of the appropriate shape $\{g'_i: X'_i \to f^*(Y')\}$ over $f^*(Y')$. This is precisely data of the required form to apply the hypotheses, and the covering family of $Y'$ thus provided ensures that the union of the images of the $g_i$ under $(f/Y)_*$ is precisely $Y'$.
\end{proof}

It is an indirect consequence of Lemma \ref{lem:relprop} that the converse of Corollary \ref{crly:polite} cannot hold in general. Indeed, if the only compact object of $\Ecal$ is the initial object, such as in Example \ref{xmpl:Nlocalic} below, then the preservation of compact or supercompact objects by $f^*$ is a vacuous condition, but the required properties in the characterisation of Lemma \ref{lem:relprop} are clearly non-trivial. However, $\Ecal$ failing to have enough supercompact (resp. compact) objects is the only obstacle.

\begin{prop}
\label{prop:relpres}
Let $f:\Fcal \to \Ecal$ be a geometric morphism, and suppose $\Ecal$ is supercompactly generated. Then $f$ is relatively precise if and only if $f^*$ preserves supercompact objects, and relatively polite if and only if $f^*$ preserves `supercompact or initial' objects. If $\Ecal$ is merely compactly generated, then $f$ is relatively proper if and only if $f^*$ preserves compact objects.
\end{prop}
\begin{proof}
Given an arbitrary (resp. inhabited) jointly epic collection of subobjects $\{ g_i: X_i \hookrightarrow f^*(Y) \}$ in $\Fcal$, consider a (possibly empty) covering of $Y$ by supercompact subobjects $h_j : Y_j \hookrightarrow Y$ in $\Ecal$. Since each $f^*(Y_j)$ is supercompact (resp. supercompact or initial), pulling back the inclusions $g_i$ along $f^*(h_j)$, we conclude that one of the resulting inclusions $f^*(Y_j)$ (if there are any) must be an isomorphism by Lemma \ref{lem:scompact}; in the inhabited case, this is trivially true when $f^*(Y_j)$ is initial. Hence the $f^*(h_j)$ each factor through one of the $g_i$, whence the criteria of Lemma \ref{lem:relprop} are fulfilled. The compactly generated case, with a directed family of subobjects, is analogous.
\end{proof}

Since the initial object in any topos is strict, the distinction between relatively precise and relatively polite morphisms is indeed as small as this proposition makes it seem.

\begin{lemma}
\label{lem:0refl}
Given a geometric morphism $f: \Fcal \to \Ecal$, $f^*$ reflects the initial object if and only if $f_*$ preserves it.
\end{lemma}
\begin{proof}
If $f^*$ reflects $0$, then considering the counit $f^*f_*(0) \to 0$, we conclude that $f^*f_*(0)$ is initial (by strictness of the initial object), whence $f_*(0) \cong 0$ so $f_*$ preserves the initial object. Conversely, if $f_*$ preserves $0$, then given $C$ with $f^*(C) \cong 0$ the unit $C \to f_*f^*(C) \cong 0$ is a morphism to $0$, whence $C$ is initial.
\end{proof}

\begin{crly}
\label{crly:precise}
Let $f: \Fcal \to \Ecal$ be a geometric morphism, and suppose $\Ecal$ is supercompactly generated. Then $f$ is relatively precise if and only if it is relatively polite and $f^*$ reflects the initial object.
\end{crly}
\begin{proof}
A (relatively) precise morphism is (relatively) polite. Considering $0$ as an empty union of subobjects of any given object, it is preserved by $f_*$ by relative preciseness, so $f^*$ reflects $0$ by Lemma \ref{lem:0refl}, as required. Note that this implication holds even when $\Ecal$ is not supercompactly generated.

Conversely, given that $f$ is relatively polite and $f^*$ reflects $0$, we have that $f^*$ preserves `supercompact or initial' objects by Proposition \ref{prop:relpres}, but a supercompact object $X$ of $\Ecal$ is sent to an initial object if and only if it is initial, which is impossible, so $f^*(X)$ is supercompact, as required.
\end{proof}

\subsection{Inclusions and Surjections}
\label{ssec:surj}

Recall that a geometric morphism $f:\Fcal \to \Ecal$ is a \textbf{surjection} if $f^*$ is faithful, or equivalently if $f^*$ is a comonadic functor. Meanwhile, a geometric morphism $f:\Fcal \to \Ecal$ is an \textbf{inclusion} (or \textit{embedding}) if its direct image $f_*$ is full and faithful. See \cite[\S VII.4]{MLM} or \cite[A4.2]{Ele} for general results regarding these, which we shall assume familiarity with. In this section we examine how these two types of geometric morphism interact with supercompact and compact objects, as well as some of the classes of geometric morphism introduced in the last section.

Recall that a geometric inclusion $f: \Fcal \to \Ecal$ is \textbf{closed} if there is some subterminal object $U$ in $\Ecal$ such that $f_*f^*$ sends an object $X$ to the pushout of the product projections from $X \times U$. The following result illustrates why the precise, polite and proper morphisms are too restrictive for analysing supercompactly generated subtoposes.

\begin{lemma}
\label{lem:incl}
An inclusion of toposes $f: \Fcal \to \Ecal$ is proper if and only if it is closed, if and only if it is polite. An inclusion is precise if and only if it is an equivalence.
\end{lemma}
\begin{proof}
The first part is the conclusion of \cite[Remark C3.2.9]{Ele}, where it is observed that a closed inclusion has a direct image functor preserving arbitrary inhabited ($\Ecal$-indexed) unions of subobjects and that closed inclusions are stable under slicing, and conversely that any proper inclusion is closed.

Given that $f$ is a closed inclusion, any nontrivial subobject of the corresponding subterminal object $U$ in $\Ecal$ is sent by $f^*$ to the initial object in $\Fcal$, so the initial object is reflected if and only if $U$ is initial, in which case $f$ is an equivalence.
\end{proof}

The relative versions of these properties are well-behaved with respect to surjections and inclusions.

\begin{prop}
\label{prop:surjinc}
Consider a factorisation of a geometric morphism $f$,
\begin{equation}
\label{eq:factors}
\begin{tikzcd}
\Fcal \ar[rr, "f"] \ar[dr, "q"'] & & \Ecal\\
& \Gcal \ar[ur, "p"']. &
\end{tikzcd}
\end{equation}
\begin{enumerate}
	\item Suppose that $p$ is an inclusion. Then if $f$ is relatively precise (resp. relatively polite, relatively proper), so is $q$.
	\item Suppose that $q$ is a surjection. Then if $f$ is relatively precise (resp. relatively polite, relatively proper), so is $p$.
\end{enumerate}
It follows that a geometric morphism is relatively precise (resp. relatively polite, relatively proper) if and only if both parts of its surjection-inclusion factorisation are.
\end{prop}
\begin{proof}
We use the characterisation of these properties from Lemma \ref{lem:relprop}.

1. Given a jointly epic family (resp. inhabited family, directed family) $\{g_i: X_i \hookrightarrow q^*(Z)\}$ in $\Ecal$, we may express $Z$ up to isomorphism as $p^*(Y)$ (taking $Y = p_*(Z)$, say), so this can without loss of generality be seen as a family $\{g_i: X_i \hookrightarrow f^*(Y)\}$. Since $f$ is relatively precise (resp. relatively polite, relatively proper), we have a jointly epic family $\{h_j:Y_j \to Y\}$ such that each $f^*(h_j)$ factors through some $g_i$, and hence $\{p^*(h_j): p^*(Y_j) \to Z\}$ is the required family to fulfill the criterion of Lemma \ref{lem:relprop}.

2. Given a jointly epic family (resp. inhabited family, directed family) $\{g_i: X_i \hookrightarrow p^*(Y)\}$ in $\Ecal$, we have $\{q^*(g_i): q^*(X_i) \hookrightarrow f^*(Y)\}$ in $\Fcal$ being of the desired form to ensure that there is a covering family $\{h_j: Y_j \to Y\}$ such that each $f^*(h_j)$ factors through some $q^*(g_i)$. But then $q^*$ being conservative forces the $p^*(h_j)$ to factor through $g_i$, as required. Indeed, the intersection of $g_i$ with the image of $p^*(h_j)$ is preserved by $q^*$, and one of the sides of the resulting pullback square is sent to an isomorphism, which is reflected by $q^*$.
\end{proof}

\begin{crly}
\label{crly:ptcompact}
Any topos $\Ecal$ with a surjective point is supercompact. Any topos with a finite jointly surjective collection of points is compact.
\end{crly}
\begin{proof}
The hypotheses correspond to the existence of a surjection $\Set \to \Ecal$, or a surjection $\Set/K \to \Ecal$ with $K$ finite. The unique morphism $\Set \to \Set$ is (an equivalence, and hence) precise, while the morphism $\Set/K \to \Set$ is proper. Thus by Proposition \ref{prop:surjinc}.2 and Lemma \ref{lem:relvsnorel}, we conclude that $\Ecal$ is supercompact (resp. compact) over $\Set$.
\end{proof}

When the codomain of the geometric morphism is supercompactly (resp. compactly) generated, the simpler characterisation of Proposition \ref{prop:relpres} comes to our aid.

\begin{lemma}
\label{lem:incl2}
Suppose that $\Ecal$ is supercompactly (resp. compactly) generated and $f: \Fcal \to \Ecal$ is an inclusion. If $p$ is relatively polite (resp. relatively proper), $\Fcal$ is also supercompactly (resp. compactly) generated.
\end{lemma}
\begin{proof}
By Proposition \ref{prop:relpres}, the separating collection of supercompact (resp. compact) objects in $\Ecal$ is mapped by the inverse image of the relatively polite (resp. relatively proper) inclusion $f$ to a separating collection of objects of the same kind (or initial objects) in $\Fcal$. 
\end{proof}

We shall see in Corollary \ref{crly:incl} that Lemma \ref{lem:incl2} is optimal, in the sense that a topos is supercompactly (resp. compactly) generated if and only if it admits a relatively polite (resp. relatively proper) inclusion into a presheaf topos.

\begin{crly}
\label{crly:factor2}
Suppose that $\Ecal$ is supercompactly (resp. compactly) generated. Then a geometric morphism $f: \Fcal \to \Ecal$ is relatively polite (resp. relatively proper) if and only if both parts of its surjection-inclusion factorisation have inverse images preserving supercompact or initial (resp. compact) objects.
\end{crly}
\begin{proof}
This follows from applying Proposition \ref{prop:relpres} to the factorisation in Proposition \ref{prop:surjinc}, using Lemma \ref{lem:incl2} to conclude that the intermediate topos must be supercompactly (resp. compactly) generated.
\end{proof}

More generally, surjections interact well with supercompact and compact objects.

\begin{lemma}
\label{lem:surjreflect}
Suppose $f:\Fcal \to \Ecal$ is a surjective geometric morphism. Then $f^*$ reflects supercompact, compact and initial objects. 
\end{lemma}
\begin{proof}
Since the inverse image functor of $f$ is comonadic, we have an equivalence between $\Ecal$ and the topos of coalgebras for the comonad on $\Fcal$ induced by $f$. Without loss of generality we work with coalgebras.

Given a coalgebra $(X,\alpha: X \to f^*f_*(X))$ and a jointly epic family $g_i: (U_i,\beta_i) \to (X,\alpha)$, since $f^*$ preserves arbitrary colimits, the underlying family of morphisms $g_i:U_i \to X$ in $\Fcal$ must be jointly epic. Thus if $X = f^*(X,\alpha)$ is supercompact (resp. compact), one of the $g_i$ must be an epimorphism (resp. there is a finite jointly epic subfamily of the $g_i$). Since $f^*$ moreover creates colimits, we conclude (via the same colimit diagram employed in Lemma \ref{lem:compactepi}) that $g_i$ is an epimorphism in $\Ecal$ too (resp. that the finite subfamily lifts to a jointly epic finite subfamily). Thus $(X,\alpha)$ is supercompact, as required.

Preservation of the initial object by $f_*$, which is equivalent to reflection of $0$ by $f^*$ by Lemma \ref{lem:0refl}, is due to strictness of the initial object forcing $f^*f_*(0) \cong 0$.
\end{proof}

It follows from Lemma \ref{lem:surjreflect} and Corollary \ref{crly:precise} that a geometric surjection is relatively precise if and only if it is relatively polite.

\subsection{Hyperconnected Morphisms}
\label{ssec:hype}

Recall that a geometric morphism $f:\Fcal \to \Ecal$ is said to be \textbf{hyperconnected} if $f^*$ is full and faithful and its image is closed in $\Fcal$ under subobjects and quotients (up to isomorphism). See \cite[Section A4.6]{Ele} for some background on hyperconnected morphisms, as well as localic morphisms, which shall feature in Section \ref{sec:xmpls}, below.

When it comes to hyperconnected morphisms into $\Set$, there are various alternative characterisations; note that these rely on properties of $\Set$, so are not true constructively.

\begin{prop}
\label{prop:hype2}
Let $\Ecal$ be a Grothendieck topos. Then the following are equivalent:
\begin{enumerate}
	\item The unique geometric morphism $\Ecal \to \Set$ is hyperconnected.
	\item $\Ecal$ is \textbf{two-valued}: the only subterminal objects are the initial and terminal objects.
	\item Every object of $\Ecal$ is either well-supported (the unique morphism $X \to 1$ is an epimorphism) or initial, but not both.
	\item $\Ecal$ is non-degenerate and has a separating set of well-supported objects.
\end{enumerate}
\end{prop}
\begin{proof}
($1 \Leftrightarrow 2$) The inverse image of a hyperconnected geometric morphism is full and faithful and closed under subobjects, so in particular the only subobjects of $1$ in $\Ecal$ is $0$. Conversely, if $\Ecal$ is two-valued, consider the hyperconnected-localic factorization of the unique geometric morphism $\Ecal \to \Set$; the intermediate topos is the localic reflection of $\Ecal$, equivalent to the topos of sheaves on the locale of subterminal objects of $\Ecal$, so is equivalent to $\Set$. Thus the morphism $\Ecal \to \Set$ is hyperconnected.

($2 \Leftrightarrow 3$) If $\Ecal$ is two-valued, the monic part of the epi-mono factorization of $X\to 1$ is non-trivial iff $X \cong 0$, so $X$ is either initial or well-supported. Conversely, any proper subterminal object fails to be well-supported, so if $3$ holds there can be no proper subterminals and $\Ecal$ is two-valued.

($3 \Leftrightarrow 4$) Since $\Ecal$ is a Grothendieck topos, it has some generating set of objects; any such set is still generating after excluding the initial object, and since $0$ is distinct from $1$, any generating set contains a non-initial object, so we obtain a generating set of non-initial objects as required. Conversely, an inhabited colimit of well-supported objects is well-supported, so any non-initial object is well-supported if there is a separating set of well-supported objects. 
\end{proof}

In particular, since supercompact objects are not initial, we obtain a necessary and sufficient condition for the hypotheses in Section \ref{ssec:twoval} to hold:
\begin{crly}
\label{crly:epic}
If $\Ecal$ is a two-valued (Grothendieck) topos, every supercompact object in $\Ecal$ is well-supported. In particular, a morphism $A \to B$ in the subcategory $\Ccal_s$ of supercompact objects objects of $\Ecal$ is an epimorphism if and only if it is epic in $\Ecal$, so every epimorphism in $\Ccal_s$ is strict. Conversely, if $\Ecal$ is supercompactly generated, then $\Ccal_s$ is closed under cokernels if and only if $\Ecal$ is two-valued.
\end{crly}

Thus we can give the counterexample, promised earlier, to the hypothesis that epimorphisms in $\Ccal_s$ are regular when $\Ecal$ is two-valued.
\begin{xmpl}
\label{xmpl:freemon}
Let $M$ be the free monoid on two generators, viewed as a category, and let $\Ecal = [M\op,\Set]$, where the objects are viewed as right $M$-sets. It is easily checked that this topos is two-valued, and being a presheaf topos it is supercompactly generated. The supercompact objects in this topos are precisely the cyclic right $M$-sets.

Given a cyclic $M$-set $N$ generated by $n$, the relation generated by a pair of morphisms $f,g: N \rightrightarrows M$ is the one identifying $f(n)k$ with $g(n)k$ for each $k \in M$. A case-by-case analysis of the possible pairs of elements $f(n),g(n)$ demonstrates that there is no pair of which the epimorphism $M \too 1$ is a coequalizer, so this epimorphism is not regular in the category of supercompact objects of $\Ecal$.
\end{xmpl}

The main reason for our interest in hyperconnected morphisms, however, is that they create the structure of supercompactly (and compactly) generated Grothendieck toposes which we are studying in this paper.

\begin{crly}
\label{crly:hypepres}
If $f: \Fcal \to \Ecal$ is a hyperconnected geometric morphism between Grothendieck toposes, then it is precise. Thus (being surjective) $f^*$ preserves and reflects supercompact, compact and initial objects.
\end{crly}
\begin{proof}
We extend Moerdijk and Vermeulen's proof that hyperconnected morphisms are proper in \cite[Proposition I.2.4]{Compact}, replacing $\infty^*$ with $\pi_!$.

Suppose $f : \Fcal \to \Ecal$ is hyperconnected, and consider a diagram of the form
\begin{equation}
\label{eq:polite2}
\begin{tikzcd}[column sep = large]
{[f^*(\Ibb){,}\Fcal]} \ar[r, "f^{\Ibb}"] \ar[d,"\pi"']
\ar[dr, phantom, "\lrcorner", very near start] &
{[\Ibb{,}\Ecal]} \ar[d, "\pi"]\\
\Fcal \ar[r, "f"] &
\Ecal.
\end{tikzcd}
\end{equation}
Since $f^{\Ibb}$ is a pullback of $f$, it is hyperconnected too, so that any $V \hookrightarrow 1$ in $[f^*(\Ibb){,}\Fcal]$ is of the form $(f^{\Ibb})^*(U)$ for some $U \hookrightarrow 1$ in $[\Ibb{,}\Ecal]$ (the restriction of a hyperconnected morphism to the subterminal objects is an equivalence). Thus,
\[f_*\pi_!(V) = f_*\pi_!(f^{\Ibb})^*(U) = f_*f^*\pi_!(U) = \pi_!(U),\]
where the last equality holds since $f^*$ is full and faithful. But $U = (f^{\Ibb})_*(f^{\Ibb})^*(U) = (f^{\Ibb})_*(V)$, so $f_*\pi_!(V) = \pi_!(f^{\Ibb})_*(V)$, as required. The same argument applied in slices gives the result.

Preservation of supercompact and compact objects by $f^*$ then follows from Proposition \ref{prop:relpres} and Corollary \ref{crly:precise} (preservation of the initial object is automatic), while reflection follows from Lemma \ref{lem:surjreflect}.
\end{proof}

\begin{thm}
\label{thm:hype2}
Let $f:\Fcal \to \Ecal$ be a hyperconnected geometric morphism between elementary toposes. If $\Fcal$ is a Grothendieck topos, then so is $\Ecal$. Assuming this is so, if $\Fcal$ is also:
\begin{enumerate}[(i)]
	\item is supercompactly generated, or
	\item is compactly generated, or
	\item has enough points, or
	\item is two-valued,
\end{enumerate}
then $\Ecal$ inherits that property.
\end{thm}
\begin{proof}
Let $\Ccal$ be a (full subcategory on a) small separating set of objects in $\Fcal$. Then every object of $\Fcal$ is a colimit of a diagram in $\Ccal$. Given an object $Q$ of $\Ecal$, we can express $f^*(Q)$ as such a colimit; write $\lambda_i: C_i \to f^*(Q)$ with $C_i \in \Ccal$ for the legs of the colimit cone. By taking image factorizations of the $\lambda_i$ we obtain an expression for $f^*(Q)$ as a colimit where the legs of the colimit cone are all monomorphisms. Since the image of $f^*$ is closed under subobjects, we obtain an expression for $f^*(Q)$ as a colimit of objects of the form $f^*(D_i)$ with $D_i$ in $\Ecal$, which moreover are quotients of objects in the separating subcategory $\Ccal$ of $\Fcal$.

Thus, since $f^*$ creates all small colimits, the quotients of objects in $\Ccal$ lying in $\Ecal$ form a separating set. Also, $\Ecal$ is locally small since $\Fcal$ is. Thus by the version of Giraud's Theorem that appears in \cite[C2.2.8(v)]{Ele}, say, $\Ecal$ is a Grothendieck topos.

The inheritance of property (i) (resp. (ii)) follows from Corollary \ref{crly:hypepres}, taking $\Ccal$ in the above to be $\Ccal_s$ (resp. $\Ccal_c$) and noting that the objects $f^*(D_i)$ in the argument above are supercompact (resp. compact) in $\Fcal$ by Lemma \ref{lem:closed}, whence the $D_i$ are so in $\Ecal$ by Corollary \ref{crly:hypepres}.

For (iii), if $\Fcal$ has enough points, which is to say that there is a collection of geometric morphisms $\Set \to \Fcal$ whose inverse images are jointly faithful, then composing each point with the hyperconnected morphism to $\Ecal$ gives such a collection for $\Ecal$. Finally, for (iv), note once again that the restriction of $f$ to subterminal objects is an equivalence.
\end{proof}
In spite of our reliance on Corollary \ref{crly:hypepres} here, we shall see in Example \ref{xmpl:surjprec} that we cannot extend Theorem \ref{thm:hype2}(i) or (ii) to relatively precise or relatively proper surjections, although parts (iii) and (iv) do apply in that situation.

We can summarise the results from the last two sections as stability results for supercompactly and compactly generated toposes.

\begin{thm}
\label{thm:closure}
Suppose $\Ecal$ is a supercompactly (resp. compactly) generated Grothendieck topos. If $\Fcal$ is:
\begin{enumerate}
	\item The domain of a closed inclusion $f:\Fcal \to \Ecal$, or more generally, the domain of a relatively polite (resp. relatively proper) inclusion into $\Ecal$,
	\item The domain of a local homeomorphism $g: \Fcal \simeq \Ecal/X \to \Ecal$, or
	\item The codomain of a hyperconnected morphism $h: \Ecal \to \Fcal$,
\end{enumerate}
then $\Fcal$ is also a supercompactly (resp. compactly) generated Grothendieck topos.
\end{thm}
\begin{proof}
Let $\Ccal_s$ and $\Ccal_c$ be the (separating) subcategories of supercompact and compact objects in $\Fcal$, respectively.

1. For any inclusion, the images of objects in $\Ccal_s$ (resp. $\Ccal_c$) under $f^*$ would form a separating set for $\Ecal$. The stated properties ensure that these objects are all supercompact or initial (resp. compact), so that $\Ecal$ is supercompactly (resp. compactly) generated, by Corollary \ref{crly:polite}.

2. The objects with domain in $\Ccal_s$ (resp. $\Ccal_c$) in any slice $\Fcal/X$ form a separating set. These lifted objects inherit the property of being supercompact (resp. compact), by Lemma \ref{lem:scompact} and the standard result $(\Fcal/X)/(Y \to X) \simeq \Fcal/Y$.

3. This is immediate from Theorem \ref{thm:hype2}.
\end{proof}

%% file: TSGT_Principal_Sites.tex
\section{Principal and Finitely Generated Sites}
\label{sec:principal}

So far, we have established properties of `canonical' sites for supercompactly and compactly generated toposes. In the spirit of Caramello's work in \cite{SCGI}, we obtain in this section a broader class of sites whose toposes of sheaves have these properties.

We will occasionally make use of the following fact regarding representable sheaves, which is easily derived from the fact that the functor $\ell$ is a dense morphism of small-generated sites, in the sense described by Caramello in \cite{Dense}:
\begin{fact}
\label{fact1}
A sieve $T$ on $\ell(C)$ in $\Sh(\Ccal,J)$ is jointly epimorphic if and only if the sieve $\{f:D \to C \mid \ell(f)\in T\}$ is $J$-covering.
\end{fact}

\subsection{Stable Classes}
\label{ssec:stable}

\begin{dfn}
\label{dfn:stable}
Let $\Ccal$ be a small category. A class $\Tcal$ of \textit{morphisms} in $\Ccal$ is called \textbf{stable} if it satisfies the following three conditions:
\begin{enumerate}
	\item $\Tcal$ contains all identities.
	\item $\Tcal$ is closed under composition.
	\item For any $f:C\to D$ in $\Tcal$ and any morphism $g:B \to D$ in $\Ccal$, there exists a commutative square
	\begin{equation}
	\begin{tikzcd}
	A \ar[r, "f'"] \ar[d, "g'"'] & B \ar[d, "g"]\\
	C \ar[r, "f"'] & D,
	\end{tikzcd}
	\label{eq:stable}
	\end{equation}
	in $\Ccal$ with $f'\in \Tcal$.
\end{enumerate}
These correspond to the necessary and sufficient conditions for $\Tcal$-morphisms to be singleton presieves generating a Grothendieck topology, as presented by Mac Lane and Moerdijk in \cite[Exercise III.3]{MLM}. We call the resulting Grothendieck topology the \textbf{principal topology generated by $\Tcal$}.
\end{dfn}

In \cite{SPM}, Kondo and Yasuda call a stable class of morphisms \textit{semi-localizing}, in reference to a related definition in \cite{GZ}. They call a principal topology an \text{$A$-topology}, presumably because the atomic topology is an example of a principal topology; see Example \ref{xmpl:atomic}. We have chosen a naming convention that we believe to be more evocative in this context.

Continuing the parallel investigation of compactness, we obtain a related concept by replacing individual morphisms by finite families of morphisms.
\begin{dfn}
\label{dfn:stable2}
Let $\Tcal'$ be a class of \textit{finite families} of morphisms with specified common codomain in $\Ccal$, we say $\Tcal'$ is \textbf{stable} if
\begin{enumerate}[1'.]
	\item $\Tcal'$ contains the families whose only member is the identity.
	\item $\Tcal'$ is closed under multicomposition, in that if $\{f_i:D_i \to C \mid i=1,\dotsc,n\}$ is in $\Tcal'$ and so are $\{g_{j,i}: E_j \to D_i \mid j=1,\dotsc,m_i\}$ for each $i = 1,\dotsc,n$, then $\{f_i \circ g_{j,i}\}$ is also a member of $\Tcal'$.
	\item For any $\{f_i:D_i \to C \mid i=1,\dotsc,n\}$ in $\Tcal'$ and any morphism $g:B \to C$ in $\Ccal$, there is a $\Tcal'$-family $\{h_j:A_j \to B \mid j=1,\dotsc,m\}$ such that each $g \circ h_j$ factors through one of the $f_i$.
\end{enumerate}
These are necessary and sufficient conditions for $\Tcal'$-families to generate a Grothendieck topology, which we call the \textbf{topology (finitely) generated by $\Tcal'$}.
\end{dfn}

Note that we do not require $\Ccal$ to have pullbacks in the above definitions, so it is sensible to compare them with the usual notion of stability with respect to pullbacks.
\begin{lemma}
\label{lem:pbstable}
Let $\Tcal$ be a stable class of morphisms in $\Ccal$ with the additional `push-forward' property:
\begin{enumerate}
	\setcounter{enumi}{3}
	\item Given any morphism $f$ of $\Ccal$ such that $f \circ g \in \Tcal$ for some morphism $g$ of $\Ccal$, we have $f \in \Tcal$.
\end{enumerate}
This in particular is true of the class of epimorphisms, for example. Then morphisms in $\Tcal$ are stable under any pullbacks which exist in $\Ccal$.
\end{lemma}
\begin{proof}
Given a pullback of a $\Tcal$ morphism, the comparison between any square provided by \eqref{eq:stable} and this pullback provides a factorization of a $\Tcal$ morphism through the pullback, so by the assumed property the pullback is also in $\Tcal$, as required.
\end{proof}

As Kondo and Yasuda remark in \cite{SPM}, we can extend a stable class of morphisms $\Tcal$ to the stable class $\hat{\Tcal}$ of morphisms whose principal sieves contain a member of $\Tcal$ (that is, morphisms $f$ such that $f \circ g \in \Tcal$ for some $g$) without changing the resulting principal topology. Thus we can safely assume that stable classes satisfy axiom 4 of Lemma \ref{lem:pbstable} if we so choose. The advantage of enforcing this assumption is that it gives a one-to-one correspondence between stable classes of morphisms in a category and the principal Grothendieck topologies on that category, since we can recover the classes as those morphisms which generate covering sieves.

\begin{rmk}
\label{rmk:split}
The class of identity morphisms in any category satisfies axioms 1,2 and 3, but in order to satisfy axiom 4 it must be extended to the class of split epimorphisms, which is easily verified to satisfy all four axioms in any category. Thus the class of split epimorphisms corresponds to the trivial Grothendieck topology. It is worth noting also that every split epimorphism is regular and hence strict.
\end{rmk}

We may similarly extend a class $\Tcal'$ to a maximal class $\hat{\Tcal}'$. However, for a class of finite families to be maximal, it must be closed under supersets as well as under the equivalent of the push-forward property of Lemma \ref{lem:pbstable}, so we have two extra axioms:
\begin{enumerate}[1'.]
	\setcounter{enumi}{3}
	\item Given a finite family $\mathfrak{f} = \{f_i:D_i \to C\}$ of morphisms in $\Ccal$ such that every morphism in some $\Tcal'$-family over $C$ factors through one of the $f_i$, we have $\mathfrak{f} \in \Tcal'$.
	\item Any finite family $\mathfrak{f} = \{f_i:D_i \to C\}$ of morphisms in $\Ccal$ containing a $\Tcal'$-family $\mathfrak{f}'$ is also a member of $\Tcal'$.
\end{enumerate}
The equivalent of the pullback stability statement of Lemma \ref{lem:pbstable} is as follows.

\begin{schl}
\label{schl:finpb}
Suppose that $\Tcal'$ is a stable class of finite families of morphisms in a category $\Ccal$ satisfying the additional axioms 4' and 5'. Then $\Tcal'$ is stable under any pullbacks which exist in $\Ccal$, in that given a family $\{g_i:E_i \to C\}$ in $\Tcal'$ and $h:C'\to C$ such that the pullback of $g_i$ along $h$ exists for each $i$, the family $\{h^*(g_i):E'_i \to C'\}$ is in $\Tcal'$.
\end{schl}
We leave the proof, and the verification that enforcing axioms 4' and 5' gives a one-to-one correspondence between stable classes of finite families and finitely generated Grothendieck topologies, to the reader.

\begin{dfn}
\label{dfn:sites}
Let $\Ccal$ be a small category, $\Tcal$ a stable class of its morphisms and $J_{\Tcal}$ the corresponding Grothendieck topology. We call a site $(\Ccal,J_{\Tcal})$ constructed in this way a \textbf{principal site}. Similarly, for a stable class of finite families $\Tcal'$ on $\Ccal$, we have a corresponding Grothendieck topology $J_{\Tcal'}$; a site of the form $(\Ccal,J_{\Tcal'})$ shall be called a \textbf{finitely generated site}.
\end{dfn}

\begin{prop}
\label{prop:representable}
Let $\Ccal$ be a small category and $J$ a Grothendieck topology on it. Then the representable sheaves are all supercompact if and only if $J = J_{\Tcal}$ is a principal topology for some stable class $\Tcal$ of morphisms in $\Ccal$. In particular, the topos of sheaves on a principal site $(\Ccal,J_{\Tcal})$ is supercompactly generated.

Similarly, the representable sheaves are all compact in $\Sh(\Ccal,J)$ if and only if $J = J_{\Tcal'}$ for a stable class $\Tcal'$ of finite families of morphisms in $\Ccal$, and the topos of sheaves on a finitely generated site $(\Ccal,J_{\Tcal'})$ is compactly generated.
\end{prop}
\begin{proof}
By Fact \ref{fact1}, for $J = J_{\Tcal}$, given a covering family on $\ell(C)$, the sieve it generates must contain the image of a $\Tcal$-morphism; call it $g$. Then any morphism in the family through which $g$ factors generates a covering sieve and so $g$ must be an epimorphism. Thus $\ell(C)$ is supercompact, as required.

Conversely, given that $\ell(C)$ is supercompact for every $C$, let $\Tcal$ be the class of morphisms $f$ such that $\ell(f)$ is epimorphic. We first claim that $\Tcal$ is a stable class. Indeed, axioms 1, 2 and 4 are immediate; to see that axiom 3 holds, suppose that $f: C \to D$ is in $\Tcal$ and $g: B \to D$ is any $\Ccal$ morphism. Then we may consider the pullback of $\ell(f)$ against $\ell(g)$ in $\Sh(\Ccal,J)$:
\[\begin{tikzcd}
A \ar[r, "f'", two heads] \ar[d, "g'"'] \ar[dr, phantom, "\lrcorner", very near start] & \ell(B) \ar[d, "\ell(g)"]\\
\ell(C) \ar[r, "\ell(f)"', two heads] & \ell(D).
\end{tikzcd}\]
Since $A$ is covered by objects of the form $\ell(C')$, by supercompactness of $\ell(B)$ there must be an epimorphism $\ell(C')\too \ell(B)$ factoring through the pullback, and in turn the sieve it generates must contain an epimorphism in the image of $\ell$, so that we ultimately recover the square \eqref{eq:stable} required for axiom 3.

Now we show that $J = J_{\Tcal}$. Given a $J$-covering sieve $S$ on $C$ (generated by a family of morphisms, for example), the sieve generated by $\ell(S) := \{\ell(g) \mid g\in S\}$ necessarily covers $\ell(C)$, and therefore $\ell(S)$ contains an epimorphism by supercompactness, so that the original sieve must have contained a member of $\Tcal$, which gives $J \subseteq J_{\Tcal}$. Conversely, $J_{\Tcal}$-covering sieves are certainly $J$-covering, so $J_{\Tcal} \subseteq J$. Thus $J$ is a principal topology, as claimed.

The argument for finitely generated sites is almost identical after replacing $\Tcal$-morphisms with finite $\Tcal'$-families in the first part, and defining $\Tcal'$ families to be those finite families which are mapped by $\ell$ to jointly epic families in the second part.
\end{proof}

In particular, we may extend a stable class of morphisms $\Tcal$ to a stable class of families of morphisms $\Tcal'$ by viewing the morphisms in $\Tcal$ as one-element families. Then it is clear that $J_{\Tcal} = J_{\Tcal'}$.

Intermediate between the two classes of sites discussed so far are a class which we call \textbf{quasi-principal sites}: these are sites $(\Ccal,J)$ such that for every object $C \in \Ccal$, either the empty sieve is a covering sieve on $C$ or every covering sieve on $C$ contains a principal sieve. Observe that if $\Ccal'$ is the full subcategory of $\Ccal$ on the latter class of objects (which we can always construct over $\Set$), then $\Sh(\Ccal,J) \cong \Sh(\Ccal',J_{\Tcal})$, where $\Tcal$ is the class of morphisms generating principal covering sieves.

The following result subsumes Lemma 4.11 of Bridge's thesis \cite{TAC}; our work up to this point allows us to avoid any direct manipulation with sheaves in the proof.

\begin{crly}
\label{crly:incl}
Let $(\Ccal,J)$ be a small site. Then the inclusion $\Sh(\Ccal,J) \to [\Ccal\op,\Set]$ is relatively precise (resp. relatively proper) if and only if $J$ is a principal (resp. finitely generated) topology. The inclusion is relatively polite if and only if $(\Ccal,J)$ is a quasi-principal site.
\end{crly}
\begin{proof}
Since $[\Ccal\op,\Set]$ is supercompactly generated, by Proposition \ref{prop:relpres}, the inclusion $\Sh(\Ccal,J) \to [\Ccal\op,\Set]$ is relatively precise if and only if the sheafification functor preserves supercompact objects; this in particular requires all of the objects $\ell(C)$ to be supercompact, which occurs if and only if $J$ is principal by Proposition \ref{prop:representable}. But any other supercompact objects are quotients of representables, so  $\ell(C)$ being supercompact for every $C$ is also sufficient. As usual, the relatively proper case is analogous.

For relative politeness, we relax the conditions above to requiring that each representable is sent to a supercompact or initial object, and note that the empty sieve is covering on $C$ if and only if $\ell(C)$ is initial.
\end{proof}

\subsection{Morphisms Between Sheaves on Principal Sites}
\label{ssec:representable}

In order to better understand the relationship between a principal site and the topos it generates, we employ some results of Caramello in \cite{Dense}, which enable us to describe morphisms in a Grothendieck topos $\Sh(\Ccal,J)$ in terms of those in a presenting site $(\Ccal,J)$.

For a general site $(\Ccal,J)$, the functor $\ell:\Ccal \to \Sh(\Ccal,J)$ is neither full nor faithful. To describe the full collection of morphisms in the sheaf topos, several notions are introduced in \cite[\S 2]{Dense}, of which we introduce the relevant special cases here.

For morphisms $h,k:A \rightrightarrows B$, we say $h$ and $k$ are \textbf{$J$-locally equal} (written $h \equiv_J k$) if there is a $J$-covering sieve $S$ on $A$ such that $h\circ f = k \circ f$ for every $f \in S$. If $J$ is principal (resp. finitely generated) then this is equivalent to saying that there is some $\Tcal$-morphism which equalizes $h$ and $k$ (resp. a $\Tcal'$-family whose members all equalize $h$ and $k$). This leads naturally to the following moderately technical definitions:

\begin{dfn}
Let $\Ccal$ be a small category and $\Tcal$ a stable class of morphisms in $\Ccal$. Then for objects $A,B$ in $\Ccal$, a \textbf{$\Tcal$-span} from $A$ to $B$ is a span
\begin{equation}
\begin{tikzcd}
& C \ar[dl,"f"'] \ar[dr,"g"] &\\
A & & B,
\end{tikzcd}
\label{eq:span}
\end{equation}
such that $f$ is in $\Tcal$. A \textbf{$\Tcal$-arch} is a $\Tcal$-span such that for any $h,k:D\rightrightarrows C$ with $f \circ h = f \circ k$ we have $g \circ h \equiv_{J_{\Tcal}} g \circ k$.

Similarly, for $\Tcal'$ a stable class of finite families of morphisms on $\Ccal$, a \textbf{$\Tcal'$-span} is a finite (possibly empty) family of spans:
\begin{equation}
\begin{tikzcd}
& C_i \ar[dl,"f_i"'] \ar[dr,"g_i"] &\\
A & & B,
\end{tikzcd}
\label{eq:multispan}
\end{equation}
such that $\{f_1, \dotsc, f_n\}$ is in $\Tcal'$. A \textbf{$\Tcal'$-multiarch} is a $\Tcal'$-multispan such that for any $h:D \to C_i$, $k:D\to C_{i'}$ with $f_i \circ h = f_{i'}\circ k$ we have $g_i \circ h \equiv_{J_{\Tcal'}} g_{i'} \circ k$.

The constituent morphisms in any span or multispan will be refered to as their \textbf{legs}.
\end{dfn}

\begin{lemma}
\label{lem:Trel}
Let $(\Ccal,J_\Tcal)$ be a principal site. Let $\Arch_\Tcal(A,B)$ be the collection of $\Tcal$-arches from $A$ to $B$ in $\Ccal$. For each $\Tcal$-arch $(t,g) \in \Arch_\Tcal(A,B)$, there is a (necessarily unique) morphism $\ell(t,g): \ell(A) \to \ell(B)$ in $\Sh(\Ccal,J_{\Tcal})$ such that $\ell(t,g) \circ \ell(t) = \ell(g)$. The mapping $\ell$ so defined is a surjection from $\Arch_\Tcal(A,B)$ to the set of morphisms from $\ell(A)$ to $\ell(B)$ in $\Sh(\Ccal,J_{\Tcal})$.

Similarly, letting $\mArch_{\Tcal'}(A,B)$ be the set of $\Tcal'$-multiarches from $A$ to $B$, $\ell$ induces a surjection from $\mArch_{\Tcal'}(A,B)$ to $\Hom_{\Sh(\Ccal,J_{\Tcal'})}(A,B)$.
\end{lemma}
\begin{proof}
This is immediate from \cite[Proposition 2.5]{Dense}.
\end{proof}

Intuitively it seems that the collections $\Arch_\Tcal(A,B)$ ``should'' be the morphisms of a category. However, while Caramello's result \cite[Proposition 2.5(iv)]{Dense} suggests a composition of arches coming from covering families that generate sieves, this composition produces a maximal family of arches presenting the composite rather than a single $\Tcal$-arch; there is a similar problem for multiarches. We therefore examine what structure exists in general, and identify some sufficient conditions under which arches and multiarches admit a composition operation.

\begin{lemma}
\label{lem:archcat}
Let $(\Ccal,J_{\Tcal})$ be a principal site. For each pair of objects $A$ and $B$ in $\Ccal$, let $\Span_\Tcal(A,B)$ be the collection of $\Tcal$-spans from $A$ to $B$. Then $\Span_\Tcal(A,B)$ admits a canonical categorical structure, where a morphism $x: (t:C \to A, g:C \to B) \to (t':C' \to A, g':C' \to B)$ is a morphism $x: C \to C'$ with $t = t' \circ x$ and $g = g' \circ x$. This restricts to give a category structure on $\Arch_\Tcal(A,B)$ too.

Expanding upon this, if $(\Ccal,J_{\Tcal'})$ is a finitely generated site, there is a canonical categorical structure on each collection of $\Tcal'$-multispans $\mSpan_{\Tcal'}(A,B)$, where $\vec{x}: (t_i:C_i \to A, g_i:C_i \to B) \to (t'_j:C'_j \to A, g'_j:C'_j \to B)$ consists of an index $j$ for each index $i$, and a morphism $x_i:C_i \to C'_j$ with $t_i = t'_j \circ x_i$ and $g_i = g'_j \circ x_i$. Note that any permutation of the spans forming a given $\Tcal'$-multispan form a $\Tcal'$-multispan which is isomorphic in this category. Once again, this structure restricts to the collections of multiarches.
\end{lemma}

Thus in the best case scenario where it \textit{is} possible to construct a weak composition on arches, we may obtain a bicategory (see \cite[Definition 1.1]{bicat} for a definition of bicategory) from the principal site $(\Ccal,J_{\Tcal})$, whose $0$-cells are the objects of $\Ccal$, whose $1$-cells are $\Tcal$-arches and whose $2$-cells are morphisms between these. The relationship between this bicategory and the subcategory of $\Sh(\Ccal,J_{\Tcal})$ on the representables is simply that of collapsing the vertical morphisms, in the following sense:

\begin{lemma}
\label{lem:Tarch}
Two $\Tcal$-arches (resp. $\Tcal'$-multiarches) from $A$ to $B$ on a principal (resp. finitely generated) site are identified by $\ell$ if and only if they are in the same component in the category $\Arch_\Tcal(A,B)$ (resp. $\mArch_{\Tcal'}(A,B)$) described in Lemma \ref{lem:archcat}.
\end{lemma}
\begin{proof}
In one direction, if $x: (t:C \to A, g:C \to B) \to (t':C' \to A, g':C' \to B)$, then by definition the unique morphism $\ell(t,g): \ell(A) \to \ell(B)$ with $\ell(g) = \ell(t,g) \circ \ell(t)$ also satisfies $\ell(g') = \ell(t,g) \circ \ell(t')$, whence $\ell(t',g') = \ell(t,g)$. Thus $\ell$ identifies $\Tcal$-spans in the same component.

Conversely, applying \cite[Proposition 2.5(iii)]{Dense}, two $\Tcal$-arches $(t:C \to A, g:C \to B)$ and $(t':C' \to A, g':C' \to B)$ induce the same morphism in $\Sh(\Ccal,J_{\Tcal})$ if and only if there is a $\Tcal$-morphism $s:D \to A$ and morphisms $h:D \to C$ and $h':D' \to C'$ with $t \circ h = s = t' \circ h'$ and $g \circ h \equiv_{J_{\Tcal}} g' \circ h'$. Expanding on the latter condition, this implies the existence of some $\Tcal$-morphism $u:E \to D$ equalizing $g \circ h$ and $g' \circ h'$. But then $(t \circ h \circ u, g \circ h \circ u) = (t' \circ h' \circ u, g' \circ h' \circ u)$ is easily shown to be a $\Tcal$-arch, and it admits morphisms $h \circ u$ and $h' \circ u'$ to $(t,g)$ and $(t',g')$ respectively, whence these are in the same connected component, as required.
\end{proof}

\begin{proposition}
\label{prop:composition}
Suppose that $\Tcal$ is a stable class of morphisms in a small category $\Ccal$ such that axiom 3 of Definition \ref{dfn:stable} provides stability squares \textbf{weakly functorially}. That is, calling an ordered pair of morphisms $(h:A \to D,s: B \to D)$ with $s \in \Tcal$ a \textbf{$\Tcal$-cospan} from $A$ to $B$, suppose the stability axiom defines a mapping from $\Tcal$-cospans to $\Tcal$-spans satisfying the following conditions:
\begin{enumerate}
	\item For any $\Tcal$-morphism $t:A \to B$, the $\Tcal$-span coming from $(\id_A,t)$ is isomorphic in $\Span_{\Tcal}(A,B)$ to $(t,\id_B)$.
	\item If $f:C \to D$, $g:B \to D$ and $k:B' \to B$ such that $f$ is a $\Tcal$-morphism, the $\Tcal$-span obtained by applying the stability mapping along $g$ and then $k$ is isomorphic in $\Span_{\Tcal}(B',C)$ to that obtained by applying it along $g \circ k$.
	\item If $f:C \to D$, $g:B \to D$ and $e:C' \to C$ such that $e$ and $f$ are $\Tcal$-morphisms, the $\Tcal$-span obtained by applying the stability mapping along $f$ and then $e$ is isomorphic in $\Span_{\Tcal}(B,C')$ to that obtained by applying it along $f \circ e$.
\end{enumerate}
Then there is a weak composition on $\Tcal$-arches, in the sense that there are mappings
\[ \circ : \Arch_{\Tcal}(B,C) \times \Arch_{\Tcal}(A,B) \to \Arch_{\Tcal}(A,C), \]
which are associative and unital up to isomorphism of $\Tcal$-arches. Moreover, this composition is natural in the second component up to isomorphism, in the sense that for each fixed $\Tcal$-arch $(u,h)$ in $\Arch_{\Tcal}(B,C)$, a morphism $x: (t,g) \to (t',g')$ in $\Arch_{\Tcal}(A,B)$ induces a morphism $(u,h) \circ x: (u,h) \circ (t,g) \to (u,h) \circ (t',g')$ in $\Arch_{\Tcal}(A,C)$, and the resulting mapping $(u,h) \circ - : \Arch_{\Tcal}(A,B) \to \Arch_{\Tcal}(A,C)$ is functorial up to unit and associativity isomorphisms.
\end{proposition}
\begin{proof}
Even without the listed conditions, stability provides a putative definition of the composition operation: given a consecutive pair of $\Tcal$-arches, we simply apply the stability axiom to the pair of morphisms with common codomain,
\[\begin{tikzcd}
& & P \ar[dl, "\Tcal \ni t''"'] \ar[dr, "g''"] & & \\
& C \ar[dl, "t"'] \ar[dr, "g"] & & C' \ar[dl, "t'"'] \ar[dr, "g'"] & \\
A & & B & & B';
\end{tikzcd}\]
that the resulting $\Tcal$-span $(t \circ t'', g' \circ g'')$ is a $\Tcal$-arch is easily checked. The extra conditions are needed to make this operation weakly unital and associative. The naturality in the second component is a direct consequence of the second condition.
\end{proof}

For brevity, we leave the analogous statement and proof of Proposition \ref{prop:composition} for finitely generated sites to the reader, noting that the analogue of $\Tcal$-cospans will not be the duals of $\Tcal'$-multispans, but the more restrictive shape of diagram relevant to the stability axiom 3'.

One situation where the hypotheses of Proposition \ref{prop:composition} are satisfied is when $\Ccal$ has pullbacks, by Lemma \ref{lem:pbstable}.

\begin{corollary}
Let $(\Ccal,J_{\Tcal})$ be a principal site where $\Ccal$ has pullbacks, such as the canonical sites on locally regular categories we shall see in Definition \ref{dfn:regcoh}. Then the objects of $\Ccal$, the $\Tcal$-spans on $\Ccal$ and the morphisms between these assemble into a bicategory. In particular, the composition operations of Proposition \ref{prop:composition} are natural in the first component.

As usual, the analogous result for finitely generated sites holds, but we omit the proof.
\end{corollary}
\begin{proof}
The verification of the conditions in Proposition \ref{prop:composition} is straightforward; note that we actually require a specified choice of pullbacks, but the mediating isomorphisms are provided by universal properties. For the final claim, we observe that in the third condition of Proposition \ref{prop:composition}, we no longer need to restrict ourselves to the case where $e:C' \to C$ is a $\Tcal$-morphism, since we can complete the defining rectangle with a pullback square ($\Tcal$-morphisms are indicated with double-headed arrows):
\[\begin{tikzcd}
A' \ar[dr, phantom, "\lrcorner", very near start]
\ar[rr, bend left, "(f\circ e)'", two heads]
\ar[d, "g''"'] \ar[r, "e'"'] &
A \ar[dr, phantom, "\lrcorner", very near start]
\ar[r,"f'"', two heads] \ar[d, "g'"'] & B \ar[d, "g"] \\
C' \ar[r, "e"] \ar[rr, bend right, "f\circ e"', two heads] &
C \ar[r,"f", two heads] & D.
\end{tikzcd}\]
The morphism $e'$ provides the morphism of $\Tcal$-spans corresponding to $e$ to make the composition natural (again, up to relevant isomorphisms) in the first component, as claimed.

The commutativity of the associativity and identity coherence diagrams which are required to formally make this a bicategory are guaranteed by the uniqueness in the universal property of the pullbacks involved.
\end{proof}

\begin{remark}
\label{rmk:sievecat}
The fact that $\Tcal$-arches do not assemble into a bicategory in full generality is not merely an artefact of us having restricted ourselves to the data of the stable classes of morphisms (resp. finite families), rather than the principal (resp. finitely generated) Grothendieck topologies they generate. If we expand our collections of morphisms to multiarches indexed by arbitrary $J_{\Tcal}$-covering families, Caramello's construction \cite[Proposition 2.5(iv)]{Dense} does give a canonical family representing the composite, but it typically fails to be unital, since composing with an identity $\Tcal$-span gives a strictly larger family. We can restrict to $J$-covering sieves to avoid this problem, but even then, without pullbacks the composition may not be weakly associative either, since multi-composition of $J$-covering sieves is not necessarily associative in the required sense.
\end{remark}

Returning to the general case, we observe that we do not need the composition operation to be well-defined at the level of $\Tcal$-arches in order to reconstruct the full subcategory of $\Sh(\Ccal,J_{\Tcal})$ on the representable sheaves.

\begin{scholium}
Let $(\Ccal,J_{\Tcal})$ be a principal site. Then the full subcategory of $\Sh(\Ccal,J_{\Tcal})$ on the representable sheaves is equivalent to the category whose objects are the objects of $\Ccal$ and whose morphisms $A \to B$ are indexed by the connected components of the category $\Arch_{\Tcal}(A,B)$.

Similarly, if $(\Ccal,J_{\Tcal'})$ is a finitely generated site, then the full subcategory of $\Sh(\Ccal,J_{\Tcal'})$ on the representable sheaves is equivalent to the category whose objects are the objects of $\Ccal$ and whose morphisms $A \to B$ are indexed by the connected components of the category $\mArch_{\Tcal'}(A,B)$.
\end{scholium}
\begin{proof}
Observe that in the definition of composition given in the proof of Proposition \ref{prop:composition}, \textit{any} choice of stability square will produce a $\Tcal$-arch (resp. $\Tcal'$-multiarch) lying in the same component of $\Arch_\Tcal(A,C)$ (resp. $\mArch_{\Tcal'}(A,C)$), since this $\Tcal$-arch (resp. $\Tcal'$-multiarch) will necessarily be mapped by $\ell$ to the composite of the morphisms corresponding to the pair of arches (resp. multiarches) being composed. Thus, even without weak functoriality, the composition is well-defined on connected components, as required.
\end{proof}

In the subcanonical case, where all $\Tcal$-morphisms are strict epimorphisms (resp. all $\Tcal'$-families are jointly strictly epimorphic), the computations from this section simplify greatly. Indeed, $\ell$ is full and faithful in this case, which means that every component of each category $\Arch_{\Tcal}(A,B)$ (resp. $\mArch_{\Tcal'}(A,C)$) contains a unique (multi)arch of the form $(\id_A,f)$. In a $\Tcal$-arch $(t,g)$, $g$ coequalizes every pair of morphisms which $t$ coequalizes by definition, whence the morphism $\ell(t,g)$ corresponds to the unique morphism $A \to B$ factorizing $g$ through $t$; the morphisms representing multiarches are recovered analogously from the universal properties of jointly strictly epic families.

\subsection{Quotients of Principal Sites}
\label{ssec:quotient}

Rather than directly computing the category of representable sheaves in $\Sh(\Ccal,J_{\Tcal})$ via $\Tcal$-arches, we might hope to simplify things by first modifying the principal site.

In Kondo and Yasuda's definition of `$B$-site', they assume that the underlying category is an `$E$-category', which is to say that every morphism is an epimorphism, \cite[Definitions 4.1.1, 4.2.1]{SPM}, which seems very restrictive. However, by taking the quotient of $\Ccal$ by a canonical congruence, we show here that we may at least assume that $\Tcal$ is contained in the class of epimorphisms of $\Ccal$ without loss of generality, since the corresponding topos of sheaves is equivalent to that on the original site. %(resp. that $\Tcal'$ is contained in the class of jointly epic families on $\Ccal$).

\begin{prop}
\label{prop:congruence}
Let $(\Ccal,J_{\Tcal})$ be a principal site. Then there is a canonical congruence $\sim$ on $\Ccal$ such that $(\Ccal/{\sim},J_{\Tcal/{\sim}})$ is a principal site with $\Tcal/{\sim}$ a subclass of the epimorphisms of $\Ccal/{\sim}$, and with $\Sh(\Ccal,J_{\Tcal}) \simeq \Sh(\Ccal/{\sim},J_{\Tcal/{\sim}})$.

Similarly, if $(\Ccal,J_{\Tcal'})$ is a finitely generated site, there is a congruence $\sim$ on $\Ccal$ such that $(\Ccal/{\sim},J_{\Tcal'/{\sim}})$ is a finitely generated site with $\Tcal'/{\sim}$ a subclass of the epimorphisms of $\Ccal/{\sim}$, and with $\Sh(\Ccal,J_{\Tcal'}) \simeq \Sh(\Ccal/{\sim},J_{\Tcal'/{\sim}})$.
\end{prop}
\begin{proof}
Simply let $f\sim f':C \to D$ whenever there is a morphism $h:C' \to C$ in $\Tcal$ with $fh=f'h$. To verify that this is a congruence, given $g \sim g':D \to E$ equalized by $k:D'\to D$, stability of $k$ along $fh = f'h$ gives $k' \in \Tcal$:
\[\begin{tikzcd}
C'' \ar[d, "k'"] \ar[rr] & & D' \ar[d, "k"] & \\
C' \ar[r, "h"] & C \ar[r,shift left, "f"] \ar[r,shift right, "f'"'] & D  \ar[r,shift left, "g"] \ar[r,shift right, "g'"'] & E;
\end{tikzcd}\]
$\Tcal$ is closed under composition and $gfhk' = g'f'hk'$, so $gf \sim g'f'$ as required.

By the definition of the congruence it is immediate that the morphisms in $\Tcal/{\sim}$ are epimorphisms. The canonical functor $(\Dcal,J_{\Tcal}) \to (\Dcal/{\sim},J_{\Tcal/{\sim}})$ is a morphism and comorphism of sites\footnote{See Definitions \ref{dfn:morsite} and \ref{dfn:comor} below for the definitions of morphisms and comorphisms of sites, respectively.}, since it is cover-preserving, cover-lifting and flat by inspection. In the terminology of \cite[Definition 5.14]{Dense}, the quotient functor is $J_{\Tcal}$-full, $J_{\Tcal}$-faithful and $J_{\Tcal/{\sim}}$-dense; indeed, being full and essentially surjective, only the $J_{\Tcal}$-faithfulness needs to be checked, and the definition of the congruence ensures that this holds. Thus this morphism of sites induces an equivalence of sheaf toposes by \cite[Proposition 7.16]{Dense}.

The congruence for a finitely generated site has $f \sim f'$ whenever $fh_i = f'h_i$ for each $h_i$ in a $\Tcal'$-family. The remainder of the proof is analogous.
\end{proof}

The reduced site $(\Ccal/{\sim},J_{\Tcal/{\sim}})$ in Proposition \ref{prop:congruence} can be obtained in various alternative ways. By construction:
\begin{scholium}
The functor $\Ccal \to \Ccal/{\sim}$ in the proof of Proposition \ref{prop:congruence} is the universal functor with domain $\Ccal$ sending $\Tcal$-morphisms to epimorphisms (resp. $\Tcal'$-families to jointly epimorphic families).
\end{scholium}
As such, it is not surprising that $\ell$ canonically factors through the functor $\Ccal \to \Ccal/{\sim}$, since $\ell$ sends $\Tcal$-morphisms to epimorphisms by Fact \ref{fact1}. The resulting site also coincides with the one obtained by lifting the hyperconnected-localic factorisation of the geometric equivalence of toposes to the level of morphisms of sites, as described in \cite[\S 6.5.4]{Dense}. See Section \ref{ssec:morsites} below for more on morphisms of sites.

After taking this quotient, $J_{\Tcal}$-local equality (resp. $J_{\Tcal'}$-local equality) reduces to ordinary equality, so that for example a $\Tcal$-arch from $A$ to $B$ is a $\Tcal$-span as in \eqref{eq:span} such that for any $h,k:D\rightrightarrows C$ with $f \circ h = f \circ k$ we have $g \circ h = g \circ k$. In particular, by considering the arches in which the $\Tcal$-morphism is an identity, we see that the functor $\ell:(\Ccal/{\sim},J_{\Tcal/{\sim}}) \to \Sh(\Ccal,J_{\Tcal})$ is \textit{faithful}, and further that $\ell:(\Ccal,J_{\Tcal}) \to \Sh(\Ccal,J_{\Tcal})$ is faithful if and only if every $\Tcal$-morphism is an epimorphism, which is to say that the congruence ${\sim}$ is trivial.

\begin{crly}
\label{crly:full}
Let $(\Ccal,J_{\Tcal})$ be a principal site and let ${\sim}$ be the congruence on $\Ccal$ from Proposition \ref{prop:congruence}. Then the functor $\ell:(\Ccal,J_{\Tcal}) \to \Sh(\Ccal,J_{\Tcal})$ is full if and only if every $\Tcal/{\sim}$-morphism in $\Ccal/{\sim}$ is a \textbf{strict} epimorphism.

Similarly, if $(\Ccal,J_{\Tcal'})$ is a finitely generated site, then $\ell:(\Ccal,J_{\Tcal'}) \to \Sh(\Ccal,J_{\Tcal'})$ is full if and only if every $\Tcal'/{\sim}$-family in $\Ccal/{\sim}$ is a strictly epimorphic family.
\end{crly}
\begin{proof}
If the hypothesis holds, then the site $(\Ccal/{\sim},J_{\Tcal/{\sim}})$ is subcanonical, so the induced functor to the topos is full and faithful, whence $\ell$ with domain $(\Ccal,J_{\Tcal})$ is full.

Conversely, given that $\ell$ is full, suppose that $t:C \to A$ is in $\Tcal$. Suppose that we have $g:C \to B$ such that whenever $h,k:C' \to C$ with $t \circ h = t \circ k$, we have $g \circ h = g \circ k$. Then $(t,g)$ is a $\Tcal$-arch, and there is a morphism $\ell(t,g)$ in $\Sh(\Ccal,J_\Tcal)$ completing the triangle. By fullness of $\ell$, this is the image of a morphism $A \to B$ in $\Ccal/{\sim}$, and by the definition of $\sim$ there is at most one such morphism, whence $t/{\sim}$ is a strict epimorphism, as claimed.

The argument for finitely generated sites is analogous.
\end{proof}

This site can alternatively be obtained by considering the lifting of the surjection-inclusion factorisation of the geometric equivalence of toposes to the level of sites, as seen in \cite[Theorem 6.3]{Dense}.

\begin{xmpl}
\label{xmpl:atomic}
Recall that a category $\Ccal$ satisfies the \textbf{right Ore condition} if any cospan can be completed to a commutative square. This is exactly the condition needed to make the class of \textit{all} morphisms of $\Ccal$ stable, and the corresponding principal topology is more commonly called the \textbf{atomic topology}, $J_{at}$, while the site $(\Ccal, J_{at})$ is called an \textbf{atomic site}. The above results show that we may reduce any atomic site to one in which every morphism is epic (hence a `$B$-site' in the terminology of Kondo and Yasuda). %\textcolor{orange}{From this point, one may extend the site via the atomic completion process for atomic sites described by Caramello in \cite[\S 4.3]{TGT}.}
\end{xmpl}

\subsection{Reductive and Coalescent Categories}
\label{ssec:redcat}

Returning to our study of the subcategories of supercompact and compact objects from the last section, we observe that the epimorphisms they inherit from $\Ecal$ always meet most of the requirements for stability.

\begin{lemma}
\label{lem:stable}
For any Grothendieck topos $\Ecal$, let $\Ccal_s$, $\Ccal_c$ the usual subcategories. Then the class $\Tcal$ of epimorphisms in $\Ccal_s$ which are inherited from $\Ecal$ satisfies axioms 1,2 and 4 for stable classes, while the class $\Tcal'$ of finite jointly epimorphic families on $\Ccal_c$ inherited from $\Ecal$ satisfies axioms 1',2',4' and 5'.
\end{lemma}
\begin{proof}
Clearly $\Ccal_s$ and $\Ccal_c$ inherit identities from $\Ecal$, so $\Tcal$ contains these and $\Tcal'$ contains the singleton families of the identities. Since epimorphisms are stable under composition in $\Ecal$, $\Tcal$ is closed under composition. Multicomposition of finite jointly covering families in $\Ecal$ is similarly direct, giving the second axiom for $\Tcal'$. Axioms 4' and 5' are also straightforward.
\end{proof}

By Theorem \ref{thm:canon}, if $\Ecal$ is supercompactly generated (resp. compactly generated) then axiom 3 (resp. axiom 3') must also be satisfied in each case. Note that the converse fails: stability of $\Ecal$-epimorphisms in $\Ccal_s$ or stability of $\Ecal$-epimorphic finite families in $\Ccal_c$ are not sufficient to guarantee that $\Ecal$ is supercompactly or compactly generated. Indeed, if $\Ecal$ a non-degenerate topos where the only compact object is the initial object, such as that exhibited in Example \ref{xmpl:R}, the stability axioms are trivially satisfied but $\Ecal$ is neither supercompactly nor compactly generated.

In the remainder of this section, we refine the concepts of principal and finitely generated sites in order to obtain a characterisation of the categories $\Ccal_s$ (resp. $\Ccal_c$) of supercompact (resp. compact) objects in supercompactly (resp. compactly) generated toposes.

\begin{dfn}
\label{dfn:reductive}
We say a small category $\Ccal$ is \textbf{reductive} if it has funneling colimits and its class of strict epimorphisms is stable. The \textbf{reductive topology} $J_r$ on a reductive category is the principal topology generated by its class of strict epimorphisms.

We say $\Ccal$ is \textbf{coalescent} if it has multifunneling colimits and its class of (strictly) epic finite families is stable. The \textbf{coalescent topology} $J_c$ on a coalescent category is the finitely generated topology on its class of strictly epic finite families.
\end{dfn}

\begin{rmk}
Some justification for this naming and notation is warranted. 

The names of the categories are intended to evoke the presence of funneling (resp. multifunneling) colimits, since any diagram in them of the respective shapes `is reduced' (resp. `coalesces') by composing with a suitable epimorphism (resp. jointly epimorphic family). Indeed, if we consider a functor from such a category to $\Set$ which preserves strict epimorphisms (resp. jointly strictly epimorphic families), then the images of these epimorphisms have the effect of reducing/coalescing equivalence classes for the relations generated by the images of these diagrams.

The names were also chosen to have their first few letters in common with \textit{regular} and \textit{coherent} respectively, since the regular and coherent objects in a topos are respectively subclasses of the supercompact and compact objects. Thus, while the $r$ and $c$ in $J_r$ and $J_c$ stand for \textit{reductive} and \textit{coalescent} respectively, we shall see that when a category is both regular and reductive, the regular topology (see Definition \ref{dfn:regtop} below) coincides with the reductive topology, so the $r$ on $J_r$ could also mean `regular'; similarly for coalescent and coherent.

%The names \textit{pre-regular category} and \textit{pre-coherent category} were considered, but Lack and Vitale used `pre-regular category' to describe a legitimate weakening of the well-established concept of regular category in \cite{ComProc}, and our reductive categories typically lack any semblance of the finite limit properties that feature prominently in the definition of regular category, as demonstrated in Example \ref{xmpl:nonreg}. The name \textit{strict category} similarly already carries the legitimate and unrelated meaning of a category whose objects are comparable by equality, so we avoided that name too. On the other hand, the name \textit{strict topology} in place of `reductive topology' might make clearer which morphisms are involved.
\end{rmk}

While not every stable class of finite families need contain a stable class of singleton morphisms, we record the fact that this does happen when the families involved are strictly epimorphic families.

\begin{lemma}
\label{lem:collred}
Any coalescent category is a reductive category with finite colimits and a strict initial object.
\end{lemma}
\begin{proof}
If $\Ccal$ is a coalescent category, it certainly has the required colimits, so it suffices to show stability of strict epimorphisms. Indeed, if $t:D \too C$ is strict and $g: B \to C$ is any morphism in $\Ccal$, then since $\{t\}$ is a strictly epic family, there is some strictly epic family $\{h_j: A_j \to B \mid i = 1, \dotsc, m\}$ over $B$ such that each $g \circ h_j$ factors through $t$. Factoring this family through the coproduct $A_1 + \cdots + A_m$ gives a strict epimorphism completing the required stability square (even in the case $m=0$).

To see that the initial object is strict, observe that the empty family is a strict jointly epic family on the initial object, so given any morphism $A \to 0$, stability forces the empty family to be jointly epic over $A$, whence $A$ is also an initial object.
\end{proof}

It would be easy to mistakenly conclude based on the results presented thus far that the subcategory of supercompact objects in the category of sheaves on a reductive category $\Ccal$ should be equivalent to $\Ccal$. The flaw in this reasoning lies in the fact that, while the functor $\ell: \Ccal \to \Sh(\Ccal,J_r)$ is full and faithful (since $(\Ccal,J_r)$ is a subcanoniical site) and this functor preserves strict epimorphisms, it \textit{does not} preserve all funneling colimits; a similar argument applies for coalescent categories.

\begin{xmpl}
\label{xmpl:tworel}
Consider the following categories, $\Ccal$ and $\Ccal'$ on the left and right respectively. It is easily checked that they are reductive, with strict epimorphisms identified with two heads.
\[\begin{tikzcd}[row sep = small]
R_1 \ar[dr, shift left] \ar[dr, shift right] & & R_2 \ar[dl, shift left] \ar[dl, shift right] \\
& A \ar[dd, two heads] \ar[dl, phantom] & \\
\cdot & & \\
& B &
\end{tikzcd}
\hspace{1em}
\begin{tikzcd}[row sep = small]
R_1 \ar[dr, shift left] \ar[dr, shift right] \ar[dd, dashed, shift left] \ar[dd, dashed, shift right] & & R_2 \ar[dl, shift left] \ar[dl, shift right] \ar[dd, dashed, shift left] \ar[dd, dashed, shift right] \\
& A \ar[dd, two heads] \ar[dr, two heads] \ar[dl, two heads] & \\
C \ar[dr, two heads] & & D \ar[dl, two heads] \\
& B &
\end{tikzcd}\]
The coequalization is that suggested by the positioning, so that in the first diagram, the coequalizer of the pair coming from $R_1$ is the terminal object $B$, but in the second diagram, $C$ is the coequalizer of the pair coming from $R_2$.

One can calculate directly that the category of supercompact objects in $\Sh(\Ccal,J_r) \simeq \Sh(\Ccal',J_r)$ is equivalent to $\Ccal'$. Indeed, the functor $\ell: \Ccal \to \Sh(\Ccal,J_r)$ does not preserve the coequalizer diagram $R_1 \rightrightarrows A \too B$.
\end{xmpl}

We shall see a further example of a failure of $\ell$ to preserve coequalizers (and hence funneling colimits) in Example \ref{xmpl:TF}. In order to understand which colimits are preserved by $\ell$, we apply criteria derived by Caramello in \cite[Corollary 2.25]{Dense}, which we recall here; we refer the reader to that monograph for the proof.
\begin{lemma}
\label{lem:colimitpres}
Let $(\Ccal,J)$ be a site, $F: \Dcal \to \Ccal$ a diagram and $\lambda = \{\lambda_D: F(D) \to C_0 \mid D \in \Dcal\}$ a cocone under $F$ with vertex $C_0$. Then $\lambda$ is sent by the canonical functor $\ell:\Ccal \to \Sh(\Ccal,J)$ to a colimit cone if and only if:
\begin{enumerate}[(i)]
	\item For any object $C$ and morphism $g:C \to C_0$ in $\Ccal$, there is a $J$-covering family $\{f_i: C_i \to C \mid i \in I\}$ and for each $i \in I$, an object $D_i$ of $\Dcal$ and an arrow $h_i: C_i \to F(D_i)$ such that $\lambda_{D_i} \circ h_i = g \circ f_i$.
	\item For any object $C$ in $\Ccal$ and morphisms $g_1:C \to F(D^1)$, $g_2:C \to F(D^2)$ such that $\lambda_{D^1} \circ g_1 = \lambda_{D^2} \circ g_2$, there is a $J$-covering family $\{f_i: C_i \to C \mid i \in I\}$ such that for each $i \in I$, $g_1 \circ f_i$ and $g_2 \circ f_i$ lie in the same connected component of $(C_i \downarrow F)$.
\end{enumerate}
\end{lemma}

Observe that the first condition can be simplified.

\begin{lemma}
\label{lem:simplify}
Condition (i) of Lemma \ref{lem:colimitpres} is equivalent to the requirement that $\{\lambda_D: F(D) \to C_0 \mid D \in \Dcal\}$ is a $J$-covering family. 
\end{lemma}
\begin{proof}
Consider the case where $g$ is the identity on $C_0$. There must be some $J$-covering family $\{f_i: C_i \to C \mid i \in I\}$ and for each $i \in I$, an object $D_i$ of $\Dcal$ and an arrow $h_i: C_i \to F(D_i)$ such that $\lambda_{D_i} \circ h_i = g \circ f_i$. Since every morphism in this covering family factors through a leg of the colimit cone, the legs of the cone must form a $J$-covering family. Conversely, given any morphism $g:C \to C_0$, since $J$-covering families are required to be stable, pulling back $\lambda$ gives the required $\Tcal$-morphism over $C$ to fulfill condition (i).
\end{proof}

Applying this in the particular case of funneling or multifunneling colimits in a principal or finitely generated site $(\Ccal,J)$, we get the following:
\begin{prop}
\label{prop:colims}
Let $(\Ccal,J_{\Tcal})$ be a principal site and $F: \Dcal \to \Ccal$ a funneling diagram with weakly terminal object $F(D_0)$ and $\lambda = \{\lambda_D: F(D) \to C_0 \mid D \in \Dcal\}$ a cocone under $F$ with vertex $C_0$. Then $\lambda$ is sent by the canonical functor $\ell:\Ccal \to \Sh(\Ccal,J_{\Tcal})$ to a colimit cone if and only if:
\begin{enumerate}[(i)]
	\item $\lambda_{D_0} \in \Tcal$.
	\item For any object $C$ in $\Ccal$ and morphisms $g_1, g_2:C \rightrightarrows F(D_0)$, such that $\lambda_{D_0} \circ g_1 = \lambda_{D_0} \circ g_2$, there is a $\Tcal$-morphism $t: C' \to C$ such that $g_1 \circ t$ and $g_2 \circ t$ lie in the same connected component of $(C' \downarrow F)$.
\end{enumerate}

Similarly, if $(\Ccal,J_{\Tcal'})$ is a finitely generated site and $F: \Dcal \to \Ccal$ a multifunneling diagram with weakly terminal objects $F(D_1),\dotsc,F(D_n)$ and $\lambda = \{\lambda_D: F(D) \to C_0 \mid D \in \Dcal\}$ a cocone under $F$ with vertex $C_0$, then $\lambda$ is sent by $\ell:\Ccal \to \Sh(\Ccal,J_{\Tcal'})$ to a colimit cone if and only if:
\begin{enumerate}[(i)]
	\item $\{\lambda_{D_1},\dotsc,\lambda_{D_n}\} \in \Tcal'$.
	\item For any object $C$ in $\Ccal$ and morphisms $g_1:C \to F(D_k)$ and $g_2: C \to F(D_l)$ with $1 \leq k,l \leq n$, such that $\lambda_{D_k} \circ g_1 = \lambda_{D_l} \circ g_2$, there is a $\Tcal'$-family $\{t_i: C_i \to C \mid 1 \leq i \leq N\}$ such that $g_1 \circ t_i$ and $g_2 \circ t_i$ lie in the same connected component of $(C_i \downarrow F)$.
\end{enumerate}
\end{prop}
\begin{proof}
For (i) in each case, we apply Lemma \ref{lem:simplify} and then the fact that every morphism in the cone factors through $\lambda_{D_0}$ (resp. one of the $\lambda_{D_k}$) to deduce, thanks to stability axiom 4 (resp. axioms 4' and 5') that condition (i) of Lemma \ref{lem:colimitpres} is equivalent to the given statement.

Condition (ii) in each case is a consequence of condition (ii) in Lemma \ref{lem:colimitpres}, having simply taken the special case $D^1 = D^2 = D_0$ (resp. $D^1 = D_k$ and $D^2 = D_l$). Conversely, for arbitrary $g_1:C \to F(D^1)$ and $g_2:C \to F(D^2)$ such that $\lambda_{D^1} \circ g_1 = \lambda_{D^2} \circ g_2$, we may extend this via any of the morphisms $p_1: D^1 \to D_0$ and $p_2: D^2 \to D_0$ in the diagram (resp. $p_1: D^1 \to D_k$ and $p_2: D^2 \to D_l$) so that $\lambda_{D_0} \circ Fp_1 \circ g_1 = \lambda_{D_0} \circ Fp_2 \circ g_2$ (resp. $\lambda_{D_k} \circ Fp_1 \circ g_1 = \lambda_{D_k} \circ Fp_2 \circ g_2$). Given a $\Tcal$-morphism $t:C' \to C$  such that $Fp_1 \circ g_1 \circ t$ and $Fp_2 \circ g_2 \circ t$ are in the same connected component of $(C' \downarrow F)$ (resp. a $\Tcal'$-family $\{t_i: C_i \to C\}$ such that $Fp_1 \circ g_1 \circ t_i$ and $Fp_2 \circ g_2 \circ t_i$ are in the same connected component of $(C_i \downarrow F)$), it is clear that $g_1 \circ t$ and $g_2 \circ t$ (resp. $g_1 \circ t_i$ and $g_2 \circ t_i$ for each $i$) also lie in this same component, as required.
\end{proof}

\begin{crly}
\label{crly:coproduct}
For $(\Ccal,J_{\Tcal'})$ a finitely generated site, a cospan $\lambda_1: X_1 \rightarrow Y \leftarrow X_2: \lambda_2$ is mapped by $\ell: \Ccal \to \Sh(\Ccal,J_{\Tcal'})$ to a coproduct cocone if and only if:
\begin{enumerate}[(i)]
	\item $\{\lambda_1,\lambda_2\} \in \Tcal'$,
	\item Whenever $f_1: C \to X_1$ and $f_2: C \to X_2$ have $\lambda_1 \circ f_1 = \lambda_2 \circ f_2$, the empty family is $\Tcal'$ covering on $C$, and
	\item Whenever $f,f':C \rightrightarrows X_1$ are coequalized by $\lambda_1$, there is a $\Tcal'$-covering family on $C$ consisting of morphisms equalizing $f$ and $f'$ (and similarly for pairs of morphisms into $X_2$).
\end{enumerate}
If $\Tcal'$-covering families are jointly epic, we can replace the last condition by the condition that $\lambda_1$ and $\lambda_2$ must be monic.
\end{crly}
\begin{proof}
This is a direct application of Proposition \ref{prop:colims} for multifunneling colimits in the case of a finite discrete diagram, where two morphisms are in the same connected component of $(C \downarrow F)$ if and only if they are equal (which is impossible if they have distinct codomains).
\end{proof}

Imposing these conditions on all funneling and multifunneling colimit cones, after using Lemma \ref{lem:multi} to decompose multifunneling colimits into finite coproducts and funneling colimits, we arrive at the following definition. 

\begin{dfn}
\label{dfn:redeff}
We say a reductive or coalescent category $\Ccal$ is \textbf{effectual} if for every funneling diagram $F: \Dcal \to \Ccal$ with colimit expressed by $\lambda: F(D_0) \too C_0$, for any object $C$ in $\Ccal$ and morphisms $g_1, g_2:C \rightrightarrows F(D_0)$ such that $\lambda \circ g_1 = \lambda \circ g_2$, there is a strict epimorphism $t: C' \too C$ such that $g_1 \circ t$ and $g_2 \circ t$ lie in the same connected component of $(C' \downarrow F)$.

We say a coalescent category is \textbf{positive} if finite coproducts are disjoint and coproduct inclusions are monomorphisms.
\end{dfn}

With these definitions to hand, we can finally express a definitive correspondence result between supercompactly or compactly generated toposes and their canonical sites from Theorem \ref{thm:canon}.

\begin{thm}
\label{thm:correspondence}
Up to equivalence, there is a one-to-one correspondence between supercompactly generated Grothendieck toposes and essentially small effectual, reductive categories. The correspondence sends a topos to its essentially small category of supercompact objects and a reductive category to the topos of sheaves for the reductive topology on that category.

Similarly, there is an up-to-equivalence correspondence between compactly generated Grothendieck toposes and effectual, positive, coalescent categories.
\end{thm}
\begin{proof}
Passing from a topos to its subcategory of supercompact (resp. compact) objects and back again gives an equivalent topos by Theorem \ref{thm:canon}; the small category must be an effectual reductive (resp. effectual positive coalescent) category, since the subcategory of supercompact objects is closed under funneling colimits, so the inclusion must preserve them (and coproducts in a topos are disjoint).

In the other direction, since the strict (resp. strict finite family) Grothendieck topology is subcanonical, a reductive (resp. coalescent) category is included faithfully as a full subcategory of the corresponding sheaf topos, and all of the representable sheaves are supercompact (resp. compact). The supercompact (resp. compact) objects are funneling (resp. multifunneling) colimits of the representable sheaves in this topos, but these colimits are preserved by $\ell$ by construction when the category is effectual (resp. effectual and positive), so the category of representable sheaves coincides with the category of supercompact (resp. compact) objects, as required.
\end{proof}

\begin{rmk}
Note that an effectual and positive coalescent category is also an effectual reductive category, by Lemma \ref{lem:collred}. However, the corresponding toposes from Theorem \ref{thm:correspondence} are always distinct, since we have seen that the initial object of a topos is never supercompact but always compact.
\end{rmk}

We shall extend this correspondence to an equivalence of $2$-categories in Section \ref{ssec:morsites}. Since geometric morphisms shall come into play at that point, we add here the following extra definition:
\begin{dfn}
\label{dfn:augmented}
A reductive category $\Ccal$ is \textbf{augmented} if it has an initial object. The \textbf{augmented reductive topology} $J_{r+}$ on such a category has covering sieves generated by singleton or empty strictly jointly epic families. The resulting \textbf{augmented reductive site} $(\Ccal,J_{r+})$ is quasi-principal.
\end{dfn}

\subsection{Locally Regular and Coherent Categories}
\label{ssec:regcoh}

As we observed earlier, supercompactness and compactness have been studied in the context of regular and coherent toposes, which are toposes of sheaves on regular and coherent categories respectively, equipped with suitable Grothendieck topologies.  Here we recall the definitions of these classes of categories, as well as some more general classes, for comparison with reductive and coalescent categories. %The correspondence of Theorem \ref{thm:correspondence} is comparable to the correspondence between coherent toposes and effective, positive coherent categories (also known as \textit{pretoposes}).

\begin{dfn}
\label{dfn:regcoh}
Recall that an epimorphism $e$ in a category $\Ccal$ is \textbf{extremal} if whenever $e = m \circ g$ with $m$ a monomorphism, then $m$ is an isomorphism.

A category is \textbf{locally regular} if it is closed under connected finite limits, it has an orthogonal (extremal epi, mono) factorisation system, and every span factors through a jointly monic pair via an extremal epimorphism. Such a category is \textbf{regular} if it also has finite products (equivalently, a terminal object). Clearly, a slice (also called an `over-category') of a locally regular category is regular.

We say a category is \textbf{locally coherent} if it is locally regular and finite unions of subobjects (including the minimal subobject) exist and are stable under pullback. Such a category is \textbf{coherent} if it also has finite products; independently, such a category is called \textbf{positive} if it has disjoint finite coproducts.
\end{dfn}

Note that a coherent category may have finite coproducts without these being disjoint; see the discussion after Definition \ref{dfn:JSL} below.

\begin{lemma}
\label{lem:extremal}
Every extremal epimorphism in a locally regular category is regular.
\end{lemma}
\begin{proof}
We adapt the proof of Johnstone that in a regular category, covers are regular epimorphisms \cite[Proposition A1.3.4]{Ele}; we omit the composition symbol for conciseness in this proof.

Let $f:A \too B$ be an extremal epimorphism in a locally regular category and let $a,b: R \rightrightarrows A$ be its kernel pair. We show that $f$ coequalizes $a$ and $b$.

Suppose $c:A \to C$ has $ca = cb$, and factorize the span $(f,c)$ as an extremal epimorphism followed by a jointly monic pair:
\[\begin{tikzcd}[row sep = small]
& & B \\
A \ar[urr, "f", two heads, bend left] \ar[drr, "c"', bend right] \ar[r, "d", two heads] &
D \ar[ur, "g"] \ar[dr, "h"'] & \\
& & C.
\end{tikzcd}\]
If we can show that $g$ is monic, then extremality of $f$ will force it to be an isomorphism, so that $c = hg^{-1}f$ factors through $f$.

Given $k,l:E \rightrightarrows D$ with $gk = gl$, consider the following diagram composed of pullback squares:
\[\begin{tikzcd}
P \ar[rr, "m", bend left] \ar[dd, "n"', bend right] \ar[d, two heads] \ar[dr, "p", two heads] \ar[r, two heads] &
\cdot \ar[d, two heads] \ar[r] & A \ar[d, two heads, "d"] \\
\cdot \ar[d] \ar[r, two heads] &
E \ar[d, "k"] \ar[r, "l"] & D \\
A \ar[r, two heads, "d"'] & D &
\end{tikzcd}\]
The morphism labelled $p$ is a composite of extremal epimorphisms (by stability) and hence is itself an extremal epimorphism. From this diagram and the preceding assumptions, we have $fm = gdm = gkp = glp = gdn = fn$, whence $(m,n)$ factors through $(a,b)$ via some morphism $q:P \to R$, and we have:
\[hkp = hdm = cm = caq = cbq = cn = hdn = hlp,\]
whence $hk = hl$ by epicness of $p$, but since $(g,h)$ was jointly monic, we have $k = l$, which completes the proof.
\end{proof}

\begin{dfn}
\label{dfn:regtop}
The \textbf{regular topology} on a regular or locally regular category is simply the principal topology generated by the extremal epimorphisms, and similarly the \textbf{coherent topology} on a coherent or locally coherent category is the finitely generated topology generated by the finite jointly extremal epic families.
\end{dfn}

By Lemma \ref{lem:extremal}, we may replace `extremal' with `regular' in the descriptions of the stable classes in this definition, whence we see that these sites are subcanonical. Accordingly we obtain a regular (resp. locally regular, coherent, locally coherent) topos of sheaves on such a site, where here the adjective merely indicates that the topos can be generated by such a site; any topos is automatically a coherent (and hence regular, locally regular and locally coherent) category\footnote{Johnstone explains the reason for the somewhat unfortunate naming convention which we are extending here in \cite[D3.3]{Ele}.}.

By the previous results in this section, any locally regular topos is supercompactly generated, and any locally coherent topos is compactly generated. We are therefore led to wonder when the classes of categories coincide.

\begin{thm}
\label{thm:intersect}
A small category is locally regular with funneling colimits if and only if it is a reductive category with pullbacks.

A small category is locally coherent with multifunneling colimits if and only if it is a coalescent category with pullbacks. The two notions of positivity coincide in this case.

In each case, we can remove the ``locally'' adjective in exchange for adding a terminal object.
\end{thm}
\begin{proof}
In one direction, by Lemma \ref{lem:extremal} we have that in a locally regular category, the classes of extremal, strict and regular epimorphisms all coincide, since any strict epimorphism is extremal, and they form a stable class by assumption, whence a locally regular category with funneling colimits is a reductive category with pullbacks.

Conversely, given a reductive category $\Ccal$ with pullbacks, we must show that $\Ccal$ has equalizers, since a category has connected finite limits if and only if it has both pullbacks and equalizers; the remaining conditions follow from Corollary \ref{crly:orthog} and Lemma \ref{lem:jmonic}, thanks to Lemma \ref{lem:pbstable}. Given a pair of morphisms $h,k: A \rightrightarrows B$ in $\Ccal$, consider their coequalizer $c:B \too C$. Then $\Ccal/C$ is regular, since it has pullbacks and a terminal object (so all finite limits), and it inherits the required factorization system, including that for spans, from $\Ccal$. Therefore there exists an equalizer of $h$ and $k$ as morphisms over $C$, and it is clear that this will also be their equalizer in $\Ccal$.

For the locally coherent case, by considering the strictly epic finite families of \textit{subobjects}, we see that the fact that unions are stable under pullback ensures that strictly epic finite families form a stable class, as required.

Conversely, given a coalescent category $\Ccal$ with pullbacks, the argument above, Corollary \ref{crly:orthog} and Lemma \ref{lem:jmonic} give that $\Ccal$ is locally regular. For finite unions of subobjects, observe that it suffices to consider nullary and binary unions. The former are guaranteed by the strict initial object of a coalescent category, seen in Lemma \ref{lem:collred}. For the latter, observe that the union can be expressed as the pushout (a multifunneling colimit) along the intersection of the two subobjects (the pullback of the monomorphisms defining the subobjects), since any subobject containing the given pair of subobjects forms a cone under this diagram. The fact that strictly epic families are stable under pullback ensures that these unions are too, again thanks to Lemma \ref{lem:pbstable}.

Finally, the two definitions of positivity consist of (existence and) disjointness of finite coproducts, whence these concepts are equivalent.
\end{proof}

The reader may have noticed that we did not include the properties of effectiveness for regular and coherent categories in Definition \ref{dfn:regcoh}: 

\begin{dfn}
\label{dfn:regeff}
A (locally) regular category is \textbf{effective}\footnote{Referred to as \textbf{Barr-exactness} in older texts; we follow Johnstone in our terminology.} if all equivalence relations are kernel pairs.
\end{dfn}

We chose a similar name, `effectual', for the concept appearing in Definition \ref{dfn:redeff} because both effectuality and effectiveness are conditions equivalent to the relevant categories being recoverable from the associated topos. Indeed, a locally regular, effective category $\Ccal$ can be recovered from the topos of sheaves on $\Ccal$ for the regular topology, $\Sh(\Ccal,J_r)$, as the category of \textit{regular objects}, which were defined in Example \ref{xmpl:regular}. Similarly, if $\Ccal$ is locally coherent, positive and effective, it can be recovered from $\Sh(\Ccal,J_c)$ as the category of \textit{coherent objects}, which are defined analogously. As a special case, we recover the familiar correspondences between effective regular categories and regular toposes, or between effective, positive coherent categories (also known as \textit{pretoposes}) and coherent toposes. These results are comparable to Theorem \ref{thm:correspondence}. The concepts of effectuality and effectiveness are directly related:

\begin{prop}
\label{prop:effective}
Let $\Ccal$ be a reductive category with pullbacks. Then if $\Ccal$ is effectual as a reductive category, it is also effective as a regular category.
\end{prop}
\begin{proof}
First suppose that $\Ccal$ is effectual, let $a,b:R \rightrightarrows A$ be an equivalence relation on $A$, and let $\lambda: A \too B$ be its coequalizer. We must show that $(a,b)$ is the kernel pair of $c$. Given $g_1,g_2: C \rightrightarrows A$ with $\lambda \circ g_1 = \lambda \circ g_2$, by effectuality of $\Ccal$ there is a strict epimorphism $t: C' \too C$ such that $g_1 \circ t$ and $g_2 \circ t$ lie in the same connected component of $(C' \downarrow F)$, where $F: \Dcal \to \Ccal$ is the diagram picking out the parallel pair $(a,b)$.

The only form a connecting zigzag can have in $(C' \downarrow F)$ is (omitting the morphisms from $C'$ and any identity morphisms):
\[\begin{tikzcd}
& R \ar[dl, "x_1"'] \ar[dr, "y_1"] & & R \ar[dl, "x_2"'] \ar[dr, "y_2"] & & \cdots \ar[dl, "x_3"'] \ar[dr, "y_{n-1}"] & & R \ar[dl, "x_n"'] \ar[dr, "y_n"] & \\
A && A && A && A && A,
\end{tikzcd}\]
with each $x_i$ and $y_i$ equal to $a$ or $b$. By reflexivity of $R$ we may construct a zigzag consisting of $n > 0$ spans. We show by downward induction that there must always be a zigzag of exactly $1$ span, which corresponds to a factorization of the span $(g_1 \circ t,g_2 \circ t)$ through the relation $(a,b)$.

Clearly if $x_1 = y_1$ we may omit the first zigzag, and similarly for all of the others, so we may assume that $x_i \neq y_i$. By symmetry of $R$, if $(x_i,y_i) = (b,a)$, this factorizes through $(a,b)$, which gives an alternative zigzag of the same length in $(C' \downarrow F)$, so we may assume $x_i = a$ and $y_i = b$. If $n \geq 2$, we may factor through the pullback of $x_2$ along $y_1$, and transitivity of $R$ means that we get a strictly shorter zigzag; iterating this, we reach a zigzag with $n = 1$, as required.

Finally, taking the (regular epimorphism, relation) factorization of $(g_1 \circ t,g_2 \circ t)$, we conclude that the resulting relation (and hence $(g_1,g_2)$) must factor uniquely through $R$, as required.
\end{proof}

As we shall see in Example \ref{xmpl:countable}, the converse of Proposition \ref{prop:effective} fails, which is why we did not employ the same name for these concepts.

\begin{rmk}
While we provided a direct proof of Proposition \ref{prop:effective} for completeness, we could more succinctly have reasoned as follows. When a category is both regular and reductive, the strict and regular topologies on the category coincide. If such a category is effectual, therefore, all of the supercompact objects in its topos of sheaves are regular, and hence it must also be effective, since it is equivalent to the category of regular objects in its associated topos.
\end{rmk}

More generally, one might take an interest in the regular objects in a supercompactly generated topos, or the coherent objects in a compactly generated topos. However, this class of objects need not be stable under pullback in general, and hence may not assemble into a locally regular (resp. locally coherent) category. Nonetheless, by considering the induced Grothendieck topology on this subcategory, we obtain a supercompactly generated subtopos of the original topos. Iterating this process recursively, in the countable limit we obtain a maximal pullback-stable class of regular objects, although the resulting subcategory still may not be a locally regular category in the sense of Definition \ref{dfn:regcoh}, since that definition also required the presence of equalizers. Since it is unclear to us whether this class of objects or the corresponding subtopos have an interesting universal property, and since we lack interesting specific examples of this construction, we terminate our analysis here.

\subsection{Morphisms of Principal Sites}
\label{ssec:morsites}

Morphisms of sites are most easily defined on sites whose underlying category has finite limits. However, there is no reason for this property to hold in a general principal or finitely generated site. We must therefore use the more general definition of morphism of sites, which we quote from Caramello \cite[Definition 3.2]{Dense}.

\begin{dfn}
\label{dfn:morsite}
Let $(\Ccal,J)$ and $(\Dcal,K)$ be sites. Then a functor $F: \Ccal \to \Dcal$ is a \textbf{morphism of sites} if it satisfies the following conditions:
\begin{enumerate}
	\item $F$ sends every $J$-covering family in $\Ccal$ to a $K$-covering family in $\Dcal$.
	\item Every object $D$ of $\Dcal$ admits a $K$-covering family $\{g_i: D_i \to D \mid i \in I\}$ by objects $D_i$ admitting morphisms $h_i: D_i \to F(C'_i)$ to objects in the image of $F$.
	\item For any objects $C_1,C_2$ of $\Ccal$ and any span $(\lambda'_1:D \to F(C_1), \lambda'_2:D \to F(C_2))$ in $\Dcal$, there exists a $K$-covering family $\{g_i: D_i \to D \mid i \in I\}$, a family of spans in $\Ccal$,
	$\{(\lambda_1^i: C'_i \to C_1, \lambda_2^i: C'_i \to C_2) \mid i \in I \},$
	and a family of morphisms in $\Dcal$
	$\{h_i: D_i \to F(C_i)\},$
	such that the following diagram commutes:
	\[\begin{tikzcd}[column sep = small]
		& D_i \ar[dl, "h_i"'] \ar[dr, "g_i"] & \\
		F(C'_i) \ar[d, "F(\lambda_1^i)"'] \ar[drr, "F(\lambda_2^i)", very near start] & &
		D \ar[dll, "\lambda'_1"', very near start] \ar[d, "\lambda'_2"] \\
		F(C_1) & & F(C_2)
	\end{tikzcd}\]
	\item For any pair of arrows $f_1,f_2:C_1 \rightrightarrows C_2$ in $\Ccal$ and any arrow $\lambda':D \to F(C_1)$ of $\Dcal$ satisfying
	$F(f_1) \circ \lambda' = F(f_2) \circ \lambda'$,
	there exist a $K$-covering family in $\Dcal$
	$\{g_i:D_i \to D \mid i \in I\},$
	and a family of morphisms of $\Ccal$
	$\{\lambda^i:C'_i \to C_1 \mid i \in I\},$
	satisfying
	$f_1 \circ \lambda^i = f_2 \circ \lambda^i$
	for all $i \in I$ and of morphisms of $\Dcal$
	$\{h_i:D_i \to F(C'_i) \mid i \in I\},$
	making the following squares commutative:
	\[\begin{tikzcd}
		D_i \ar[d, "h_i"'] \ar[r, "g_i"] & D \ar[d, "g"] \\
		F(C'_i) \ar[r, "F(\lambda^i)"'] & F(C_1)
	\end{tikzcd}\]
\end{enumerate}
\end{dfn}

\begin{rmk}
\label{rmk:findiagcover}
It is not difficult to show by induction on finite diagrams that the last three conditions are equivalent to the following more condensed condition:

Given a finite diagram $A:\Ical \to \Ccal$ and a cone $L'$ over $F \circ A$ in $\Dcal$ with apex $D$, there is a $K$-covering family of morphisms $\{g_i: D_i \to D \mid i \in I\}$ and cones $L_i$ over $A$ in $\Ccal$ with apex $C'_i$ such that $L \circ g_i$ factors through $F(L_i)$ for each $i \in I$, in the sense that there exist morphisms $h_i:D_i \to F(C_i)$ with
\[\begin{tikzcd}
	D_i \ar[d, "h_i"'] \ar[r, "g_i"] & D \ar[d, "\lambda'_j"] \\
	F(C'_i) \ar[r, "F(\lambda^i_j)"'] & F(A(X_j))
\end{tikzcd}\]
for each $i$, where $\lambda'_j$ and $\lambda^i_j$ are the $j$th legs of cones $L'$ and $L_i$ respectively.

Moreover, when the domain site \textit{does} have finite limits, these three conditions reduce to the requirement that $F$ preserves finite limits.
\end{rmk}

A functor is a morphism of sites precisely if $\ell_{\Dcal} \circ F: \Ccal \to \Sh(\Dcal,K)$ is a $J$-continuous flat functor, so that this composite extends along $\ell_{\Ccal}$ to provide the inverse image functor of a geometric morphism $\Sh(\Dcal,K) \to \Sh(\Ccal,J)$.

\begin{crly}
\label{crly:siteprecise}
Suppose $(\Ccal,J)$ and $(\Dcal,K)$ are principal (resp. quasi-principal, finitely generated) sites. Then any morphism of sites $F: (\Ccal,J) \to (\Dcal,K)$ induces a relatively precise (resp. relatively polite, relatively proper) geometric morphism $f: \Sh(\Dcal,K) \to \Sh(\Ccal,J)$.
\end{crly}
\begin{proof}
Since the conditions on the sites ensure that the representables are supercompact (resp. supercompact or initial, compact), and the restriction of the inverse image functor to these is precisely $\ell_{\Dcal} \circ F$, we conclude that $f^*$ preserves these objects, whence $f$ is relatively precise (resp. relatively polite, relatively proper) by Proposition \ref{prop:relpres}.
\end{proof}

\begin{rmk}
\label{rmk:morsite}
More generally, suppose $(\Ccal,J)$ is any small-generated site such that $\Sh(\Ccal,J)$ is supercompactly (resp. compactly) generated and $(\Dcal,K)$ is a principal (resp. finitely generated) site. The geometric morphism induced by a morphism of sites $F: (\Ccal,J) \to (\Dcal,K)$ has inverse image functor sending any funneling (resp. multifunneling) colimit of representables to a colimit of supercompact (resp. compact) objects of the same shape, whence by Lemma \ref{lem:closed2} it must in particular preserve supercompact (resp. compact) objects. As such, we can replace the hypotheses of Corollary \ref{crly:siteprecise} with these weaker conditions if we so choose.
\end{rmk}

Beyond morphisms, it is natural to make the extra step of forming a $2$-category of sites. Indeed, any natural transformation between functors underlying morphisms of sites induces a natural transformation between the inverse images of the corresponding geometric morphisms; for subcanonical sites, this mapping is full and faithful. Thus, for example, any equivalence of sites (an equivalence of categories which respects the Grothendieck topologies) lifts to an equivalence between the corresponding toposes.

Having formed these $2$-categories of sites, we may examine $2$-functors between them. We say that a principal site $(\Ccal,J_{\Tcal})$ is \textit{epimorphic} (resp. \textit{strictly epimorphic}) if $\Tcal$ is contained in the class of epimorphisms (resp. strict epimorphisms). An (effectual) reductive site is an (effectual) reductive category equipped with its reductive topology. We employ the \textit{ad hoc} notation of $\mathbf{EffRedSite}$, $\mathbf{RedSite}$, $\mathbf{StrEpPSite}$, $\mathbf{EpPSite}$ and $\mathbf{PSite}$ for the $2$-categories of effectual reductive sites, reductive sites, strictly epimorphic principal sites, epimorphic principal sites and all principal sites respectively, each endowed with morphisms of sites as $1$-cells and natural transformations as $2$-cells. Clearly, we have forgetful $2$-functors:
\begin{equation}
\label{eq:redsite}
\mathbf{EffRedSite} \to \mathbf{RedSite} \to \mathbf{StrEpPSite} \to \mathbf{EpPSite} \to \mathbf{PSite}.
\end{equation}

We apply analogous terminology and notation for the comparable kinds of finitely generated sites and coalescent sites. For example, we write $\mathbf{EPCoalSite}$, $\mathbf{EffCoalSite}$, $\mathbf{PosCoalSite}$ and $\mathbf{FGSite}$ for the $2$-categories of effectual positive coalescent sites, effectual coalescent sites, positive coalescent sites and finitely generated sites, respectively. There results an analogous chain of $2$-functors:
\begin{equation}
\label{eq:coalsite}
\begin{tikzcd}[column sep = 5pt]
& \mathbf{EffCoalSite} \ar[dr] & & & & \\
\mathbf{EPCoalSite} \ar[ur] \ar[dr] & & \mathbf{CoalSite} \ar[r] & \mathbf{StrEpFGSite} \ar[r] & \mathbf{EpFGSite} \ar[r] & \mathbf{FGSite}.\\
& \mathbf{PosCoalSite} \ar[ur] & & & &
\end{tikzcd}
\end{equation}

Consolidating the results of Section \ref{ssec:representable}, we find that several of these forgetful functors have adjoints. 

\begin{crly}
\label{crly:reflect}
Let $(\Ccal,J_{\Tcal})$ be a principal site, ${\sim}$ the canonical congruence of Proposition \ref{prop:congruence}, $\ell(\Ccal)$ the full subcategory of $\Sh(\Ccal,J_{\Tcal})$ on the representable sheaves and $\Ccal_s$ the (essentially small) category of supercompact objects in that topos. Then the canonical functors underly morphisms of sites:
\[\begin{tikzcd}
(\Ccal_s{,}J_r)  &
(\ell(\Ccal){,}J_{can}|_{\ell(\Ccal)}) \ar[l] &
(\Ccal/{\sim}{,}J_{\Tcal/{\sim}}) \ar[l] &
(\Ccal{,}J_{\Tcal}) \ar[l]
\end{tikzcd}\]
which are the units of reflections to the forgetful functors
\[\mathbf{EffRedSite} \to \mathbf{StrEpPSite} \to \mathbf{EpPSite} \to \mathbf{PSite}\]
found in Diagram \eqref{eq:redsite}.

Similarly, if $(\Ccal,J_{\Tcal'})$ is a finitely generated site, then with analogous notation, we have morphisms of sites:
\[\begin{tikzcd}
(\Ccal_c{,}J_c) &
(\ell(\Ccal){,}J_{can}|_{\ell(\Ccal)}) \ar[l] &
(\Ccal/{\sim}{,}J_{\Tcal'/{\sim}}) \ar[l] &
(\Ccal{,}J_{\Tcal'}) \ar[l]
\end{tikzcd}\]
which are units for the reflections of the forgetful functors
\[\mathbf{EPCoalSite} \to \mathbf{StrEpFGSite} \to \mathbf{EpFGSite} \to \mathbf{FGSite}\]
appearing in Diagram \eqref{eq:coalsite}.

All of these units induce equivalences at the level of the associated toposes.
\end{crly}
\begin{proof}
We omit the straightforward checks that these are indeed morphisms of sites. The universality of the middle and right hand units has been discussed in and beneath Proposition \ref{prop:congruence} and Corollary \ref{crly:full}; it remains only to show that the final morphism $(\ell(\Ccal){,}J_{can}|_{\ell(\Ccal)}) \to (\Ccal_s{,}J_r)$ is universal.

Let $\Ccal'$ be an effectual, reductive category. Then a morphism of sites $F: (\Ccal,J_{\Tcal}) \to (\Ccal',J_r)$ corresponds to a geometric morphism $\Sh(\Ccal',J_r) \to \Sh(\Ccal,J_{\Tcal})$ whose inverse image functor restricts to $F$ on the representable sheaves, and so sends these to supercompact objects in $\Sh(\Ccal',J_r)$. Since a quotient of a supercompact object is supercompact and inverse image functors preserve quotients, $F$ extends uniquely (up to isomorphism) to a morphism of sites $(\Ccal_s,J_r) \to (\Ccal',J_r)$ inducing the same geometric morphism, as required.

As usual, the proof for finitely generated sites is analogous.
\end{proof}

The morphisms appearing in Corollary \ref{crly:reflect} allow us to give another characterisation of reductive and coalescent categories.

\begin{lemma}
\label{lem:redsect}
Let $(\Ccal,J_{\Tcal})$ be a strictly epimorphic principal site in which $\Tcal$ is the class of all strict epimorphisms of $\Ccal$. Then, assuming the axiom of choice, $\Ccal$ is reductive if and only if the (underlying functor of the) composed unit morphism $(\Ccal,J_{\Tcal}) \to (\Ccal_s,J_r)$ has a left adjoint.

Similarly, a strictly epimorphic finitely generated site $(\Ccal,J_{\Tcal'})$ where $\Tcal'$ consists of the strict jointly epic families has $\Ccal$ coalescent if and only if the morphism of sites $(\Ccal,J_{\Tcal}) \to (\Ccal_c,J_c)$ has a left adjoint.
\end{lemma}
\begin{proof}
If $\Ccal$ is reductive and $\Tcal$ is its class of strict epimorphisms, consider a supercompact object $C$ in $\Sh(\Ccal,J_{\Tcal}) = \Sh(\Ccal,J_r)$. $C$ is a quotient of some representable $\ell(C_0)$, and so is the colimit of some funneling diagram in $\ell(\Ccal)$ with weakly terminal object $\ell(C_0)$. Lifting this to a funneling diagram in $\Ccal$, call its colimit $L(C)$. There is a universal morphism $\eta: C \to \ell(L(C))$, since the image of the strict epimorphism $\ell(C_0 \too L(C))$ forms a cone under the original funnel in $\Sh(\Ccal,J_{\Tcal})$. This $\eta$ is the universal morphism from $C$ to a representable object, since given $C \to \ell(D)$, we have that the composite $\ell(C_0) \to C \to \ell(D)$ is a morphism in the image of $\ell$ (since $\ell$ is full and faithful on a strict principal site) forming a cone under the same funnel, so there is a factoring morphism $\ell(L(C)) \to \ell(D)$, as required. This universality means that $\ell(L(C))$ is well-defined up to isomorphism, and we can use choice to select a representative for each $C$; the universality then ensures that $L$ is functorial, and is a left adjoint to the inclusion $\Ccal \to \Ccal_s$, as required.

Conversely, suppose we have a left adjoint functor $L: \Ccal_s \to \Ccal$. Given a funnel in $\Ccal$, consider its colimit in $\Ccal_s$; this is preserved by $L$, so the colimit exists in $\Ccal$, which is enough to make $\Ccal$ a reductive category.

The argument for coalescent categories is analogous, passing via a finite coproduct to define $L(C)$ in the first part.
\end{proof}

In other words, any reductive category is a coreflective subcategory of an effectual reductive category, and similarly any coalescent category is a coreflective subcategory of an effectual positive coalescent category.

\begin{thm}
\label{thm:2equiv}
Let $\mathbf{SG}\TOP_{\mathrm{rel prec}}$ be the $2$-category of supercompactly generated Grothendieck toposes, relatively precise geometric morphisms and all geometric transformations. Then the object mapping:
\begin{align*}
\mathbf{PSite} &\to \mathbf{SC}\TOP_{\mathrm{rel prec}} \\
(\Ccal,J_{\Tcal}) &\mapsto \Sh(\Ccal,J_{\Tcal}).
\end{align*}
extends to a $2$-functor between these $2$-categories. This functor is faithful on $2$-cells if we restrict the domain to $\mathbf{EpPSite}$. It is full and faithful on $2$-cells and faithful (up to isomorphism) on $1$-cells if we restrict the domain to $\mathbf{StrEpPSite}$. Finally, it is a $2$-categorical equivalence if we restrict the domain to $\mathbf{EffRedSite}$.

Analogously, letting $\mathbf{CG}\TOP_{\mathrm{rel prop}}$ be the $2$-category of compactly-generated Grothendieck toposes and relatively proper geometric morphisms, there is a $2$-functor whose effect on objects is:
\begin{align*}
\mathbf{FGSite} &\to \mathbf{CG}\TOP_{\mathrm{rel prop}} \\
(\Ccal,J_{\Tcal'}) &\mapsto \Sh(\Ccal,J_{\Tcal'}).
\end{align*}
This restricts to an equivalence between $\mathbf{CG}\TOP_{\mathrm{rel prop}}$ and the $2$-category $\mathbf{EPCoalSite}$.

Finally, there is an equivalence between the $2$-category $\mathbf{SG}\TOP_{\mathrm{rel pol}}$ of supercompactly generated Grothendieck toposes, relatively polite geometric morphisms and all geometric transformations, and the $2$-category $\mathbf{EffRed}^{+}\mathbf{Site}$ of augmented reductive sites of Definition \ref{dfn:augmented}.
\end{thm}
\begin{proof}
The claim of $2$-functoriality is fulfilled thanks to Corollary \ref{crly:siteprecise}.

Faithfulness when we restrict to $\mathbf{EpPSite}$ is by virtue of the observations after Proposition \ref{prop:congruence} that $\ell$ is faithful (on $\Tcal$-arches and hence) on morphisms coming from the site.

Fullness on $2$-cells and faithfulness on $1$-cells when we restrict further to $\mathbf{StrEpPSite}$ is more directly derived from full faithfulness of $\ell$ for such subcanonical sites, and the fact that natural transformations between inverse image functors are determined by their components at the representables. Any such natural transformation (including a natural isomorphism) restricts along $\ell$ to a natural transformation between the underlying morphisms of sites.

The equivalence when we restrict to $\mathbf{EffRedSite}$ is thanks to Theorem \ref{thm:correspondence}. Indeed, since an effectual reductive category is equivalent to the category of supercompact objects in the corresponding topos, any relatively precise morphism between such toposes restricts to a morphism of sites between the underlying effective reductive sites.

As ever, the argument for finitely generated sites and quasi-principal sites is analogous.
\end{proof}

%\subsection{Comorphisms of Principal Sites}
%\label{ssec:comorsite}

We would be remiss not to mention comorphisms of sites, which \textit{in principle} should provide an alternative choice for extending the correspondence of Theorem \ref{thm:correspondence} to a covariant duality.

\begin{definition}
\label{dfn:comor}
Given sites $(\Ccal,J)$ and $(\Dcal,K)$, a \textbf{comorphism of sites} $F: (\Ccal,J) \to (\Dcal,K)$ is a functor $F:\Ccal \to \Dcal$ which has the \textbf{cover-lifting property}, that for any object $C$ of $\Ccal$ and $K$-covering sieve $S$ on $F(C)$, there exists a $J$-covering sieve $R$ on $C$ with $F(R) \subseteq S$.
\end{definition}

When the topology on the sites is trivial, any functor is a comorphism of sites, and we can characterise the morphisms induced by comorphisms of sites as the essential geometric morphisms. If we restrict to sites on idempotent complete categories, this results in the familiar $2$-equivalence between the $2$-category $\Cat_{ic}$ of small idempotent complete categories, functors and natural transformations, and the $2$-category $\TOP_{ess}$ of presheaf toposes, essential geometric morphisms and geometric transformations, up to reversal of $2$-cells:
\[\Cat_{ic} \simeq \TOP_{ess}\co.\]
We see a special case of this in Section \ref{ssec:localic}.

However, this special case belies the difficulty of characterising geometric morphisms coming from comorphisms of sites, even when the sites are subcanonical and their Grothendieck topologies have simple descriptions. Without such a characterisation, it is difficult to produce a meaningful $2$-equivalence statement. We therefore leave this endeavour as a challenge for the future.

\subsection{Points of Supercompactly and Compactly Generated Toposes}
\label{ssec:Deligne}

We shall study supercompactly and compactly generated localic toposes in Section \ref{sec:xmpls}. Even before characterising these, we can make some observations about their points.

\begin{lemma}
\label{lem:point}
Every point of a localic topos is relatively polite. A localic topos has a relatively precise point if and only if it is totally connected.
\end{lemma}
\begin{proof}
We use Lemma \ref{lem:relprop}. Since $\Set$ is atomic, given a point $p:\Set \to \Ecal$ and an object $X$ of $\Ecal$, it suffices to consider the minimal inhabited covering of $p^*(X)$ by singletons $\{ x_i : 1 \hookrightarrow p^*(X) \}$ or, if $p^*(X)$ is empty, by the identity $\{ 0 \hookrightarrow p^*(X) \}$. The characterisation of localic toposes as being generated by subterminal objects means we have a covering of $X$ by subterminals $\{ y_j: X_j \to X \}$, and since $f^*(X_j)$ is necessarily subterminal for every $j$, these factorize through this minimal covering, as required.

Since the inverse image functor of a relatively precise point must reflect $0$, every non-initial open $U$ of $\Ecal$ must have $p^*(U) = 1$, whence the non-initial opens form a completely prime filter, which is equivalent to $\Ecal$ being totally connected. 
\end{proof}

We can say more about the existence of relatively precise points in general.
\begin{lemma}
\label{lem:cofiltered}
Let $(\Ccal,J)$ be a site such that every $J$-covering sieve is inhabited. Consider the $1$ object category equipped with its canonical (equivalently, trivial) topology, $(1,J_{can})$, whose topos of sheaves is equivalent to $\Set$. The unique functor $\Ccal \to 1$ underlies a morphism of sites $(\Ccal,J) \to (1,J_{can})$ if and only if $\Ccal\op$ is filtered, in which case the corresponding point of $\Sh(\Ccal,J)$ is relatively precise point. Moreover, the direct image functor of this point is the inverse image functor of the global sections morphism of $\Sh(\Ccal,J)$, so this topos is totally connected.
\end{lemma}
\begin{proof}
Checking that the given conditions on $J$ and $\Ccal$ are necessary and sufficient to produce the claimed morphism of sites is straightforward: condition 1 is ensured by inhabitedness of $J$-covering sieves, while the remaining conditions translate into $\Ccal\op$ being filtered. Given that $\Ccal\op$ is filtered, all inhabited sieves on $\Ccal$ are necessarily connected. It is shown by Johnstone in \cite[Example C3.6.17(c)]{Ele} that if $(\Ccal,J)$ is a cofiltered site in which all $J$-covering sieves are connected, then $\Sh(\Ccal,J)$ is totally connected, and in particular locally connected. Since all of the representable objects are connected (which is to say they are mapped by the left adjoint to the inverse image functor of the global sections geometric morphism to $1$), we see that this coincides up to isomorphism with the inverse image functor of the point induced by the morphism of sites, as required. The point $\Set \to \Sh(\Ccal,J)$ therefore has direct image preserving all colimits (since it has an extra right adjoint) whence it is certainly precise and hence relatively precise.
%Given any object $X$ of $\Sh(\Ccal,J)$, just as in Lemma \ref{lem:point}, it suffices to consider the cover of $p^*(X)$ by singletons. Since $X$ is covered by representables which are sent to singletons, $p$ is relatively precise, by Lemma \ref{lem:relprop}.
\end{proof}

\begin{prop}
\label{prop:precpoint}
Given a principal site $(\Ccal,J_{\Tcal})$, the topos $\Sh(\Ccal,J_{\Tcal})$ has a relatively precise point if and only if $\Ccal\op$ is filtered, if and only if $\Sh(\Ccal,J_{\Tcal})$ is totally connected.
\end{prop}
\begin{proof}
By Theorem \ref{thm:2equiv}, any relatively precise point of $\Sh(\Ccal,J_{\Tcal})$ comes from a morphism of sites $(\Ccal_s,J_s) \to (1,J_{can})$, where $\Ccal_s$ is the usual subcategory in $\Sh(\Ccal,J_{\Tcal})$. Composing with the canonical morphism of sites $(\Ccal,J_{\Tcal}) \to (\Ccal_s,J_s)$ from Corollary \ref{crly:reflect}, we conclude that such a morphism exists if and only if one exists $(\Ccal,J_{\Tcal}) \to (1,J_{can})$. Applying Lemma \ref{lem:cofiltered} gives the result.
\end{proof}
Note that since the conditions on $\Ccal$ are independent of the choice of principal topology it is equipped with, we may deduce that any relatively precise subtopos of $\Sh(\Ccal,J_{\Tcal})$ is totally connected when the hypotheses hold.

\begin{rmk}
Clearly any totally connected topos has a unique relatively precise point. We do not know of a counterexample of the reverse implication more generally.
\end{rmk}

We immediately recover the following result, which although deducible from \cite[Example C3.6.17(c)]{Ele}, does not seem to have been recorded explicitly anywhere that we know of.
\begin{crly}
Any regular topos is totally connected. 
\end{crly}
\begin{proof}
It suffices to observe that if $\Ccal$ has finite limits, $\Ccal\op$ is filtered, so Proposition \ref{prop:precpoint} applies.
\end{proof}
In particular, by \cite[Theorem C3.6.16(iv)]{Ele}, the category of models of a regular theory in any Grothendieck topos has a terminal object (easily described as the model in which every sort is interpreted as the terminal object). The logic of supercompactly generated toposes more generally is a subject for a future paper.

Recall that Deligne's theorem states that every (locally) coherent topos has enough points. However, the original proof of that theorem, \cite{SGA4Coh}, involves nothing beyond the presentation of such a topos as the topos of sheaves on a finitely-generated site with finite limits. Since the proof only involves taking pullbacks in a single place, we find that we can easily remove the assumption that the site has finite limits thanks to the stability axioms of Definition \ref{dfn:stable2}. For completeness, we do this extension here, although we must note that this is an almost verbatim reproduction of the version of Deligne's proof presented by Johnstone in \cite[\S7.4]{TT}.

Suppose we are given a finitely-generated site $(\Ccal,J_{\Tcal'})$. Given any directed poset $P$ and a functor $F: P \to \Ccal\op$, we can construct a point of $[\Ccal\op,\Set]$ whose inverse image functor is
\[ X \mapsto \colim_{p \in P} X(F(p)),\]
since directed colimits commute with finite limits and all colimits. Following Johnstone, we call such an $F$ a \textit{pseudo-point} of $\Ccal$; we shall write $F'$ to denote the induced inverse image functor described above.

\begin{lemma}
\label{lem:Del1}
Let $F:P \to \Ccal\op$ be a pseudo-point of $\Ccal$, $X$ a $J_{\Tcal'}$-sheaf on $\Ccal$ and $x \neq y \in F'(X)$. Let $\mathfrak{f} = \{f_i: V_i \to V \mid i = 1,\dotsc,n\} \in \Tcal'$ and $v \in F'(\yon(V))$, where $\yon(V)$ is the representable presheaf corresponding to $V$. Then there exists a refinement $G:Q \to \Ccal\op$ (that is, there exists an inclusion of directed posets $i:P \to Q$ with $F = G \circ i$) such that:
\begin{enumerate}[(a)]
	\item The images of $x$ and $y$ under the natural map $F'(X) \to G'(X)$ are distinct, and
	\item The image of $v$ in $G'(\yon(V))$ is in $\bigcup_{i = 1}^n \im(G'(\yon(V_i)) \to G'(\yon(V)))$.
\end{enumerate}
\end{lemma}
\begin{proof}
By definition, $v$ corresponds to some morphism $F(p) \to V$ with $p \in P$. Given $p' \geq p$, applying stability of $\Tcal'$ along the composite $F(p') \to F(p) \to V$, we obtain\footnote{We assume for the purposes of this proof that stability axiom 3' provides a choice of $\Tcal'$-family completing the relevant squares.} a finite family $\mathfrak{f}_{p'}' \in \Tcal'$ on $F(p')$ whose members factor through the members of $\mathfrak{f}$. As such, since $X$ is a $J_{\Tcal}$-sheaf, we have that the morphism $X(F(p')) \to \prod_{i' = 1}^{n'} X(U_{p',i'})$, where $U_{p',i'}$ are the domains of the morphisms in $\mathfrak{f}_{p'}'$, is a monomorphism.

But filtered colimits preserve finite products and monomorphisms, so $F'(X) = \colim_{p' \in p/P} X(F(p')) \hookrightarrow \prod_{i' = 1}^{n'} \colim_{p' \in p/P}X(U_{p',i'})$. In particular, we can find some index $i_{p'}'$ such that $x,y$ have distinct images in $\colim_{p' \in p/P}X(U_{p',i_{p'}'})$. Now let $Q$ be the poset $P \times \{0\} \sqcup p/P \times \{1\}$ with ordering $(p,m) \leq (p',m')$ if and only if $p \leq p'$ and $m \leq m'$. Then $Q$ is clearly filtered, and we may define the pseudo-point $G: Q \to \Ccal\op$ via the assignment
\[(p,0) \mapsto F(p); \, (p',1) \mapsto U_{p',i_{p'}'}.\]
Then for any presheaf $Y$ on $\Ccal$, we have $G'(Y) = \colim_{p' \in p/P} Y(U_{p',i'})$, whence the pseudo-point has the required properties.
\end{proof}

\begin{lemma}
\label{lem:Del2}
Let $F,X,x,y$ be as in Lemma \ref{lem:Del1}. Then there exists a refinement $G$ of $F$ such that:
\begin{enumerate}[(a)]
	\item The images of $x$ and $y$ under the natural map $F'(X) \to G'(X)$ are distinct, and
	\item For each object $V$ of $\Ccal$, each $\Tcal'$-family $\mathfrak{f}$ over $V$ and each $v \in F'(\yon(V))$, the image of $v$ in $G'(\yon(V))$ is in the image of $G'(S(\mathfrak{f})) \hookrightarrow G'(\yon(V))$, where $S(\mathfrak{f})$ is the sieve on $V$ generated by $\mathfrak{f}$.
\end{enumerate}
\end{lemma}
\begin{proof}
Let $Z$ be the set of triples $(\mathfrak{f},V,v)$, where $\mathfrak{f} \in \Tcal'$ over the object $V$ and $v \in F'(\yon(V))$. Well-ordering $Z$, we may index its members by some ordinal $\alpha$. We now use transfinite induction to define a sequence of pseudo-points $\{ G_{\beta}: Q_{\beta} \to \Ccal \mid \beta \leq \alpha \}$ as follows:

$G_{0} = F$.

If $\beta = \gamma +1$, $G_\beta$ is obtained from $G_{\gamma}$ by applying the construction of Lemma \ref{lem:Del1} to $\mathfrak{f}$ and the image of $v_{\gamma}$ in $G_{\gamma}'(V_{\gamma})$.

If $\beta$ is a limit ordinal, then $G_{\beta}$ is the unique common refinement of the $G_{\gamma}$ with $\gamma < \beta$, whose underlying filtered poset is the colimit of the $Q_{\gamma}$.

It is then easily seen that the images of $x$ and $y$ in $G_{\beta}'(X)$ are distinct for all $\beta$ and that if $\gamma < \beta$ then $G_{\beta}$ is a refinement of $G_{\gamma}$ such that the image of $v_{\gamma}$ in $G_{\beta}'(\yon(V_{\gamma}))$ is in the image of $G_{\beta}'(S(\mathfrak{f}_{\gamma}))$. So $G_{\alpha}$ is the required refinement of $F$.
\end{proof}

\begin{lemma}
\label{lem:Del3}
Let $F,X,x,y$ be as in Lemma \ref{lem:Del1}. Then there exists a refinement $G$ of $F$ such that the induced point restricts to a point of $\Sh(\Ccal,J_{\Tcal'})$ and such that the images of $x$ and $y$ in $G'(X)$ are distinct.
\end{lemma}
\begin{proof}
Define a sequence of pseudopoints $F_n$ by setting $F_0 = F$ and defining $F_{n+1}$ to be the refinement of $F_n$ constructed in Lemma \ref{lem:Del3}. Define $G$ to be the colimit of the $F_n$; then we clearly have $G'(Y) = \colim_{n} F_{n}'(Y)$ for any presheaf $Y$ on $\Ccal$. Hence $x$ and $y$ have distinct images in $G'(X)$, and if $S(\mathfrak{f}) \hookrightarrow \yon(V)$ is a $J_{\Tcal'}$-covering sieve, then each element of $G'(\yon(V))$ derives from an element of $F_{n}'(\yon(V))$ for some $n$, and hence from an element of $F_{n+1}'(S(\mathfrak{f}))$. So $G'$ defines a point of $[\Ccal\op,\Set]$ having the desired properties which factors through $\Sh(\Ccal,J_{\Tcal'})$, as required.
\end{proof}

\begin{thm}
\label{thm:Del}
Employing the axiom of choice, any compactly generated topos has enough points.
\end{thm}
\begin{proof}
Expressing $\Ecal$ as a topos of the form $\Sh(\Ccal,J_{\Tcal'})$ of sheaves on a finitely generated site (always possible, since we may take the canonical site on the compact objects), we apply Lemma \ref{lem:Del3} to each member of the jointly epimorphic set of points of the presheaf topos $[\Ccal\op,\Set]$ to obtain a jointly epimorphic set of points of $\Ecal$.
\end{proof}

Following a comment of Johnstone in \cite[Remark D1.5.11]{Ele}, we prove a constructive version of this result in the special case that we are given the data required to avoid the axiom of choice. 

\begin{schl}
Suppose that $(\Ccal,J_{\Tcal'})$ is a finitely generated site such that the morphisms of $\Ccal$ admit a well-ordering. Then the corresponding toposes have enough points, constructively.
\end{schl}
\begin{proof}
The source of non-constructiveness is the use of the well-ordering principle in the proof of Lemma \ref{lem:Del2}. With the given well-orderings on the site, however, we may construct such a well-ordering directly.

First, observe that in the construction of Lemma \ref{lem:Del1}, if we have a well-ordering of the elements of $P$, then the elements of $Q$, which consist of the set $P \times \{0\} \sqcup p/P \times \{1\}$ can be endowed with a well-ordering where every element of the former set comes before every element of the latter\footnote{Observe that this well-ordering need not be compatible with the partial ordering defined in that Lemma.}; this similarly extends to the limit ordinals in the transfinite induction of Lemma \ref{lem:Del2}. Thus we may assume throughout, without loss of generality, that $P$ is endowed with such a well-ordering. As such, given a pseudopoint $F:P \to \Ccal\op$, each $v \in F'(\yon(V))$ has a canonical minimal representative in the diagram, consisting of the morphism $v_0: F(p) \to V$ such that $p$ is the minimal element in the well-ordering on $P$ carrying a representative, and $v_0$ is the minimal representing morphism in this hom-set according to the well-ordering on $\Ccal$. Taking the lexicographic ordering on the minimal representatives $(p,v_0:F(p) \to V)$, we obtain a well-ordering on $F'(\yon(V))$.

We can also well-order the finite families in $\Tcal'$ using the well-ordering on $\Ccal$, ordering first by the number of members, and then by the lexicographical ordering on the family's members.

As such, the set $Z$ of triples $(\mathfrak{f},V,v)$ where $\mathfrak{f} \in \Tcal'$ has codomain $V$ and $v \in F'(\yon(V))$, acquires a well-ordering, as required.
\end{proof}

\subsection{Examples of Reductive Categories}
\label{ssec:xmpl}

In this subsection and the next section, we present examples of reductive and coalescent categories, as well as principal and finitely generated sites, in order to address some hypotheses about relationships between the concepts presented in this paper.

\begin{xmpl}
\label{xmpl:nonreg}
There exist reductive categories without equalizers, products or pullbacks (or even pullbacks of monomorphisms), so which in particular are not locally regular. Indeed, simplifying Example \ref{xmpl:tworel}, consider the category of presheaves on $f,g: A \rightrightarrows B$. The subcategory of supercompact objects in this topos is simply the coequalizer diagram $A \rightrightarrows B \too C$, so that in particular the pair of monomorphisms $f,g$ has neither an equalizer nor a pullback, and the product $B \times B$ does not exist. We obtain a similar example from any finite category containing a parallel pair of morphisms lacking an equalizer. 
\end{xmpl}

\begin{xmpl}
\label{xmpl:simple}
Any discrete category (a category with no non-identity morphisms) with more than one object is a reductive and locally regular category which is not regular.

The free finite cocompletion of a discrete category $\Ccal$ is a coalescent category, since it is equivalent to $[\Ccal\op,\FinSet]$, where $\FinSet$ is the category of finite sets\footnote{This expression for the free finite cocompletion applies if and only if $\Ccal$ has finite hom-sets, since this is necessary and sufficient for the representable presheaves to lie in $[\Ccal\op,\FinSet]$. This is trivially the case when $\Ccal$ is discrete.}, and this is a coalescent category, whence $[\Ccal\op,\FinSet]$ has funneling colimits computed pointwise. $\FinSet$ is in some sense the simplest possible coalescent category, since by inspection, $\Sh(\FinSet,J_c) \simeq \Set$. Since $\FinSet$ has pullbacks, these free finite cocompletions are coherent categories. Dually, the free finite limit completion of a discrete category, $[\Ccal, \FinSet]\op$, is a coalescent category with pullbacks.
\end{xmpl}

\begin{xmpl}
We may also consider the category $\FinSet_+$ of \textit{inhabited} finite sets; since a funneling colimit of inhabited finite sets is inhabited and finite, and taking the pullback in $\FinSet$ of an epimorphism in $\FinSet_+$ along a morphism with inhabited domain gives another epimorphism with inhabited domain, we have that $\FinSet_+$ is another example of a reductive category without pullbacks (since the two inclusions $1 \rightrightarrows 2$ `should' have empty intersection).

By reintroducing the empty set and declaring that the empty sieve is covering over it, we obtain an augmented reductive site with underlying category $\FinSet$ having an equivalent topos of sheaves; since the coalescent topology on $\FinSet$ is a refinement of the augmented reductive one, we see that we have a (relatively proper) inclusion of toposes:
\[\Set \simeq \Sh(\FinSet,J_c) \hookrightarrow \Sh(\FinSet_+,J_r), \]
which is not an equivalence since the sheaves represented by the finite sets of cardinality at least $2$ are supercompact in the latter topos but merely compact in the former.
\end{xmpl}

\begin{xmpl}
Expanding on the dual construction, we have that $\Sh(\FinSet\op,J_c)$ embeds into the \textit{classifying topos for the theory of objects}, $[\FinSet,\Set]$. We mention this as a counterexample to an extension of Lemma \ref{lem:point} to the idea that a point of any Grothendieck topos is polite: the points of $[\FinSet,\Set]$ correspond to sets (objects of $\Set$), and the correspondence sends a geometric morphism to the set which is the image of the representable functor $\yon(1)$. But $\yon(1)$ is supercompact, so a point corresponding to any set with more than one element fails to be polite.
\end{xmpl}

\begin{xmpl}
\label{xmpl:simplex}
For yet another related example, consider the simplex category $\Delta$, whose objects are \textit{inhabited} finite ordinals,
\[ [n] = \{0, \dotsc, n-1\}, \, \text{ } \, n \geq 1 \]
and whose morphisms are the order-preserving maps between these. Clearly this does not have pullbacks (since the intersection of the two inclusions $[1] \rightrightarrows [2]$ would be the empty ordinal which is not an object of $\Delta$).

$\Delta$ has funneling colimits: given any collection of morphisms into the object $[n]$ of $\Delta$, their colimit is the quotient of $[n]$ identifying $f(x),f(x)+1,\cdots,g(x)$ (or $g(x),g(x)+1,\cdots,f(x)$) for each parallel pair $f,g:[n']\rightrightarrows [n]$ in the diagram and each $x \in [n']$. Moreover, each epimorphism $g:[n] \too [m]$ is split (so in particular is strict) by the monomorphism $\mathrm{min}(g^{-1}):[m]\rightrightarrows [n]$, say.

In particular, by Remark \ref{rmk:split} the collection of strict epimorphisms is stable, which makes $\Delta$ a reductive category with $\Sh(\Delta,J_r) = [\Delta\op,\Set]$, the topos of \textit{simplicial sets}. While this tells us little about the theory of strict linear intervals classified by $[\Delta\op,\Set]$ (referred to as `orders' in \cite{MLM}[\S8.VIII]), it does recover a non-trivial fact about the topos of simplicial sets: every quotient of a representable simplicial set is also representable, so that every simplicial set is a union of its representable subsets.
\end{xmpl}

\begin{xmpl}
\label{xmpl:nonpresh}
To contrast Examples \ref{xmpl:simple} and \ref{xmpl:simplex}, we recall an example of a supercompactly generated topos which is not equivalent to a presheaf topos; this should be contrasted with Proposition \ref{prop:localic}, below.

Consider the Schanuel topos, $\Sh(\FinSet\op_{\mathrm{mono}},J_{at})$. We see that $(\FinSet\op_{\mathrm{mono}},J_{at})$ is an atomic site with pullbacks, but moreover it is a reductive and regular site, since all of the morphisms in the category are regular epimorphisms which are stable under pullback.

We know of two ways to show that this topos is not a presheaf topos. The first is to show that the site is effectual, which explicitly involves providing an algorithm which, given a cofunnel $F$ in $\FinSet_{\mathrm{mono}}$ with weakly initial object $F(D_0)$ and $x,y:F(D_0) \rightrightarrows C$ equalized by its limit, constructs an inclusion $i: C \hookrightarrow C'$ and a connecting zigzag between $i \circ x$ and $i \circ y$ in $(F \downarrow C')$.

It follows that the supercompact objects (equivalently, atoms) in the Schanuel topos are precisely the representable sheaves coming from $\FinSet\op_{\mathrm{mono}}$. However, if $I$ is a finite set, $A$ is any set of cardinality larger than that of $I$, and we have inclusions from a further finite set $B$ into both $A$ and $I$, there can be no monomorphism completing the triangle,
\begin{equation*}
\label{eq:injective}
\begin{tikzcd}
& A \ar[dl, dashed] \\
I & B \ar[u, hook] \ar[l, hook],
\end{tikzcd}
\end{equation*}
whence $I$ is not injective in $\FinSet_{\mathrm{mono}}$ and hence is not projective in $\FinSet\op_{\mathrm{mono}}$. It follows that no object of $\FinSet\op_{\mathrm{mono}}$ is projective, whence the Schanuel topos has no indecomposable projective objects, but is non-degenerate, and so is not a presheaf topos (the category of representable presheaves in a presheaf topos being recoverable up to idempotent completion as the projective indecomposable objects, which in particular must be supercompact).

An alternative proof, which we thank Olivia Caramello for describing to us, is to observe that the category of representables in a presheaf topos can be identified, up to idempotent-completion, with the full subcategory of finitely presentable objects in its category of points. We also have that a presheaf topos is atomic if and only if the representing category is a groupoid. The Schanuel topos classifies infinite decidable objects, so its category of points corresponds to the category of infinite sets (and monomorphisms); since any inhabited full subcategory of this category is not a groupoid (every infinite set has an injective endomorphism which is not invertible) and the topos is non-degenerate, it again follows that this cannot be a presheaf topos.
\end{xmpl}

\begin{xmpl}
\label{xmpl:abelian}
Since any abelian category is effective regular, any small abelian category with funneling colimits is reductive. This is the case for the finitely presented (right) modules of a (right) Noetherian ring, say, since these coincide with finitely generated modules and so is closed under quotients in the large category of modules. For example, the category of finitely generated abelian groups is a reductive category with finite limits and colimits.

In order to construct a small abelian category which does \textit{not} have funneling colimits, we look for a coherent ring $R$ whose collection of finitely generated ideals is not closed under infinite intersections (note that if the ring is an integral domain, coherence ensures that it will be closed under finite intersections). Let $I$ be an infinitely generated ideal which is obtained as such an intersection. In the large category of modules, we may identify $R/I$ as the colimit of the funneling diagram consisting of the inclusions of finitely generated sub-ideals of $I$, along with the parallel zero maps, into the ring, viewed as a (right) module over itself. If the colimit of this diagram existed in the category of finitely presented $R$-modules, it would have to be a quotient of $R$ by some finitely generated ideal, but by construction there is no initial finitely presented quotient under this diagram, whence the colimit does not exist.

Consider the ring $R$ of eventually constant sequences valued in the field on two elements (with point-wise operations). Observe that all finitely generated ideals in this ring are principal. An ideal generated by a sequence $g$ is the cokernel of the module homomorphism $R \to R$ sending $x$ to $x \cdot (1-g)$, so this is indeed a coherent ring. For each index $i$ we have a `basis element' $e_i$ which is $1$ at $i$ and $0$ elsewhere. The ideal $I_i$ generated by $(1-e_i)$ consists of those sequences which are $0$ at $i$. Consider $\bigcap_{j = 1}^{\infty} I_{2j}$: this consist of sequences which are non-zero only at odd indices, but by the eventually constant criterion, no single sequence can generate this ideal, whence the ideal fails to be finitely generated.

We would like to thank Jens Hemelaer for helping us to identify the sufficient structure needed to find this counterexample and Ryan C. Schwiebert for identifying a ring realising that structure (via math.stackexchange.com).

Note that a non-trivial abelian category cannot be coherent or coalescent since the initial object is not strict in such a category (cf Lemma \ref{lem:collred}).
\end{xmpl}

\begin{xmpl}
\label{xmpl:TF}
As an example of a regular and reductive category that fails to be effective, let alone effectual, we adapt an example of Johnstone, \cite[before Example A1.3.7]{Ele}.

Consider the category $\mathbf{TF}_{fg}$ of finitely-generated torsion-free abelian groups. By considering it as a reflective subcategory of the category of finitely generated abelian groups, we find that it is regular and has all coequalizers. We can moreover check that it has funneling colimits, since the full category of abelian groups is cocomplete and any quotient of a finitely generated group is finitely generated (to obtain the quotient in $\mathbf{TF}_{fg}$, the torsion parts of such a quotient are annihilated). Thus it is a reductive category with finite limits. However, as Johnstone points out, the equivalence relation
\[R = \{(a,b) \in \Zbb \times \Zbb \mid a \equiv b \, \mathrm{mod} \, 2 \} \cong \Zbb \times \Zbb \]
is not a kernel pair of any morphism in $\mathbf{TF}_{fg}$, so this category is not effective regular.
\end{xmpl}

\begin{xmpl}
\label{xmpl:countable}
For an example of a regular and reductive category that is effective but not effectual, we modify another of Johnstone's examples, \cite[Example D3.3.9]{Ele}. Let $\Set_\omega$ be the full subcategory of $\Set$ on the finite and countable sets. This has pullbacks and funneling colimits which are stable under pullback, inherited from $\Set$. Moreover, all epimorphisms are regular, so this is a regular and reductive category (indeed, it is a coherent and coalescent category too!). It also inherits the property of being effective from $\Set$.

However, it is not effectual. Johnstone exhibits the following coequalizer diagram:
\begin{equation}
\label{eq:Ncoeq}
\begin{tikzcd}
\Nbb \ar[r, shift left, "\id"] \ar[r, shift right, "s"'] & \Nbb \ar[r, two heads] & 1,
\end{tikzcd}
\end{equation}
where $s$ is the successor function. He concludes by considering the natural number object in $\Sh(\Set_\omega, J_c)$ that this coequalizer is not preserved by the canonical functor $\ell: \Set_\omega \to \Sh(\Set_\omega,J_c)$; we could deduce the failure of effectuality from that. Instead, we prove it directly by considering the morphisms
\[\begin{tikzcd}
\Nbb \ar[rr, shift left, "\id"] \ar[rr, shift right, "g := 2 \times -"'] & & \Nbb,
\end{tikzcd}\]
which are clearly coequalized by the epimorphism in \eqref{eq:Ncoeq}. If $\Set_\omega$ were effectual, there would exist some epimorphism $t: X \too \Nbb$ such that $t$ and $g \circ t$ lie in the same connected component of $(X \downarrow F)$, where $F$ is the parallel arrow diagram $(\id,s)$ whose coequalizer is shown in \eqref{eq:Ncoeq}. However, given any finite zigzag,
\[\begin{tikzcd}
X \ar[dd, "t"'] \ar[r, equal] & X \ar[d] \ar[r, equal]& X \ar[dd] \ar[rr, equal] & & X \ar[dd] \ar[r, equal] & X \ar[d] \ar[r, equal] & X \ar[dd, "g \circ t"] \\
& \Nbb \ar[dl, "x_1"'] \ar[dr, "y_1"] & & \cdots \ar[dl, "x_2"'] \ar[dr, "y_{n-1}"] & & \Nbb \ar[dl, "x_n"'] \ar[dr, "y_n"] & \\
\Nbb && \Nbb && \Nbb && \Nbb,
\end{tikzcd}\]
where $x_i$ and $y_i$ are $\id$ or $s$, composing with any morphism $m: 1 \to X$ such that $t \circ m > n$ as elements of $\Nbb$, we conclude that since the difference between the image of $m$ in the first copy of $\Nbb$ and the last copy of $\Nbb$ is at most $n$, we have a contradiction. That is, no zigzag of finite length is sufficient to connect $t$ and $g \circ t$ in $(X \downarrow F)$.
\end{xmpl}

%% file: TSGT_Examples.tex
\section{Case Study: Localic Supercompactly and Compactly Generated Toposes}
\label{sec:xmpls}

A natural question is how the classes of supercompactly or compactly generated toposes interact with the most well-studied class of toposes, the localic toposes.
Recall that a Grothendieck topos is \textbf{localic} if it is of the form $\Sh(L)$, the category of (set-valued) sheaves on some locale $L$. These toposes can equivalently be characterised by the fact that their set of subterminal objects is separating, and these subterminals form a frame isomorphic to the frame of opens of the locale.

\subsection{Posets and Supercompact Generation}
\label{ssec:localic}

As usual, the supercompactly generated case of localic toposes are the more straightforward case to present. A $1$-categorical account of the dualities in this section can be found in the work of Caramello on Stone-type dualities \cite[\S 4.1]{Stonetype}. We include a full exposition both to compare with the compactly generated case in the next section and to illustrate the extra insight gained from the results obtained in the present article.

\begin{prop}
\label{prop:localic}
A localic topos is supercompactly generated if and only if it is equivalent to $[\Ccal\op,\Set]$ for some poset $\Ccal$. Moreover, any poset is an instance of a reductive category.
\end{prop}
\begin{proof}
If $\Ecal$ is supercompactly generated, each subterminal must be covered by its supercompact subobjects, so the supercompact subterminal objects generate the topos.

A subterminal object is supercompact if and only if it has no proper cover by strictly smaller subterminals. This forces the canonical topology on the supercompact objects to be trivial, whence $\Ecal \simeq [\Ccal_s^{\mathrm{op}},\Set]$. By considering the expression of $\Ecal$ as a category of sheaves on a locale, we see that the supercompact objects must all be quotients of subterminals, and hence themselves subterminals, so $\Ecal$ really is the category of presheaves on $\Ccal_s$.
\end{proof}

\begin{dfn}
We say a locale is an \textbf{Alexandroff locale} if it has a sub-base of supercompact opens. These are precisely the locales presenting the toposes appearing in Proposition \ref{prop:localic}.
\end{dfn}

There are several ways to extend this to a full duality, of which we present two. The first is to consider all order-preserving functions between the posets, or equivalently all functors between them when they are viewed as categories. At the level of toposes, these correspond to essential geometric morphisms, and descending back to the level of Alexandroff locales they correspond to locale maps whose inverse image mappings preserve arbitrary intersections of opens, which we call \textbf{completely continuous} maps.

We may also regard the category of posets as a $2$-category (or more precisely a $(1,2)$-category, since there is at most one natural transformation between any pair of order-preserving functions). Recall that the Yoneda embedding induces a $2$-cell-reversing $2$-equivalence between the $2$-category of idempotent-complete categories and the $2$-category of presheaf toposes. Since posets are always idempotent complete, this restricts to a $2$-equivalence between posets and localic presheaf toposes. At the level of locale maps, we can consider natural transformations between corresponding frame homomorphisms.

This recovers the constructive version of Alexandroff duality, stated by Caramello in \cite[Theorems 4.1 and 4.2]{Stonetype}, and extended here to a $2$-categorical (or $(1,2)$-categorical) result:
\begin{prop}
\label{prop:dual1}
Let $\Pos$ be the $2$-category of posets and order-preserving functions with their pointwise ordering. Let $\mathbf{LocPSh}_{\mathrm{ess}}$ be the $2$-category of localic presheaf toposes and essential geometric morphisms with geometric transformations. Let $\mathbf{AlexLoc}_{cc}$ be the $2$-category of Alexandroff locales and completely continuous maps between these with the pointwise ordering on frame homomorphisms. Then we have equivalences of $2$-categories:
\[ \Pos\co \simeq \mathbf{LocPSh}_{\mathrm{ess}} \simeq \mathbf{AlexLoc}_{cc}. \]
\end{prop}

As Caramello remarks in \cite{Stonetype}, there are two canonical ways to recover a topological duality from this localic result (we omit the extension of these dualities to preorders here). Since the toposes of sheaves on Alexandroff locales are presheaf toposes, Alexandroff locales always have enough points; each point corresponds to an inhabited, upward-closed subset of the poset. The corresponding presheaf toposes moreover have enough \textit{essential} points, and (again since posets are trivially idempotent-complete) these correspond to elements of the posets.

Classically, we have a correspondence between preorders and \textbf{Alexandroff spaces}, which are characterised by the fact that their open sets are closed under arbitrary intersections. The correspondence sends a preorder to its underlying set, topologized by making subsets which are upward-closed with respect to the ordering open, and conversely it puts the specialization order on the points of an Alexandroff space; we recognize this as the space of essential points of the associated presheaf topos on the \textit{opposite} of the preorder. On morphisms, there is a correspondence between order-preserving functions and \textit{all} continuous maps. Indeed, we have that:
\begin{lemma}
Any continuous map between Alexandroff spaces is completely continuous (the inverse image mapping preserves arbitrary meets of open sets).
\end{lemma}
\begin{proof}
For a map of topological spaces, the inverse image map is defined on all subsets of points of the codomain, and preserves arbitrary intersections of these. Combined with the fact that an arbitrary intersection of opens is open, it follows that the inverse image map preserves arbitrary intersections of opens.
\end{proof}

The disparity between the correspondence in Proposition \ref{prop:dual1} and the classical duality is non-trivial: there are continuous maps between Alexandroff locales which are \textit{not} completely continuous.
\begin{xmpl}
Consider the ordinals $\omega$ and $\omega + 1$ as posets\footnote{As is conventional, we identify the elements of an ordinal $\alpha$ with the ordinals $\beta < \alpha$.}. The corresponding Alexandroff spaces $X_{\omega}$ and $X_{\omega +1}$ have frames of opens $(\omega + 1)\op$ and $(\omega + 2)\op$, respectively, since any upward-closed subset is either empty or has a least element. These are, of course, also the frames of opens of Alexandroff locales, being the frames of subterminals in $[\omega,\Set]$ and $[(\omega + 1),\Set]$, respectively.

The collection of all points in these locales correspond to the inhabited, upward-closed subsets of $\omega\op$ and $(\omega + 1)\op$, of which there are $\omega + 1$ and $\omega + 2$ respectively. Of these, all but one of the points are essential: we have an essential point for each element of the original poset, plus a limiting non-essential point corresponding to the upward-closed subset with no least element.

The topological spaces $X^+_{\omega}$ and $X^+_{\omega +1}$ obtained by topologizing the full collections of points are the respective sobrifications of $X_{\omega}$ and $X_{\omega +1}$. In each case, the added point (indexed by the limit ordinal $\omega$) is contained in every open set indexed by $n < \omega$, but the singleton consisting of that point is not open. That is, the open sets in $X^+_{\omega}$ and $X^+_{\omega +1}$ are not closed under arbitrary intersection, so these are not Alexandroff spaces!

Consider the frame homomorphism $(\omega + 1)\op \to (\omega + 2)\op$ defined by $\omega \mapsto \omega + 1$ and $n \mapsto n$ for $n < \omega$. This necessarily corresponds (contravariantly) to a continuous map $X^+_{\omega + 1} \to X^+_{\omega}$; indeed, it is the map sending the points $\omega$ and $\omega + 1$ to $\omega$ and sending $n \mapsto n$ for $n < \omega$. However, it fails to correspond to any continous map $X_{\omega + 1} \to X_{\omega}$, since it does not preserve the intersection of the family of opens indexed by $n < \omega$.
\end{xmpl}

In summary, we have:
\begin{crly}
Let $\mathbf{AlexSp}$ be the $2$-category of Alexandroff spaces and continuous maps with the pointwise specialization ordering. The equivalences of Proposition \ref{prop:dual1} extend to:
\[ \Pos\co \simeq \mathbf{LocPSh}_{\mathrm{ess}} \simeq \mathbf{AlexLoc}_{cc} \simeq \mathbf{AlexSp}\co. \]
\end{crly}

The second way to extend the correspondence between posets and locales to a duality is to allow all locale maps whose inverse images preserve supercompact opens, which by Theorem \ref{thm:2equiv} correspond to morphisms of sites between posets equipped with their trivial topologies; equivalently, they corresponds to the flat functors between these posets.

\begin{dfn}
Let $\Ccal$, $\Dcal$ be posets and let $f:\Ccal \to \Dcal$ be an order-preserving function. We say $f$ is \textbf{flat} if:
\begin{itemize}
	\item For any $d \in \Dcal$ there exists $c \in \Ccal$ such that $d \leq f(c)$;
	\item For any element $d \in \Dcal$ and any elements $c, c' \in \Ccal$ such that $d \leq f(c)$ and $d \leq f(c')$ there exists $c'' \in \Ccal$ such that $c'' \leq c$, $c'' \leq c'$, and $d \leq f(c'')$.
\end{itemize}
\end{dfn}

\begin{prop}
\label{prop:dual2}
Let $\Pos_f$ be the $2$-category of posets and flat functions with their pointwise ordering. Let $\mathbf{LocPSh}_{\mathrm{rpre}}$ be the $2$-category of localic presheaf toposes and relatively precise geometric morphisms with geometric transformations. Let $\mathbf{AlexLoc}_{sc}$ be the $2$-category of Alexandroff locales and locale maps between these whose inverse images preserve supercompact opens, with the pointwise ordering on frame homomorphisms. Then we have equivalences of $2$-categories:
\[ \Pos_{f}\op \simeq \mathbf{LocPSh}_{\mathrm{rpre}} \simeq \mathbf{AlexLoc}_{sc}. \]
\end{prop}

\subsection{Join Semilattices and Compact Generation}
\label{ssec:jsl}

\begin{dfn}
\label{dfn:JSL}
In this paper, a \textbf{join semi-lattice} is a poset having all finite joins, including the bottom element. We say a join semi-lattice is \textbf{distributive} if for any triple of objects $(a,b,c)$ with $a \leq b \vee c$, there are elements $b' \leq b$ and $c' \leq c$ such that $a = b' \vee c'$; see \cite[\S II.5]{GLT}. This in particular holds in any distributive lattice.
\end{dfn}

Distributivity, which inductively extends to arbitrary finite joins, is precisely the condition ensuring that the collection of finite join covers is a stable class of finite families. Note that finite join covers are coproduct injections, so that \textit{any distributive join semi-lattice is a coalescent category}. Moreover, a distributive join semi-lattice has pullbacks if and only if it is a distributive lattice, so any distributive lattice is an example of a coalescent and coherent category which fails to be positive.

Even if we had not insisted on the presence of a bottom element in a join semi-lattice, by taking $c = a$ in the definition of distributivity, we would have that any pair of elements in a distributive join semi-lattice has a lower bound, although there need not be a greatest such.

\begin{lemma}
Suppose that $a,b,c$ are elements of a distributive join semilattice. If the meets $a \wedge b$ and $a \wedge c$ exist then so does $a \wedge (b \vee c)$, and the distributive law holds:
\[a \wedge (b \vee c) = (a \wedge b) \vee (a \wedge c).\]
Similarly, if $b \wedge c$ exists then so does $(a \vee b) \wedge (a \vee c)$ via the dual distributivity law:
\[a \vee (b \wedge c) = (a \vee b) \wedge (a \vee c).\]
\end{lemma}
\begin{proof}
Suppose $a \wedge b$ and $a \wedge c$ exist. Certainly $(a \wedge b) \vee (a \wedge c)$ is a lower bound for both $a$ and $b \vee c$. Given any element $x$ with $x \leq a$ and $x \leq b \vee c$, by assumption there exist $b' \leq b$, $c' \leq c$ with $b' \vee c' = x \leq a$. Thus $b' \leq a \wedge b$ and $c' \leq a \wedge c$, whence $x = b' \vee c' \leq (a \wedge b) \vee (a \wedge c)$, as required.

Now suppose $b \wedge c$ exists. This time, $a \vee (b \wedge c)$ is a lower bound for both $a \vee b$ and $a \vee c$. Given $y$ such that $y \leq a \vee b$ and $y \leq a \vee c$, firstly we have $a' \leq a$ and $b' \leq b$ with $a' \vee b' = y$. Then $b' \leq y \leq a \vee c$ so we have $a'' \leq a$ and $c' \leq c$ with $a'' \vee c' = b'$. In summary,
\[y = a' \vee a'' \vee c' \leq a \vee c' \leq a \vee (b \wedge c),\]
since $c' \leq c$ and $c' \leq b' \leq b$.
\end{proof}

\begin{rmk}
This result looks like it should be deducible from Scholium \ref{schl:finpb}, since we can view meets as pullbacks. That result does allow us to prove that the distributivity laws hold when all of the terms are well-defined, but it doesn't seem to provide enough constraints to construct one meet from others as we have done in the proof above.
\end{rmk}

\begin{prop}
\label{prop:localic2}
A localic topos is compactly generated if and only if it is the category of sheaves on a distributive join semilattice with respect to the topology that makes finite joins covering. 
\end{prop}
\begin{proof}
Unlike in the supercompact case, the compact subterminal objects no longer populate all of $\Ccal_c$, but a similar argument applies: if $\Ccal$ is the subcategory of $\Ecal$ on the compact \textit{subterminal} objects, then $\Ccal$ is a distributive join semilattice since a finite union of compact subterminal objects is compact, and so $\Ecal = \Sh(\Ccal,J)$ where $J$ is the topology on $\Ccal$ whose covering families are finite joins.
\end{proof}

By considering the locale of subterminal objects in the topos of sheaves on a distributive join semi-lattice, we recover a constructive version of the topological duality results for distributive join semi-lattices presented by Gr\"{a}tzer in \cite[\S II.5]{GLT}\footnote{Note that Gr\"{a}tzer uses the term \textit{Stone space} to refer to the result of the general dualising construction presented there, rather than the more specific meaning of this term.}. Without meets, one cannot define prime ideals directly as one would in a distributive lattice. However, one can define prime filters (or prime \textit{dual ideals} as they are called in \cite{GLT}) and define prime ideals as complements of these. Working constructively, it makes more sense to work with the prime filters, especially given that the topology that Gr\"{a}tzer defines on the prime ideals has a sub-base consisting of prime ideals which do \textit{not} contain a particular element of the distributive semilattice. Indeed, the points of the locale corresponding to (the localic topos of sheaves on) a distributive join semilattice, when viewed as completely prime filters in the frame of opens, correspond bijectively by restriction to prime filters in the distributive join semi-lattice. When there are enough prime filters, which is always the case if we allow ourselves to assume the axiom of choice, the subsets of the corresponding sober space of filters containing a given element of the semilattice do form a sub-base for the topology, too.

In \cite[Theorem II.5.8]{GLT}, Gr\"{a}tzer characterises the spaces arising from join semilattices via this duality as the $T_0$ spaces $X$ having a base of compact open subsets and satisfying the additional condition,
\begin{quote}
(S2) If $F$ is a closed set in $X$, $\{U_k \mid k \in K\}$ is a directed family of compact opens in $\Ocal(X)\op$, and $U_k \cap F \neq \emptyset$ for all $k$, then $\left(\bigcap_{k \in K} U_k \right) \cap F \neq \emptyset$.
\end{quote}
Gr\"{a}tzer himself describes this condition as ``complicated''. It follows from our reasoning above that we may discard (S2) by replacing the $T_0$ condition with the condition that $X$ should be sober:
\begin{crly}
A sober topological space is the space of prime ideals of a distributive join semilattice if and only if it has a sub-base of compact open sets. Moreover, the space is compact if and only if the semilattice has a top element, and is coherent if and only if the distributive semilattice is a distributive lattice.
\end{crly}
Celani and Calomino also arrive at this simplification in \cite[Theorem 20]{DSLat}\footnote{Note that Celani et al work dually with distributive meet semilattices, but this does not affect the results to any significant degree}, although their direct proof is that sobriety of a $T_0$ space is equivalent to condition (S2) takes over a page. Celani also extended the correspondence between semilattices and spaces of Gr\"{a}tzer to a duality between meet semi-lattices and a class of ordered topological spaces, which on the space side required introducing \textit{meet relations} between spaces, in \cite[Definition 22]{Celani}, which do not in general reduce to continuous maps. We shall not explicitly reproduce that duality result here, although we observe that it might be possible to recover it by considering comorphisms of join semilattice sites.

%Let $X_1,X_2$ be sober spaces with a base of compact opens. A subset $R \subseteq X_1 \times X_2$ is a binary \textbf{meet-relation} if for every $K_2 \subseteq X_2$ whose complement is compact open, the subset $K_1 \subseteq X_1$ of points related to some element of $K_2$ also has compact open complement...

%Instead, we observe that the natural analogue of Corollary \ref{crly:dual1} can be obtained by considering \textbf{comorphisms of sites} between distributive join semi-lattices, which translates to the following:
%\begin{dfn}
%Let $\Ccal$, $\Dcal$ be distributive join semi-lattices and let $f:\Ccal \to \Dcal$ be an order-preserving function. We say $f$ \textbf{reflects join decompositions} if whenever $f(c) = d \vee e$, there exist $a,b \in \Ccal$ with $c = a \vee b$, $f(a) \leq d$ and $f(b) \leq e$.
%\end{dfn}

Our topos-theoretic approach gives us an alternative duality result analogous to Proposition \ref{prop:dual2}. Translating the definition of morphism of sites into this setting, we arrive at the following definition.
\begin{dfn}
Let $\Ccal$, $\Dcal$ be distributive join semilattices. An order-preserving map $f: \Ccal \to \Dcal$ is a \textbf{distributive join homomorphism} if it:
\begin{itemize}
	\item Preserves finite joins (including the bottom element).
	\item For any $d \in \Dcal$ there exists $c \in \Ccal$ such that $d \leq f(c)$;
	\item For any element $d \in \Dcal$ and any elements $c, c' \in \Ccal$ such that $d \leq f(c)$ and $d \leq f(c')$ there exists $c_1, \dotsc, c_n \in \Ccal$ and $d_1, \dotsc, d_n \in \Dcal$ such that $c_i \leq c$, $c_i \leq c'$, $d_i \leq f(c_i)$ and $d = d_1 \vee \cdots \vee d_n$.
\end{itemize}
\end{dfn}

\begin{crly}[Stone duality for distributive join-semilattices]
\label{crly:Stone}
Let $\mathbf{DJSL}_f$ be the $2$-category of distributive join semilattices, distributive join homomorphisms, and pointwise comparison morphisms. Let $\mathbf{LocCG}_{rpro}$ be the $2$-category of compactly generated localic toposes, relatively proper geometric morphisms, and geometric transformations. Let $\mathbf{preCoh}_{c}$ be the $2$-category of locales having a sub-base of compact opens, continuous maps whose inverse image mappings preserve compact open sets, and pointwise comparisons of these. Then we have equivalences:
\[\mathbf{DJSL}_f \simeq \mathbf{LocCG}_{rpro} \simeq \mathbf{preCoh}_{c}\]
Given the axiom of choice, by Theorem \ref{thm:Del}, we may identify the objects of the latter category with the sober topological spaces having a sub-base of compact open subsets.
\end{crly}

Observe that this duality is not merely an extension of the Stone duality between distributive lattices and coherent locales recalled by Caramello in \cite[\S 4.2]{Stonetype} (with reference to \cite[II.3.2, II.3.3]{Stone}), but the maximal extension within Caramello's framework, in the sense that the toposes involved are the full class of compactly generated localic toposes, and hence correspond to the full class of locales having a sub-base of compact opens.

%If we wanted to extend these dualities to Priestley-type dualities, in the spatial case we could refine the topologies involved to the corresponding \textit{patch topology}, as outlined in Priestley's survey \cite{Spectral}. Alternatively, working constructively, we could enlarge the compactly generated locales by generating a suitable free Boolean algebra using the general logic methodology outlined by Caramello in \cite{PTDfPOS}.

\subsection{Localic Examples}

By applying Theorem \ref{thm:closure}, we can immediately conclude that the above characterisations apply to the localic reflections of supercompactly and compactly generated toposes.
\begin{crly}
The localic reflection of any supercompactly generated topos is a (localic) presheaf topos. The localic reflection of a compactly generated topos is a topos of sheaves on a distributive join semilattice.
\end{crly}

\begin{xmpl}
In the proof of Corollary \ref{crly:hypepres}, we observe that hyperconnected morphisms are precise, so that any supercompactly generated two-valued topos is supercompact. Taking $\Ccal$ to be any non-trivial poset with a maximal element and taking $\Ecal$ to be the category of presheaves on this poset provides an example of a supercompactly generated, supercompact topos which is not two-valued.
\end{xmpl}

\begin{xmpl}
\label{xmpl:nonsc}
The objects of a reductive category need not be supercompact within this category, in spite of Corollary \ref{crly:strict}. For example, consider the four-element lattice:
\[\begin{tikzcd}
	& 1 & \\
	a \ar[ur] & & b \ar[ul] \\
	& 0 \ar[ur] \ar[ul] &
\end{tikzcd}\]
By Proposition \ref{prop:localic}, it forms a reductive category. However, the colimit of the span $a \leftarrow 0 \rightarrow b$ is $1$, which is to say that the arrows from $a$ and $b$ to $1$ form a strictly epic family containing no strict epimorphism. Even if we relax to mere epimorphisms here, the empty family is strictly epic over $0$ yet has no inhabited subfamilies.

By considering the lattice of subsets of $\Nbb$ as a distributive join semilattice, we similarly find that objects of coalescent categories need not be compact in those categories.
\end{xmpl}

\begin{xmpl}
\label{xmpl:R}
A familiar example of a localic, locally connected topos which is not supercompactly (or even compactly) generated is the topos of sheaves on the real numbers: no non-trivial open sets in the reals are compact. $\Sh(\Rbb)$ is not totally connected, so by Lemma \ref{lem:point} none of its points are relatively precise.
\end{xmpl}

\begin{xmpl}
\label{xmpl:Nlocalic}
As a more original non-example, we construct a localic, totally connected topos which is not compactly generated. Consider the poset $P$ whose objects are the natural numbers (excluding $0$), and with order given by $n \leq m$ iff $n$ is divisible by $m$, so that $1$ is terminal. Endow this poset with the Grothendieck topology $J$ whose covering sieves on a natural number $n$ are those containing cofinitely many prime multiples of $n$.

All of these sieves are connected and effective epimorphic; that is, $(P,J)$ is a localic, locally connected, subcanonical site. For every $n$, $\ell(n)$ is therefore an indecomposable subterminal object of $\Sh(P,J)$. Since $P$ has finite limits, the topos $\Sh(P,J)$ is moreover totally connected. To show that $\Sh(P,J)$ fails to be compactly generated it suffices to show that none of the $\ell(n)$ are. But by construction each $\ell(n)$ has a nontrivial infinite covering family by other representables which contains no finite subcovers. Thus this topos has no supercompact objects, and the only compact object is the initial object.
\end{xmpl}

\begin{xmpl}
\label{xmpl:surjprec}
As promised earlier, we demonstrate that it is not possible to extend Theorem \ref{thm:hype2}(i) or (ii) to relatively precise or relatively proper surjections.

Consider the poset $P$ constructed as a fractal tree with countably many roots and branches. Explicitly, it has elements \textit{non-empty} finite sequences of natural numbers, $\coprod_{n=1}^\infty \Nbb^n$, with $\vec{x} \leq \vec{y}$ if $\vec{y}$ is an initial segment of the $\vec{x}$. The Alexandroff locale $L$ corresponding to $P$ has opens which are downward-closed subsets in this ordering, so that for any sequence $\vec{x}$ in an open set, all extensions of $\vec{x}$ also lie in that open.

Consider the collection of opens $U$ such that if $(x_1,\dotsc,x_{k-1},x_k) \in U$, then $(x_1,\dotsc,x_{k-1},y) \in U$ for cofinitely many values of $y$. This collection is clearly closed under finite intersections and arbitrary unions (we needed the sequences in $P$ to be non-empty to ensure that the empty intersection of opens was included here), which makes it a subframe of $\Ocal(L)$. This corresponds to some locale $L'$ such that there is a surjective locale map $L \to L'$, and hence a geometric surjection $s: \Sh(L) \to \Sh(L')$. Moreover, this surjection is relatively precise. Indeed, if $X$ is a sheaf on $L'$ and we are given a covering of $s^*(X)$ in $\Sh(L)$, we may without loss of generality assume that $s^*(X)$ is covered by supercompact opens of $L$, and each supercompact open contains the open of $L'$ consisting of the strict extensions of sequences it contains; $X$ is necessarily the union of these in $\Sh(L')$.

However, $\Sh(L')$ is not supercompactly or even compactly generated, since the opens of the form $s^*(U)$ are not compact, with the exception of the initial open, despite $s^*$ preserving any supercompact objects which exist.
\end{xmpl}